\documentclass[11pt,twoside]{amsart}

\numberwithin{equation}{section}
\usepackage[table]{xcolor}
\usepackage{latexsym,amssymb,amsmath,amsthm,amsfonts}
\usepackage[cmtip,arrow]{xy}
\usepackage{pb-diagram,pb-xy}
\usepackage{graphicx}
\usepackage{tikz}
\usepackage{multicol}
\usepackage{multirow}
\usepackage{enumerate}
\usepackage{url}
\usetikzlibrary{matrix,arrows}
\definecolor{VerdeOlivo}{rgb}{0.3,0.5,0.1}
\definecolor{Magenta}{rgb}{.65,0.15,.2}
\definecolor{Gris}{gray}{0.7}

\newtheorem{Theorem}{Theorem}[section] 
\newtheorem{Definition}[Theorem]{Definition}

\newtheorem{Remark}[Theorem]{Remark}

\newtheorem{Proposition}[Theorem]{Proposition} 
\newtheorem{Claim}[Theorem]{Claim}

\theoremstyle{definition}

\newcommand{\Gaa}{\sf P_5}
\newcommand{\Gab}{\sf G_{6,1}}
\newcommand{\Gac}{\sf G_{6,2}}
\newcommand{\Gad}{\sf G_{6,3}}
\newcommand{\Gae}{\sf G_{6,4}}
\newcommand{\Gaf}{\sf G_{6,5}}
\newcommand{\Gag}{\sf G_{6,6}}
\newcommand{\Gah}{\sf G_{6,7}}
\newcommand{\Gai}{\sf G_{6,8}}
\newcommand{\Gaj}{\sf G_{6,9}}
\newcommand{\Gak}{\sf G_{6,10}}
\newcommand{\Gal}{\sf G_{6,11}}
\newcommand{\Gam}{\sf G_{6,12}}
\newcommand{\Gan}{\sf G_{6,13}}
\newcommand{\Gao}{\sf G_{6,14}}
\newcommand{\Gap}{\sf G_{6,15}}
\newcommand{\Gaq}{\sf G_{6,16}}
\newcommand{\Gar}{\sf G_{6,17}}
\newcommand{\Gas}{\sf G_{6,18}}
\newcommand{\Gat}{\sf G_{6,19}}
\newcommand{\Gau}{\sf G_{6,20}}
\newcommand{\Gav}{\sf G_{6,21}}
\newcommand{\Gaw}{\sf G_{6,22}}
\newcommand{\Gax}{\sf G_{6,23}}
\newcommand{\Gay}{\sf G_{6,24}}
\newcommand{\Gaz}{\sf G_{6,25}}
\newcommand{\Gba}{\sf G_{6,26}}
\newcommand{\Gbb}{\sf G_{6,27}}
\newcommand{\Gbc}{\sf G_{7,1}}
\newcommand{\Gbd}{\sf G_{7,2}}
\newcommand{\Gbe}{\sf G_{7,3}}
\newcommand{\Gbf}{\sf G_{7,4}}
\newcommand{\Gbg}{\sf G_{7,5}}
\newcommand{\Gbh}{\sf G_{7,6}}
\newcommand{\Gbi}{\sf G_{7,7}}
\newcommand{\Gbj}{\sf G_{7,8}}
\newcommand{\Gbk}{\sf G_{7,9}}
\newcommand{\Gbl}{\sf G_{7,10}}
\newcommand{\Gbm}{\sf G_{7,11}}
\newcommand{\Gbn}{\sf G_{7,12}}
\newcommand{\Gbo}{\sf G_{7,13}}
\newcommand{\Gbp}{\sf G_{7,14}}
\newcommand{\Gbq}{\sf G_{7,15}}
\newcommand{\Gbr}{\sf G_{7,16} }
\newcommand{\Gbs}{\sf G_{7,17}}
\newcommand{\Gbt}{\sf G_{8,1}}
\newcommand{\Gbu}{\sf G_{8,2}}
\newcommand{\Gbv}{\sf G_{8,3}}
\newcommand{\Gbw}{\sf G_{8,4}}

\begin{document} 


\title[Small clique number graphs with three trivial critical ideals]{Small clique number graphs with three trivial critical ideals}


\author{Carlos A.  Alfaro}
\author{Carlos E. Valencia}
\address{
Departamento de
Matem\'aticas\\
Centro de Investigaci\'on y de Estudios Avanzados del
IPN\\
Apartado Postal
14--740 \\
07000 Mexico City, D.F. 
} 
\email{alfaromontufar@gmail.com,cvalencia@math.cinvestav.edu.mx}
\thanks{Both authors were partially supported by CONACyT grant 166059, the first author was partially supported by CONACyT and the second author was partially supported by SNI}


\maketitle

\begin{abstract}
The critical ideals of a graph are the determinantal ideals of the generalized Laplacian matrix associated to a graph.
In this article we provide a set of minimal forbidden graphs for the set of graphs with at most three trivial critical ideals.
Then we use these forbidden graphs to characterize the graphs with at most three trivial critical ideals and clique number equal to $2$ and $3$.
\end{abstract}

\section{Introduction}
Given a connected graph $G=(V,E)$ and a set of indeterminates $X_G=\{x_u \, : \, u\in V(G)\}$, 
the {\it generalized Laplacian matrix} $L(G,X_G)$ of $G$
is the matrix with rows and columns indexed by the vertices of $G$ given by
\[
L(G,X_G)_{u v}=
\begin{cases}
x_u & \text{ if } u=v,\\
-m_{u v} & \text{ otherwise},
\end{cases}
\]
where $m_{u v}$ is the number of the edges between vertices $u$ and $v$ of $G$.

\begin{Definition}
For all $1\leq i \leq |V|$, the $i$-{\it th critical ideal} of $G$ is the determinantal ideal given by
\[
I_i(G,X_G)=\langle  \{ {\rm det} (m) \, : \, m \text{ is an }i\times i \text{ submatrix of }L(G,X_G)\}\rangle\subseteq \mathbb{Z}[X_G].
\]
\end{Definition}

We say that a critical ideal is {\it trivial} when it is equal to $\langle 1 \rangle$.
Critical ideals were firstly defined in \cite{corrval} and have been studied in a more general framework in \cite{alfacorrval}.
The {\it algebraic co-rank} $\gamma(G)$ of $G$ is the number of trivial critical ideals of $G$.
The algebraic co-rank allows to separate the set of all simple connected graphs in the following graph families:
\[
\Gamma_{\leq i}=\{G\, :\, G \text{ is a simple connected graph with } \gamma(G)\leq i\}.
\]
In \cite{corrval} it was proven that if $H$ is an induced subgraph of $G$, then $I_i(H,X_H)\subseteq I_i(G,X_G)$  for all $i\leq |V(H)|$.
Thus $\gamma(H)\leq \gamma(G)$.
This implies that the set $\Gamma_{\leq i}$ is closed under induced subgraphs. 
Therefore we can say that a graph $G$ is \textit{forbidden} for $\Gamma_{\leq k}$ when $\gamma(G)\geq k+1$.
Moreover, we define the set of minimal forbidden as follows:
\begin{Definition}
Let ${\bf Forb}(\Gamma_{\leq k})$ be the set of minimal (under induced subgraphs property) forbidden graphs for $\Gamma_{\leq k}$.
\end{Definition}
A graph $G$ is called $\gamma$-\textit{critical} if $\gamma(G- v)< \gamma(G)$ for all $v\in V(G)$.
Then ${\bf Forb}(\Gamma_{\leq k})$ is equal to the set of $\gamma$-critical graphs with $\gamma(G)\geq k+1$ and $\gamma(G- v)\leq k$ for all $v\in V(G)$.
Given a family $\mathcal F$ of graphs, a graph $G$ is called $\mathcal F$-free if no induced subgraph of $G$ is isomorphic to a member of $\mathcal F$.
Thus $G\in\Gamma_{\leq k}$ if and only if $G$ is ${\bf Forb}(\Gamma_{\leq k})$-free. 
Therefore, characterizing ${\bf Forb}(\Gamma_{\leq k})$ leads to a characterization of $\Gamma_{\leq k}$. 
 
In \cite{alfaval}, these ideas were used to obtain a characterization of $\Gamma_{\leq 1}$ and $\Gamma_{\leq 2}$.
More precisely, it was found that ${\bf Forb}(\Gamma_{\leq 1})=\{P_3\}$ and thus $\Gamma_{\leq 1}$ consists of the complete graphs.
On the other hand, ${\bf Forb}(\Gamma_{\leq 2})$ consists of 5 graphs: $P_4$, $K_2\setminus P_2$, $K_6\setminus M_2$, {\sf cricket} and {\sf dart}.
And $\Gamma_{\leq 2}$ consists of the graphs isomorphic to an induced subgraph of one of the following graphs: {\it tripartite complete graph} $K_{n_1,n_2,n_3}$ or $T_{n_1}\vee(K_{n_2}+K_{n_3})$,
where $G+H$ denote the {\it disjoint union} of the graphs of $G$ and $H$, and $G\vee H$ denote the {\it join} of $G$ and $H$. 

The main goal of this paper is to provide a set of minimal forbidden graphs for $\Gamma_{\leq3}$ and a partial description of $\Gamma_{\leq3}$.
Specifically, we will characterize the graphs in $\Gamma_{\leq3}$ with clique number equal to $2$ and $3$.
Therefore, we prove that if a graph $G\in\Gamma_{\leq3}$ has clique number at most 3, then $G$ is an induced subgraph of one graph in the family $\mathcal F_1$ (see Figure \ref{figure:omega3}). 
The converse is stronger, each induced subgraph of a graph in $\mathcal F_1$ belongs to $\Gamma_{\leq3}$, but not all graphs in $\mathcal F_1^1$ have clique number less than 4.

\begin{figure}[h!]
\begin{center}
\begin{tabular}{c@{\extracolsep{2cm}}c@{\extracolsep{2cm}}c}
\begin{tikzpicture}[scale=1.2]
 	\tikzstyle{every node}=[minimum width=0pt, inner sep=1pt, circle]
 	\draw (90:.4) node[draw, fill=gray] (1) {};
 	\draw (210:.4) node[draw, fill=gray] (2) {};
 	\draw (330:.4) node[draw, fill=gray] (3) {};
 	\draw (90:1) node[draw, fill=gray] (4) {};
 	\draw (210:1) node[draw, fill=gray] (5) {};
 	\draw (330:1) node[draw, fill=gray] (6) {};
 	\draw (1) -- (4) -- (5) -- (2) -- (1) -- (3) -- (6) -- (4);
 	\draw (2) -- (3);
 	\draw (5) -- (6);
 \end{tikzpicture}
 &
\begin{tikzpicture}[scale=1]
 	\tikzstyle{every node}=[minimum width=0pt, inner sep=1pt, circle]
 	\draw (30:1) node[draw, fill=white] (2) {\tiny $n_1$};
 	\draw (150:1) node[draw, fill=white] (3) {\tiny $n_2$};
 	\draw (270:1) node[draw, fill=white] (7) {\tiny $n_3$};
 	\draw (330:1) node[draw, fill] (5) {\bf\tiny\color{white} $n_4$};
 	\draw (210:1) node[draw, fill] (6) {\bf\tiny\color{white} $n_6$};
 	\draw (90:1) node[draw, fill] (1) {\bf\tiny\color{white} $n_5$};
 	\draw (0:0) node[draw, fill=black] (4) {\bf\tiny\color{white} $n_7$};
 	\draw (1) -- (4) -- (5) -- (2) -- (1) -- (3) -- (6) -- (4);
 	\draw (5) -- (7) -- (2) -- (3) -- (7) -- (6);
 \end{tikzpicture}
 &
\begin{tikzpicture}[scale=1.1]
 	\tikzstyle{every node}=[minimum width=0pt, inner sep=1pt, circle]
 	\draw ({90+72*2}:1) node[draw, fill=gray] (1) {};
 	\draw ({90+72*3}:1) node[draw, fill=gray] (2) {};
 	\draw ({90+72*4}:1) node[draw, fill=gray] (3) {};
 	\draw (90:1) node[draw, fill=gray] (4) {};
 	\draw ({90+72}:1) node[draw, fill=gray] (5) {};
 	\draw (0,0) node[draw, fill=white] (6) {\tiny ${n_1}$};
 	\draw (1) -- (2) -- (3) -- (4) -- (5) -- (1) -- (6);
 	\draw (2) -- (6) -- (3);
 	\draw (4) -- (6) -- (5);
 \end{tikzpicture}
 \\
 (i) graph $G_1$ & (ii) family of graphs $\mathcal F_1^1$ & (iii) family of graphs $\mathcal F_1^2$ \\
\end{tabular}
\end{center}
\caption{ The family of graphs $\mathcal F_1$.
A black vertex represents a clique of cardinality $n_v$, a white vertex represents a stable set of cardinality $n_v$ and a gray vertex represents a single vertex.}
\label{figure:omega3}
\end{figure}
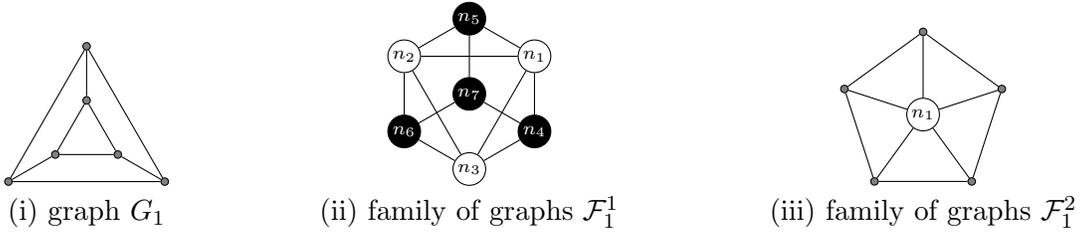

This article is divided as follows.
In Section 1.1, we will show how the characterization of $\Gamma_{\leq i}$ leads to a characterization of the graphs having critical group with $i$ invariant factors equal to 1.
In Section 2, we will construct a graph $G^{\bf d}$ by replacing its vertices by cliques and stable sets.
And we will show a novel method to verify whether the $k$-{\it th} critical ideal of $G^{\bf d}$ is trivial or not.
This method will be applied to prove that each induced subgraph of a graph in the family $\mathcal F_1$ belongs to $\Gamma_{\leq3}$.
In Section 3, we will give a family of minimal forbidden graphs for $\Gamma_{\leq 3}$, and it will be used to prove that a graph $G\in\Gamma_{\leq3}$ has clique number 2 if and only if $G$ is an induced subgraph of a graph in $\mathcal F_2$ (see Figure \ref{figure:omega2}).
Section 4 is devoted to prove that if a graph $G\in\Gamma_{\leq3}$ and has clique number 3, then $G$ is an induced subgraph of one graph in the family $\mathcal F_1$.

\subsection{Applications to the critical group}

The {\it Laplacian matrix} $L(G)$ of $G$ is the evaluation of $L(G,X_G)$ at $X_G=D_G$, where $D_G$ is the degree vector of $G$.
By considering $L(G)$ as a linear map $L(G):\mathbb{Z}^V\rightarrow \mathbb{Z}^V$, the {\it cokernel} of $L(G)$ is the quotient module $\mathbb{Z}^{V}/{\rm Im}\, L(G)$. 
The torsion part of this module is the {\it critical group} $K(G)$ of $G$.
The critical group has been studied intensively on several contexts over the last 30 years: the {\it group of components} \cite{lorenzini1991,lorenzini2008}, the {\it Picard group} \cite{bhn,biggs1999}, the {\it Jacobian group} \cite{bhn,biggs1999}, the {\it sandpile group} \cite{cori},  {\it chip-firing game} \cite{biggs1999,merino}, or {\it Laplacian unimodular equivalence} \cite{gmw,merris}.

It is known \cite[Theorem 1.4]{lorenzini1989} that the critical group of a connected graph $G$ with $n$ vertices can be described as follows:
\[
K(G)\cong \mathbb{Z}_{d_1} \oplus \mathbb{Z}_{d_2} \oplus \cdots \oplus\mathbb{Z}_{d_{n-1}},
\]
where  $d_1, d_2, ..., d_{n-1}$ are positive integers with $d_i \mid d_j$ for all $i\leq j$.
These integers are called {\it invariant factors} of the Laplacian matrix of $G$.
Besides, if $\Delta_i(G)$ is the {\it greatest common divisor} of the $i$-minors of the Laplacian matrix $L(G)$ of $G$, then the $i$-{\it th} invariant factor $d_i$ is equal to $\Delta_i(G)/ \Delta_{i-1}(G)$, where $\Delta_0(G)=1$ (for details see \cite[Theorem 3.9]{jacobson}).

\begin{Definition}
Given an integer number $k$, let $f_k(G)$ be the number of invariant factors of the critical group of $G$ equal to $k$.
Also, let $\mathcal{G}_i=\{G\, :\, G \text{ is a simple connected graph with } f_1(G)=i\}$.
\end{Definition}
The study and characterization of $\mathcal G_i$ is of great interest. 
In particular, some results and conjectures on the graphs with cyclic critical group can be found in \cite[Section 4]{lorenzini2008} and \cite[Conjectures 4.3 and 4.4]{wagner}.
On the other hand, it is easy to see \cite{cori,lorenzini1991,merris} that $\mathcal G_1$ consists only of the complete graphs.
Besides, several people (see \cite{merino}) posed interest on the characterization of $\mathcal G_2$ and $\mathcal G_3$. 
In this sense, few attempts have been done.
In \cite{pan} it was characterized the graphs in $\mathcal G_2$ whose third invariant factor is equal to $n$, $n-1$, $n-2$, or $n-3$.
Later in \cite{chan} the characterizations of the graphs in $\mathcal G_2$ with a cut vertex, 
and the graphs in $\mathcal G_2$ with number of independent cycles equal to $n-2$ are given.
Recently, a complete characterization of $\mathcal G_2$ was given in \cite{alfaval}.
However, nothing is known about $\mathcal G_3$.

A crucial result linking critical groups and critical ideals is \cite[Theorem 3.6]{corrval}, which states if $D_G$ is the degree vector of $G$, and $d_1\mid\cdots\mid d_{n-1}$ are the invariant factors of $K(G)$, then
\[
I_i(G,D_G)=\left< \prod_{j=1}^{i} d_j \right> =\left< \Delta_i(G)\right>\text{ for all }1\leq i\leq n-1.
\]
Thus if the critical ideal $I_i(G,X_G)$ is trivial, then $\Delta_i(G)$ and $d_i$ are equal to $1$.
Equivalently, if $\Delta_i(G)$ and $d_i$ are not equal to $1$, then the critical ideal $I_i(G, X_G)$ is not trivial.

The critical ideals behave better than critical group under induced subgraph property.
It is not difficult to see that unlike of $\Gamma_{\leq k}$, $\mathcal G_k$ is not closed under induced subgraphs.
For instance, the cone of the claw graph $c(K_{1,3})$ belongs to $\mathcal G_2$, but the claw graph $K_{1,3}$ belongs to $\mathcal G_3$. 
Also, $K_6\setminus \{ 2P_2\}$ belongs to $\mathcal G_3$, meanwhile $K_5\setminus \{ 2P_2\}$ belongs to $\mathcal G_2$.
Moreover, if $H$ is an induced subgraph of $G$, then it is not always true that $K(H)\trianglelefteq K(G)$.
For example, $K(K_4)\cong \mathbb{Z}_4^2\ntrianglelefteq K(K_5)\cong \mathbb{Z}_5^3$.

On the other hand, a consequence of \cite[Theorem 3.6]{corrval} is that $\mathcal{G}_i\subseteq \Gamma_{\leq i}$ for all $i\geq 0$.
Therefore, after an analysis of the $i$-{\it th} invariant factor of the Laplacian matrix of the graphs in $\Gamma_{\leq i}$, the characterization of $\mathcal G_i$ can be obtained.
See for instance \cite{alfaval} for the characterizations of $\mathcal G_1$ and $\mathcal G_2$.

\section{Cliques, stable sets and critical ideals}

Let $G=(V,E)$ a simple graph.
Suppose that $V'$ is a subset of $V$, the {\it induced subgraph} $G[V']$ of $G$ is the subgraph of $G$ whose vertex set is $V'$ and whose edge set is the set of those edges of $G$ that have both ends in $V'$.
If $E'$ is a subset of $E$, the {\it edge-induced subgraph} $G[E']$ is the subgraph of $G$ whose edge set is $E'$ and whose vertex set consists of all ends of edges of $E'$. 
Let $P,Q$ be two subsets of $V$, we denote by $E(P, Q)$ the set of edges of $G$ with one end in $P$ and the other end in $Q$.
A {\it clique} of $G$ is a subset $S$ of $V$ of mutually adjacent vertices, and the maximum size of a clique of $G$ is the {\it clique number} $\omega(G)$ of $G$.
A subset $S$ of $V$ is called an {\it independent set}, or {\it stable set}, of $G$ if no two vertices of $S$ are adjacent in $G$.
The graph with $n$ vertices whose vertex set induces a stable set is called the trivial graph, and denoted by $T_n$.
The cardinality of a maximum stable set in $G$ is called the {\it stability number} of $G$ and is denoted by $\alpha(G)$.

Given a simple graph $G=(V,E)$ and a vector ${\bf d}\in {\mathbb Z}^V$, the graph $G^{\bf d}$ is constructed as follows.
For each vertex $u\in V$ is associated a new vertex set $V_u$, where $V_u$ is a clique of cardinality $-{\bf d}_u$ if ${\bf d}_u$ is negative, and $V_u$ is a stable set of cardinality ${\bf d}_u$ if ${\bf d}_u$ is positive.
And each vertex in $V_u$ is adjacent with each vertex in $V_v$ if and only if $u$ and $v$ are adjacent in $G$.
Then the graph $G$ is called the {\it underlying graph} of $G^{\bf d}$.

A convenient way to visualize $G^{\bf d}$ is by means of a drawing of $G$, where the vertex $u$ is colored in black if ${\bf d}_u$ is negative, and colored in white if ${\bf d}_u$ is positive.
We indicate the cardinality of $V_u$ by writing it inside the drawing of vertex $u$.
When $|{\bf d}_u|=1$, we may color $u$ in gray and avoid writing the cardinality (see Figure \ref{figure:omega3}).
On the other hand, it will be useful to avoid writing the cardinality when $|{\bf d}_u|=2$ (see Figure \ref{figure:FamilyGraphF3}).

In general, the computation of the Gr\"{o}bner bases of the critical ideals is more than complicated.
However, in the rest of this section we will show a novel method, developed in \cite{alfacorrval}, to decide for $i\leq|V(G)|$ whether the $i$-{\it th} critical ideal of $G^{\bf d}$ is trivial or not.

For $V'\subseteq V(G)$ and ${\bf d}\in {\mathbb Z}^{V'}$, we define $\phi({\bf d})$ as follows:
\[
\phi({\bf d})_v =
\begin{cases}
0 & \text{ if }{\bf d}_v>0,\\
-1 & \text{ if }{\bf d}_v<0.
\end{cases}
\]

\begin{Theorem}\cite[Theorem 3.7]{alfacorrval}\label{Theo:trivialcriticalideal}
Let $n\geq 2$ and $G$ be a graph with $n$ vertices.
For $V'\subseteq V(G)$, $1\leq j\leq n$ and ${\bf d}\in \mathbb{Z}^{V'}$,  the critical ideal $I_j(G^{\bf d},X_{G^{\bf d}})$ is trivial if and only if the evaluation of $I_j(G,X_G)$ at $X_G=\phi({\bf d})$ is trivial.
\end{Theorem}

Thus the procedure of verify whether a family of graphs belongs to $\Gamma_{\leq i}$ becomes in an evaluation of the $i$-{\it th} critical ideal of the underlying graph of the family.

Let $G_2$ be the underlying graph of the family of graphs $\mathcal F_1^1$ (see Figure \ref{figure:omega3}.ii) with vertex set $V=\{v_1,v_2,v_3,v_4,v_5,v_6,v_7\}$.
Let ${\bf d}\in {\mathbb Z}^V$ such that ${\bf d}_1,{\bf d}_2,{\bf d}_3$ are positive integers and ${\bf d}_4,{\bf d}_5,{\bf d}_6,{\bf d}_7$ are negative integers.
Thus $\phi({\bf d})=(0,0,0,-1,-1,-1,-1)$.
By using a computer algebra system we can check that
\[
I_4(G_2,X_{G_2})=\langle2,x_1,x_2,x_3,x_4+1,x_5+1,x_6+1,x_7+1\rangle.
\]
Since the evaluation $I_4(G_2,X_{G_2})$ at $X_{G_2}=\phi({\bf d})$ is equal to $\langle2\rangle$, then by Theorem \ref{Theo:trivialcriticalideal} the critical ideal $I_4(G_2^{\bf d},X_{G_2^{\bf d}})$ is non-trivial. 
Therefore, each graph in this family of graphs has algebraic co-rank at most 3.
Now let $G_3$ be the underlying graph of the family of graphs $\mathcal F_1^2$ (see Figure \ref{figure:omega3}.iii) with vertex set $V=\{v_1,v_2,v_3,v_4,v_5,v_6\}$.
Let ${\bf d}_1$ be positive integer.
By using a computer algebra system we can check that
\[
I_4(G_3,X_{G_3})=\langle x_1^2+5x_1+5,x_1+x_2+3,x_1+x_3+3,x_1+x_4+3,x_1+x_5+3,x_1+x_6+3\rangle.
\]
Since the evaluation of $I_4(G_3,X_{G_3})$ at $x_1=0$ is non-trivial, then by Theorem \ref{Theo:trivialcriticalideal} the critical ideal $I_4(G_3^{\bf d},X_{G_3^{\bf d}})$ is non-trivial. 
Therefore, each graph in $\mathcal F_1^2$ has algebraic co-rank at most 3.

On the other hand, it can be verified that $\gamma(G_1)$, $\gamma(G_2)$ and $\gamma(G_3)$ are equal to 3.
Then the graphs in $\mathcal F_1$ have algebraic co-rank 3, and so each induced subgraph of a graph in $\mathcal F_1$ has algebraic co-rank at most 3.
\begin{Proposition}\label{theorem:F1belongstoGamma3}
Each graph in $\mathcal F_1$ belongs to $\Gamma_{\leq 3}$.
\end{Proposition}

\section{A description of $\Gamma_{\leq 3}$}
It is possible to compute the algebraic co-rank of all connected graphs with at most 9 vertices using the software Macaulay2 \cite{gs} and Nauty \cite{mckay}.
The computation of the algebraic co-rank of the connected graphs with at most 8 vertices required at most 3 hours on a MacBookPro with a 2.8 GHz Intel i7 quad core processor and 16 GB RAM.
Besides, the computation of the algebraic co-rank of the connected graphs with 9 vertices required 4 weeks of computation on the same computer.

Let $\mathcal F_3$ be the family of graphs shown in Figure \ref{figure:FamilyGraphF3}.
This family represents the graphs in ${\bf Forb}(\Gamma_{\leq 3})$ with at most 8 vertices.
Since there exists no minimal forbidden graph with 9 vertices for $\Gamma_{\leq 3}$, then it is likely that $\mathcal F_3={\bf Forb}(\Gamma_{\leq 3})$.
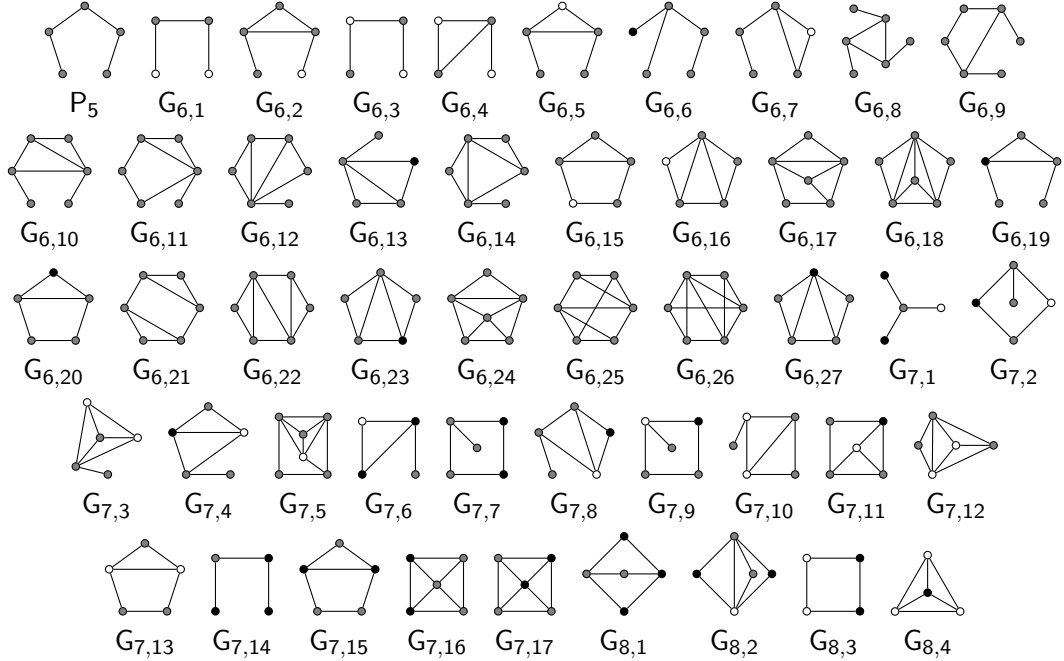
\begin{figure}[h]
\begin{center}
\begin{tabular}{cccccccccc}
\begin{tikzpicture}[scale=.5]
\tikzstyle{every node}=[minimum width=0pt, inner sep=1pt, circle]
\draw (306:1) node (v1) [draw,fill=gray] {};
\draw (18:1) node (v2) [draw,fill=gray] {};
\draw (90:1) node (v3) [draw,fill=gray] {};
\draw (162:1) node (v4) [draw,fill=gray] {};
\draw (234:1) node (v5) [draw,fill=gray] {};
\draw (v1) -- (v2) -- (v3) -- (v4) -- (v5);
\end{tikzpicture}
&
\begin{tikzpicture}[scale=.5]
\tikzstyle{every node}=[minimum width=0pt, inner sep=1pt, circle]
\draw (225:1) node (v1) [draw,fill=white] {};
\draw (315:1) node (v4) [draw,fill=white] {};
\draw (45:1) node (v6) [draw,fill=gray] {};
\draw (135:1) node (v5) [draw,fill=gray] {};
\draw (v1) -- (v5);
\draw (v4) -- (v6);
\draw (v5) -- (v6);
\end{tikzpicture}
&
\begin{tikzpicture}[scale=.5]
\tikzstyle{every node}=[minimum width=0pt, inner sep=1pt, circle]
\draw (162:1) node (v5) [draw,fill=gray] {};
\draw (234:1) node (v2) [draw,fill=gray] {};
\draw (306:1) node (v4) [draw,fill=white] {};
\draw (18:1) node (v6) [draw,fill=gray] {};
\draw (90:1) node (v1) [draw,fill=gray] {};
\draw (v1) -- (v5);
\draw (v1) -- (v6);
\draw (v2) -- (v5);
\draw (v4) -- (v6);
\draw (v5) -- (v6);
\end{tikzpicture}
&
\begin{tikzpicture}[scale=.5]
\tikzstyle{every node}=[minimum width=0pt, inner sep=1pt, circle]
\draw (225:1) node (v5) [draw,fill=gray] {};
\draw (315:1) node (v4) [draw,fill=white] {};
\draw (45:1) node (v6) [draw,fill=gray] {};
\draw (135:1) node (v1) [draw,fill=white] {};
\draw (v1) -- (v5);
\draw (v1) -- (v6);
\draw (v4) -- (v6);
\end{tikzpicture}
&
\begin{tikzpicture}[scale=.5]
\tikzstyle{every node}=[minimum width=0pt, inner sep=1pt, circle]
\draw (225:1) node (v5) [draw,fill=gray] {};
\draw (315:1) node (v4) [draw,fill=white] {};
\draw (45:1) node (v6) [draw,fill=gray] {};
\draw (135:1) node (v1) [draw,fill=white] {};
\draw (v1) -- (v5);
\draw (v1) -- (v6);
\draw (v4) -- (v6);
\draw (v5) -- (v6);
\end{tikzpicture}
&
\begin{tikzpicture}[scale=.5]
\tikzstyle{every node}=[minimum width=0pt, inner sep=1pt, circle]
\draw (162:1) node (v5) [draw,fill=gray] {};
\draw (234:1) node (v3) [draw,fill=gray] {};
\draw (306:1) node (v4) [draw,fill=gray] {};
\draw (18:1) node (v6) [draw,fill=gray] {};
\draw (90:1) node (v1) [draw,fill=white] {};
\draw (v1) -- (v5);
\draw (v1) -- (v6);
\draw (v3) -- (v5);
\draw (v4) -- (v6);
\draw (v5) -- (v6);
\end{tikzpicture}
&
\begin{tikzpicture}[scale=.5]
\tikzstyle{every node}=[minimum width=0pt, inner sep=1pt, circle]
\draw (162:1) node (v1) [draw,fill] {};
\draw (234:1) node (v3) [draw,fill=gray] {};
\draw (306:1) node (v5) [draw,fill=gray] {};
\draw (18:1) node (v2) [draw,fill=gray] {};
\draw (90:1) node (v6) [draw,fill=gray] {};
\draw (v1) -- (v6);
\draw (v2) -- (v5);
\draw (v2) -- (v6);
\draw (v3) -- (v6);
\end{tikzpicture}
&
\begin{tikzpicture}[scale=.5]
\tikzstyle{every node}=[minimum width=0pt, inner sep=1pt, circle]
\draw (162:1) node (v1) [draw,fill=gray] {};
\draw (234:1) node (v4) [draw,fill=gray] {};
\draw (306:1) node (v5) [draw,fill=gray] {};
\draw (18:1) node (v2) [draw,fill=white] {};
\draw (90:1) node (v6) [draw,fill=gray] {};
\draw (v1) -- (v4);
\draw (v1) -- (v6);
\draw (v2) -- (v5);
\draw (v2) -- (v6);
\draw (v5) -- (v6);
\end{tikzpicture}
&
\begin{tikzpicture}[scale=.5]
\tikzstyle{every node}=[minimum width=0pt, inner sep=1pt, circle]
\draw (180:.7) node (v5) [draw,fill=gray] {};
\draw (240:1) node (v2) [draw,fill=gray] {};
\draw (300:.7) node (v1) [draw,fill=gray] {};
\draw (360:1) node (v4) [draw,fill=gray] {};
\draw (420:.7) node (v6) [draw,fill=gray] {};
\draw (480:1) node (v3) [draw,fill=gray] {};
\draw (v1) -- (v4);
\draw (v1) -- (v5);
\draw (v1) -- (v6);
\draw (v2) -- (v5);
\draw (v3) -- (v6);
\draw (v5) -- (v6);
\end{tikzpicture}
&
\begin{tikzpicture}[scale=.5]
\tikzstyle{every node}=[minimum width=0pt, inner sep=1pt, circle]
\draw (180:1) node (v2) [draw,fill=gray] {};
\draw (240:1) node (v6) [draw,fill=gray] {};
\draw (300:1) node (v3) [draw,fill=gray] {};
\draw (360:1) node (v4) [draw,fill=gray] {};
\draw (420:1) node (v1) [draw,fill=gray] {};
\draw (480:1) node (v5) [draw,fill=gray] {};
\draw (v1) -- (v4);
\draw (v1) -- (v5);
\draw (v1) -- (v6);
\draw (v2) -- (v5);
\draw (v2) -- (v6);
\draw (v3) -- (v6);
\end{tikzpicture}
\\
$\Gaa$
&
$\Gab$
&
$\Gac$
&
$\Gad$
&
$\Gae$
&
$\Gaf$
&
$\Gag$
&
$\Gah$
&
$\Gai$
&
$\Gaj$
\end{tabular}
\end{center}

\begin{center}
\begin{tabular}{cccccccccc}
\begin{tikzpicture}[scale=.5]
\tikzstyle{every node}=[minimum width=0pt, inner sep=1pt, circle]
\draw (180:1) node (v1) [draw,fill=gray] {};
\draw (240:1) node (v4) [draw,fill=gray] {};
\draw (300:1) node (v3) [draw,fill=gray] {};
\draw (360:1) node (v6) [draw,fill=gray] {};
\draw (420:1) node (v2) [draw,fill=gray] {};
\draw (480:1) node (v5) [draw,fill=gray] {};
\draw (v1) -- (v4);
\draw (v1) -- (v5);
\draw (v1) -- (v6);
\draw (v2) -- (v5);
\draw (v2) -- (v6);
\draw (v3) -- (v6);
\draw (v5) -- (v6);
\end{tikzpicture}
&
\begin{tikzpicture}[scale=.5]
\tikzstyle{every node}=[minimum width=0pt, inner sep=1pt, circle]
\draw (180:1) node (v5) [draw,fill=gray] {};
\draw (240:1) node (v2) [draw,fill=gray] {};
\draw (300:1) node (v3) [draw,fill=gray] {};
\draw (360:1) node (v6) [draw,fill=gray] {};
\draw (420:1) node (v4) [draw,fill=gray] {};
\draw (480:1) node (v1) [draw,fill=gray] {};
\draw (v1) -- (v4);
\draw (v1) -- (v5);
\draw (v1) -- (v6);
\draw (v2) -- (v5);
\draw (v2) -- (v6);
\draw (v3) -- (v6);
\draw (v4) -- (v6);
\end{tikzpicture}
&
\begin{tikzpicture}[scale=.5]
\tikzstyle{every node}=[minimum width=0pt, inner sep=1pt, circle]
\draw (180:1) node (v4) [draw,fill=gray] {};
\draw (240:1) node (v6) [draw,fill=gray] {};
\draw (300:1) node (v3) [draw,fill=gray] {};
\draw (360:1) node (v2) [draw,fill=gray] {};
\draw (420:1) node (v5) [draw,fill=gray] {};
\draw (480:1) node (v1) [draw,fill=gray] {};
\draw (v1) -- (v4);
\draw (v1) -- (v5);
\draw (v1) -- (v6);
\draw (v2) -- (v5);
\draw (v2) -- (v6);
\draw (v3) -- (v6);
\draw (v4) -- (v6);
\draw (v5) -- (v6);
\end{tikzpicture}
&
\begin{tikzpicture}[scale=.5]
\tikzstyle{every node}=[minimum width=0pt, inner sep=1pt, circle]
\draw (162:1) node (v6) [draw,fill=gray] {};
\draw (234:1) node (v2) [draw,fill=gray] {};
\draw (306:1) node (v5) [draw,fill=gray] {};
\draw (18:1) node (v1) [draw,fill] {};
\draw (90:1) node (v3) [draw,fill=gray] {};
\draw (v1) -- (v5);
\draw (v1) -- (v6);
\draw (v2) -- (v5);
\draw (v2) -- (v6);
\draw (v3) -- (v6);
\draw (v5) -- (v6);
\end{tikzpicture}
&
\begin{tikzpicture}[scale=.5]
\tikzstyle{every node}=[minimum width=0pt, inner sep=1pt, circle]
\draw (180:1) node (v2) [draw,fill=gray] {};
\draw (240:1) node (v5) [draw,fill=gray] {};
\draw (300:1) node (v3) [draw,fill=gray] {};
\draw (360:1) node (v1) [draw,fill=gray] {};
\draw (420:1) node (v4) [draw,fill=gray] {};
\draw (480:1) node (v6) [draw,fill=gray] {};
\draw (v1) -- (v4);
\draw (v1) -- (v5);
\draw (v1) -- (v6);
\draw (v2) -- (v5);
\draw (v2) -- (v6);
\draw (v3) -- (v5);
\draw (v4) -- (v6);
\draw (v5) -- (v6);
\end{tikzpicture}
&
\begin{tikzpicture}[scale=.5]
\tikzstyle{every node}=[minimum width=0pt, inner sep=1pt, circle]
\draw (162:1) node (v6) [draw,fill=gray] {};
\draw (234:1) node (v2) [draw,fill=white] {};
\draw (306:1) node (v5) [draw,fill=gray] {};
\draw (18:1) node (v1) [draw,fill=gray] {};
\draw (90:1) node (v4) [draw,fill=gray] {};
\draw (v1) -- (v4);
\draw (v1) -- (v5);
\draw (v1) -- (v6);
\draw (v2) -- (v5);
\draw (v2) -- (v6);
\draw (v4) -- (v6);
\end{tikzpicture}
&
\begin{tikzpicture}[scale=.5]
\tikzstyle{every node}=[minimum width=0pt, inner sep=1pt, circle]
\draw (18:1) node (v4) [draw,fill=gray] {};
\draw (90:1) node (v6) [draw,fill=gray] {};
\draw (162:1) node (v2) [draw,fill=white] {};
\draw (234:1) node (v5) [draw,fill=gray] {};
\draw (306:1) node (v1) [draw,fill=gray] {};
\draw (v1) -- (v4);
\draw (v1) -- (v5);
\draw (v1) -- (v6);
\draw (v2) -- (v5);
\draw (v2) -- (v6);
\draw (v4) -- (v6);
\draw (v5) -- (v6);
\end{tikzpicture}
&
\begin{tikzpicture}[scale=.5]
\tikzstyle{every node}=[minimum width=0pt, inner sep=1pt, circle]
\draw (234:1) node (v3) [draw,fill=gray] {};
\draw (162:1) node (v6) [draw,fill=gray] {};
\draw (90:1) node (v2) [draw,fill=gray] {};
\draw (18:1) node (v4) [draw,fill=gray] {};
\draw (306:1) node (v5) [draw,fill=gray] {};
\draw (-90:.2) node (v1) [draw,fill=gray] {};
\draw (v1) -- (v4);
\draw (v1) -- (v5);
\draw (v1) -- (v6);
\draw (v2) -- (v4);
\draw (v2) -- (v6);
\draw (v3) -- (v5);
\draw (v3) -- (v6);
\draw (v4) -- (v5);
\draw (v4) -- (v6);
\end{tikzpicture}
&
\begin{tikzpicture}[scale=.5]
\tikzstyle{every node}=[minimum width=0pt, inner sep=1pt, circle]
\draw (18:1) node (v2) [draw,fill=gray] {};
\draw (90:1) node (v6) [draw,fill=gray] {};
\draw (162:1) node (v3) [draw,fill=gray] {};
\draw (234:1) node (v5) [draw,fill=gray] {};
\draw (0,-.2) node (v1) [draw,fill=gray] {};
\draw (306:1) node (v4) [draw,fill=gray] {};
\draw (v1) -- (v4);
\draw (v1) -- (v5);
\draw (v1) -- (v6);
\draw (v2) -- (v4);
\draw (v2) -- (v6);
\draw (v3) -- (v5);
\draw (v3) -- (v6);
\draw (v4) -- (v5);
\draw (v4) -- (v6);
\draw (v5) -- (v6);
\end{tikzpicture}
&
\begin{tikzpicture}[scale=.5]
\tikzstyle{every node}=[minimum width=0pt, inner sep=1pt, circle]
\draw (90:1) node (v1) [draw,fill=gray] {};
\draw (162:1) node (v5) [draw,fill] {};
\draw (234:1) node (v2) [draw,fill=gray] {};
\draw (306:1) node (v3) [draw,fill=gray] {};
\draw (18:1) node (v6) [draw,fill=gray] {};
\draw (v1) -- (v5);
\draw (v1) -- (v6);
\draw (v2) -- (v5);
\draw (v3) -- (v6);
\draw (v5) -- (v6);
\end{tikzpicture}
\\
$\Gak$
&
$\Gal$
&
$\Gam$
&
$\Gan$
&
$\Gao$
&
$\Gap$
&
$\Gaq$
&
$\Gar$
&
$\Gas$
&
$\Gat$
\end{tabular}
\end{center}

\begin{center}
\begin{tabular}{cccccccccc}
\begin{tikzpicture}[scale=.5]
\tikzstyle{every node}=[minimum width=0pt, inner sep=1pt, circle]
\draw (90:1) node (v1) [draw,fill] {};
\draw (162:1) node (v6) [draw,fill=gray] {};
\draw (234:1) node (v4) [draw,fill=gray] {};
\draw (306:1) node (v2) [draw,fill=gray] {};
\draw (18:1) node (v5) [draw,fill=gray] {};
\draw (v1) -- (v5);
\draw (v1) -- (v6);
\draw (v2) -- (v4);
\draw (v2) -- (v5);
\draw (v4) -- (v6);
\draw (v5) -- (v6);
\end{tikzpicture}
&
\begin{tikzpicture}[scale=.5]
\tikzstyle{every node}=[minimum width=0pt, inner sep=1pt, circle]
\draw (180:1) node (v1) [draw,fill=gray] {};
\draw (240:1) node (v3) [draw,fill=gray] {};
\draw (300:1) node (v5) [draw,fill=gray] {};
\draw (360:1) node (v2) [draw,fill=gray] {};
\draw (420:1) node (v4) [draw,fill=gray] {};
\draw (480:1) node (v6) [draw,fill=gray] {};
\draw (v1) -- (v3);
\draw (v1) -- (v5);
\draw (v1) -- (v6);
\draw (v2) -- (v4);
\draw (v2) -- (v5);
\draw (v2) -- (v6);
\draw (v3) -- (v5);
\draw (v4) -- (v6);
\end{tikzpicture}
&
\begin{tikzpicture}[scale=.5]
\tikzstyle{every node}=[minimum width=0pt, inner sep=1pt, circle]
\draw (180:1) node (v4) [draw,fill=gray] {};
\draw (240:1) node (v2) [draw,fill=gray] {};
\draw (300:1) node (v5) [draw,fill=gray] {};
\draw (360:1) node (v3) [draw,fill=gray] {};
\draw (420:1) node (v1) [draw,fill=gray] {};
\draw (480:1) node (v6) [draw,fill=gray] {};
\draw (v1) -- (v3);
\draw (v1) -- (v5);
\draw (v1) -- (v6);
\draw (v2) -- (v4);
\draw (v2) -- (v5);
\draw (v2) -- (v6);
\draw (v3) -- (v5);
\draw (v4) -- (v6);
\draw (v5) -- (v6);
\end{tikzpicture}
&
\begin{tikzpicture}[scale=.5]
\tikzstyle{every node}=[minimum width=0pt, inner sep=1pt, circle]
\draw (90:1) node (v6) [draw,fill=gray] {};
\draw (162:1) node (v2) [draw,fill=gray] {};
\draw (234:1) node (v5) [draw,fill=gray] {};
\draw (306:1) node (v1) [draw,fill] {};
\draw (18:1) node (v4) [draw,fill=gray] {};
\draw (v1) -- (v4);
\draw (v1) -- (v5);
\draw (v1) -- (v6);
\draw (v2) -- (v5);
\draw (v2) -- (v6);
\draw (v4) -- (v6);
\draw (v5) -- (v6);
\end{tikzpicture}
&
\begin{tikzpicture}[scale=.5]
\tikzstyle{every node}=[minimum width=0pt, inner sep=1pt, circle]
\draw (-90:.2) node (v6) [draw,fill=gray] {};
\draw (306:1) node (v2) [draw,fill=gray] {};
\draw (18:1) node (v5) [draw,fill=gray] {};
\draw (90:1) node (v3) [draw,fill=gray] {};
\draw (162:1) node (v1) [draw,fill=gray] {};
\draw (234:1) node (v4) [draw,fill=gray] {};
\draw (v1) -- (v3);
\draw (v1) -- (v4);
\draw (v1) -- (v5);
\draw (v1) -- (v6);
\draw (v2) -- (v4);
\draw (v2) -- (v5);
\draw (v2) -- (v6);
\draw (v3) -- (v5);
\draw (v4) -- (v6);
\draw (v5) -- (v6);	
\end{tikzpicture}
&
\begin{tikzpicture}[scale=.5]
\tikzstyle{every node}=[minimum width=0pt, inner sep=1pt, circle]
\draw (180:1) node (v1) [draw,fill=gray] {};
\draw (240:1) node (v5) [draw,fill=gray] {};
\draw (300:1) node (v3) [draw,fill=gray] {};
\draw (360:1) node (v6) [draw,fill=gray] {};
\draw (420:1) node (v2) [draw,fill=gray] {};
\draw (480:1) node (v4) [draw,fill=gray] {};
\draw (v1) -- (v3);
\draw (v1) -- (v4);
\draw (v1) -- (v5);
\draw (v1) -- (v6);
\draw (v2) -- (v4);
\draw (v2) -- (v5);
\draw (v2) -- (v6);
\draw (v3) -- (v5);
\draw (v3) -- (v6);
\draw (v4) -- (v6);
\end{tikzpicture}
&
\begin{tikzpicture}[scale=.5]
\tikzstyle{every node}=[minimum width=0pt, inner sep=1pt, circle]
\draw (180:1) node (v4) [draw,fill=gray] {};
\draw (240:1) node (v2) [draw,fill=gray] {};
\draw (300:1) node (v5) [draw,fill=gray] {};
\draw (360:1) node (v1) [draw,fill=gray] {};
\draw (420:1) node (v3) [draw,fill=gray] {};
\draw (480:1) node (v6) [draw,fill=gray] {};
\draw (v1) -- (v3);
\draw (v1) -- (v4);
\draw (v1) -- (v5);
\draw (v1) -- (v6);
\draw (v2) -- (v4);
\draw (v2) -- (v5);
\draw (v2) -- (v6);
\draw (v3) -- (v5);
\draw (v3) -- (v6);
\draw (v4) -- (v6);
\draw (v5) -- (v6);
\end{tikzpicture}
&
\begin{tikzpicture}[scale=.5]
\tikzstyle{every node}=[minimum width=0pt, inner sep=1pt, circle]
\draw (162:1) node (v2) [draw,fill=gray] {};
\draw (234:1) node (v4) [draw,fill=gray] {};
\draw (306:1) node (v1) [draw,fill=gray] {};
\draw (18:1) node (v3) [draw,fill=gray] {};
\draw (90:1) node (v5) [draw,fill] {};
\draw (v1) -- (v3);
\draw (v1) -- (v4);
\draw (v1) -- (v5);
\draw (v2) -- (v4);
\draw (v2) -- (v5);
\draw (v3) -- (v5);
\draw (v4) -- (v5);
\end{tikzpicture}
&
\begin{tikzpicture}[scale=.5]
\tikzstyle{every node}=[minimum width=0pt, inner sep=1pt, circle]
\draw (240:1) node (v1) [draw,fill] {};
\draw (0:1) node (v3) [draw,fill=white] {};
\draw (120:1) node (v2) [draw,fill] {};
\draw (0,0) node (v7) [draw,fill=gray] {};
\draw (v1) -- (v7);
\draw (v2) -- (v7);
\draw (v3) -- (v7);
\end{tikzpicture}
&
\begin{tikzpicture}[scale=.5]
\tikzstyle{every node}=[minimum width=0pt, inner sep=1pt, circle]
\draw (180:1) node (v1) [draw,fill] {};
\draw (270:1) node (v6) [draw,fill=gray] {};
\draw (0:1) node (v2) [draw,fill=white] {};
\draw (437:0) node (v4) [draw,fill=gray] {};
\draw (90:1) node (v7) [draw,fill=gray] {};
\draw (v1) -- (v6);
\draw (v1) -- (v7);
\draw (v2) -- (v6);
\draw (v2) -- (v7);
\draw (v4) -- (v7);
\end{tikzpicture}
\\
$\Gau$
&
$\Gav$
&
$\Gaw$
&
$\Gax$
&
$\Gay$
&
$\Gaz$
&
$\Gba$
&
$\Gbb$
&
$\Gbc$
&
$\Gbd$
\end{tabular}
\end{center}

\begin{center}
\begin{tabular}{cccccccccc}
\begin{tikzpicture}[scale=.5]
\tikzstyle{every node}=[minimum width=0pt, inner sep=1pt, circle]
\draw (230:1) node (v7) [draw,fill=gray] {};
\draw (282:1) node (v3) [draw,fill=gray] {};
\draw (334:0) node (v6) [draw,fill=gray] {};
\draw (0:1) node (v2) [draw,fill=white] {};
\draw (110:1) node (v4) [draw,fill=white] {};
\draw (v2) -- (v4);
\draw (v2) -- (v6);
\draw (v2) -- (v7);
\draw (v3) -- (v7);
\draw (v4) -- (v6);
\draw (v4) -- (v7);
\draw (v6) -- (v7);
\end{tikzpicture}
&
\begin{tikzpicture}[scale=.5]
\tikzstyle{every node}=[minimum width=0pt, inner sep=1pt, circle]
\draw (162:1) node (v5) [draw,fill] {};
\draw (234:1) node (v7) [draw,fill=gray] {};
\draw (306:1) node (v3) [draw,fill=gray] {};
\draw (18:1) node (v1) [draw,fill=white] {};
\draw (90:1) node (v6) [draw,fill=gray] {};
\draw (v1) -- (v5);
\draw (v1) -- (v6);
\draw (v1) -- (v7);
\draw (v3) -- (v7);
\draw (v5) -- (v6);
\draw (v5) -- (v7);
\end{tikzpicture}
&
\begin{tikzpicture}[scale=.5]
\tikzstyle{every node}=[minimum width=0pt, inner sep=1pt, circle]
\draw (130:1) node (v7) [draw,fill=gray] {};
\draw (230:1) node (v3) [draw,fill=gray] {};
\draw (310:1) node (v6) [draw,fill=gray] {};
\draw (50:1) node (v4) [draw,fill=gray] {};
\draw (90:.3) node (v5) [draw,fill=gray] {};
\draw (270:.3) node (v2) [draw,fill=white] {};
\draw (v2) -- (v4);
\draw (v2) -- (v5);
\draw (v2) -- (v6);
\draw (v2) -- (v7);
\draw (v3) -- (v6);
\draw (v3) -- (v7);
\draw (v4) -- (v5);
\draw (v4) -- (v6);
\draw (v4) -- (v7);
\draw (v5) -- (v7);
\end{tikzpicture}
&
\begin{tikzpicture}[scale=.5]
\tikzstyle{every node}=[minimum width=0pt, inner sep=1pt, circle]
\draw (315:1) node (v3) [draw,fill=gray] {};
\draw (45:1) node (v6) [draw,fill] {};
\draw (135:1) node (v1) [draw,fill=white] {};
\draw (225:1) node (v4) [draw,fill] {};
\draw (v1) -- (v4);
\draw (v1) -- (v6);
\draw (v3) -- (v6);
\draw (v4) -- (v6);
\end{tikzpicture}
&
\begin{tikzpicture}[scale=.5]
\tikzstyle{every node}=[minimum width=0pt, inner sep=1pt, circle]
\draw (135:1) node (v5) [draw,fill=gray] {};
\draw (0,0) node (v3) [draw,fill=gray] {};
\draw (225:1) node (v2) [draw,fill=gray] {};
\draw (315:1) node (v4) [draw,fill] {};
\draw (45:1) node (v1) [draw,fill] {};
\draw (v1) -- (v4);
\draw (v1) -- (v5);
\draw (v2) -- (v4);
\draw (v2) -- (v5);
\draw (v3) -- (v5);
\end{tikzpicture}
&
\begin{tikzpicture}[scale=.5]
\tikzstyle{every node}=[minimum width=0pt, inner sep=1pt, circle]
\draw (162:1) node (v5) [draw,fill=gray] {};
\draw (234:1) node (v3) [draw,fill=gray] {};
\draw (306:1) node (v2) [draw,fill=white] {};
\draw (18:1) node (v4) [draw,fill] {};
\draw (90:1) node (v7) [draw,fill=gray] {};
\draw (v2) -- (v4);
\draw (v2) -- (v5);
\draw (v2) -- (v7);
\draw (v3) -- (v5);
\draw (v4) -- (v7);
\draw (v5) -- (v7);
\end{tikzpicture}
&
\begin{tikzpicture}[scale=.5]
\tikzstyle{every node}=[minimum width=0pt, inner sep=1pt, circle]
\draw (225:1) node (v6) [draw,fill=gray] {};
\draw (315:1) node (v3) [draw,fill=gray] {};
\draw (45:1) node (v5) [draw,fill] {};
\draw (135:1) node (v2) [draw,fill=white] {};
\draw (0,0) node (v4) [draw,fill=gray] {};
\draw (v2) -- (v4);
\draw (v2) -- (v5);
\draw (v2) -- (v6);
\draw (v3) -- (v5);
\draw (v3) -- (v6);
\end{tikzpicture}
&
\begin{tikzpicture}[scale=.5]
\tikzstyle{every node}=[minimum width=0pt, inner sep=1pt, circle]
\draw (130:1) node (v1) [draw,fill=white] {};
\draw (230:1) node (v5) [draw,fill=white] {};
\draw (50:1) node (v7) [draw,fill=gray] {};
\draw (310:1) node (v3) [draw,fill=gray] {};
\draw (180:1) node (v4) [draw,fill=gray] {};
\draw (v1) -- (v4);
\draw (v1) -- (v5);
\draw (v1) -- (v7);
\draw (v3) -- (v5);
\draw (v3) -- (v7);
\draw (v5) -- (v7);
\end{tikzpicture}
&
\begin{tikzpicture}[scale=.5]
\tikzstyle{every node}=[minimum width=0pt, inner sep=1pt, circle]
\draw (135:1) node (v3) [draw,fill=gray] {};
\draw (225:1) node (v6) [draw,fill=gray] {};
\draw (0,0) node (v1) [draw,fill=white] {};
\draw (315:1) node (v4) [draw,fill=gray] {};
\draw (45:1) node (v7) [draw,fill] {};
\draw (v1) -- (v4);
\draw (v1) -- (v6);
\draw (v1) -- (v7);
\draw (v3) -- (v6);
\draw (v3) -- (v7);
\draw (v4) -- (v6);
\draw (v4) -- (v7);	
\end{tikzpicture}
&
\begin{tikzpicture}[scale=.5]
\tikzstyle{every node}=[minimum width=0pt, inner sep=1pt, circle]
\draw (180:1) node (v3) [draw,fill=gray] {};
\draw (231:1) node (v5) [draw,fill=white] {};
\draw (0:1) node (v4) [draw,fill=gray] {};
\draw (385:0) node (v2) [draw,fill=white] {};
\draw (488:1) node (v7) [draw,fill=gray] {};
\draw (v2) -- (v4);
\draw (v2) -- (v5);
\draw (v2) -- (v7);
\draw (v3) -- (v5);
\draw (v3) -- (v7);
\draw (v4) -- (v5);
\draw (v4) -- (v7);
\draw (v5) -- (v7);
\end{tikzpicture}
\\
$\Gbe$
&
$\Gbf$
&
$\Gbg$
&
$\Gbh$
&
$\Gbi$
&
$\Gbj$
&
$\Gbk$
&
$\Gbl$
&
$\Gbm$
&
$\Gbn$
\end{tabular}
\end{center}

\begin{center}
\begin{tabular}{cccccccccc}
\begin{tikzpicture}[scale=.5]
\tikzstyle{every node}=[minimum width=0pt, inner sep=1pt, circle]
\draw (162:1) node (v1) [draw,fill=white] {};
\draw (234:1) node (v6) [draw,fill=gray] {};
\draw (306:1) node (v3) [draw,fill=gray] {};
\draw (90:1) node (v7) [draw,fill=gray] {};
\draw (18:1) node (v5) [draw,fill=white] {};
\draw (v1) -- (v5);
\draw (v1) -- (v6);
\draw (v1) -- (v7);
\draw (v3) -- (v5);
\draw (v3) -- (v6);
\draw (v5) -- (v7);
\end{tikzpicture}
&
\begin{tikzpicture}[scale=.5]
\tikzstyle{every node}=[minimum width=0pt, inner sep=1pt, circle]
\draw (135:1) node (v1) [draw,fill=gray] {};
\draw (225:1) node (v5) [draw,fill] {};
\draw (315:1) node (v2) [draw,fill] {};
\draw (45:1) node (v7) [draw,fill] {};
\draw (v1) -- (v5);
\draw (v1) -- (v7);
\draw (v2) -- (v7);
\end{tikzpicture}
&
\begin{tikzpicture}[scale=.5]
\tikzstyle{every node}=[minimum width=0pt, inner sep=1pt, circle]
\draw (90:1) node (v4) [draw,fill=gray] {};
\draw (162:1) node (v1) [draw,fill] {};
\draw (234:1) node (v5) [draw,fill=gray] {};
\draw (306:1) node (v2) [draw,fill=gray] {};
\draw (18:1) node (v7) [draw,fill] {};
\draw (v1) -- (v4);
\draw (v1) -- (v5);
\draw (v1) -- (v7);
\draw (v2) -- (v5);
\draw (v2) -- (v7);
\draw (v4) -- (v7);
\end{tikzpicture}
&
\begin{tikzpicture}[scale=.5]
\tikzstyle{every node}=[minimum width=0pt, inner sep=1pt, circle]
\draw (135:1) node (v1) [draw,fill] {};
\draw (225:1) node (v3) [draw,fill] {};
\draw (315:1) node (v2) [draw,fill=gray] {};
\draw (45:1) node (v4) [draw,fill=gray] {};
\draw (0,0) node (v7) [draw,fill=gray] {};
\draw (v1) -- (v3);
\draw (v1) -- (v4);
\draw (v1) -- (v7);
\draw (v2) -- (v3);
\draw (v2) -- (v4);
\draw (v2) -- (v7);
\draw (v3) -- (v7);
\draw (v4) -- (v7);
\end{tikzpicture}
&
\begin{tikzpicture}[scale=.5]
\tikzstyle{every node}=[minimum width=0pt, inner sep=1pt, circle]
\draw (135:1) node (v1) [draw,fill=gray] {};
\draw (225:1) node (v4) [draw,fill=gray] {};
\draw (315:1) node (v2) [draw,fill=gray] {};
\draw (45:1) node (v3) [draw,fill] {};
\draw (0,0) node (v6) [draw,fill] {};
\draw (v1) -- (v3);
\draw (v1) -- (v4);
\draw (v1) -- (v6);
\draw (v2) -- (v3);
\draw (v2) -- (v4);
\draw (v2) -- (v6);
\draw (v3) -- (v6);
\draw (v4) -- (v6);
\end{tikzpicture}
&
\begin{tikzpicture}[scale=.5]
\tikzstyle{every node}=[minimum width=0pt, inner sep=1pt, circle]
\draw (0,0) node (v3) [draw,fill=gray] {};
\draw (180:1) node (v7) [draw,fill=gray] {};
\draw (-90:1) node (v4) [draw,fill] {};
\draw (0:1) node (v6) [draw,fill] {};
\draw (90:1) node (v5) [draw,fill] {};
\draw (v3) -- (v6);
\draw (v3) -- (v7);
\draw (v4) -- (v6);
\draw (v4) -- (v7);
\draw (v5) -- (v6);
\draw (v5) -- (v7);
\end{tikzpicture}
&
\begin{tikzpicture}[scale=.5]
\tikzstyle{every node}=[minimum width=0pt, inner sep=1pt, circle]
\draw (180:1) node (v1) [draw,fill] {};
\draw (-90:1) node (v6) [draw,fill=white] {};
\draw (0:1) node (v5) [draw,fill] {};
\draw (0:.5) node (v3) [draw,fill=gray] {};
\draw (90:1) node (v8) [draw,fill=gray] {};
\draw (v1) -- (v6);
\draw (v1) -- (v8);
\draw (v3) -- (v6);
\draw (v3) -- (v8);
\draw (v5) -- (v6);
\draw (v5) -- (v8);
\draw (v6) -- (v8);
\end{tikzpicture}
&
\begin{tikzpicture}[scale=.5]
\tikzstyle{every node}=[minimum width=0pt, inner sep=1pt, circle]
\draw (135:1) node (v7) [draw,fill=white] {};
\draw (225:1) node (v3) [draw,fill=white] {};
\draw (315:1) node (v5) [draw,fill] {};
\draw (45:1) node (v4) [draw,fill] {};
\draw (v3) -- (v5);
\draw (v3) -- (v7);
\draw (v4) -- (v5);
\draw (v4) -- (v7);
\end{tikzpicture}
&
\begin{tikzpicture}[scale=.5]
\tikzstyle{every node}=[minimum width=0pt, inner sep=1pt, circle]
\draw (90:1) node (v1) [draw,fill=white] {};
\draw (210:1) node (v3) [draw,fill=white] {};
\draw (330:1) node (v5) [draw,fill=white] {};
\draw (0,0) node (v7) [draw,fill] {};
\draw (v1) -- (v3);
\draw (v1) -- (v5);
\draw (v1) -- (v7);
\draw (v3) -- (v5);
\draw (v3) -- (v7);
\draw (v5) -- (v7);
\end{tikzpicture}
\\
$\Gbo$
&
$\Gbp$
&
$\Gbq$
&
$\Gbr$
&
$\Gbs$
&
$\Gbt$
&
$\Gbu$
&
$\Gbv$
&
$\Gbw$
\end{tabular}
\end{center}
\caption{The family of graphs ${\mathcal F}_3$.
A black vertex represents a clique of cardinality $2$, a white vertex represents a stable set of cardinality $2$ and a gray vertex represent a single vertex.}
\label{figure:FamilyGraphF3}
\end{figure}

\begin{Proposition}\label{lem::G3W2}
Each graph in $\mathcal F_3$ belongs to ${\bf Forb}(\Gamma_{\leq 3})$.
\end{Proposition}
\begin{proof}
It can easily checked, using a computer algebra system, that each graph in $\mathcal F_3$ is $\gamma$-critical and has algebraic co-rank equal to 4.
\end{proof}

One of the main results of this article is the following:

\begin{Theorem}\label{theorem:Gamma3belongstoF1}
If a graph $G\in \Gamma_{\leq3}$ has clique number at most 3, then $G$ is an induced subgraph of a graph in $\mathcal F_1$.
\end{Theorem}

We divide the proof in two characterizations: the graphs in $\Gamma_{\leq3}$ with clique number equal to $2$ and $3$.
The converse of Theorem \ref{theorem:Gamma3belongstoF1} is stronger, by Proposition \ref{theorem:F1belongstoGamma3} 
we have that each induced subgraph of a graph in $\mathcal F_1$ belongs to $\Gamma_{\leq3}$. 
However, it is not difficult to recognize the graphs in $\mathcal F_1^1$ with clique number greater or equal than 4.

\begin{Theorem}\label{pro::G3W2}
Let $G$ be a simple connected graph with $\omega(G)=2$.
Then, $G$ is $\mathcal F_3$-free if and only if $G$ is isomorphic to an induced subgraph of a graph in $\mathcal F_2$ (see Figure \ref{figure:omega2}).
\begin{figure}[h!]
\begin{tabular}{c@{\extracolsep{2cm}}c@{\extracolsep{2cm}}c@{\extracolsep{2cm}}c}
\begin{tikzpicture}[scale=1]
 	\tikzstyle{every node}=[minimum width=0pt, inner sep=1pt, circle]
 	\draw (-.5,0) node[draw] (1) {\tiny ${n_1}$};
 	\draw (.5,0) node[draw] (2) {\tiny ${n_2}$};
 	\draw (1) -- (2);
 \end{tikzpicture}
 &
\begin{tikzpicture}[scale=1]
 	\tikzstyle{every node}=[minimum width=0pt, inner sep=1pt, circle]
 	\draw (-.5,0) node[draw] (1) {\tiny ${n_1}$};
 	\draw (.5,0) node[draw] (2) {\tiny ${n_2}$};
 	\draw (.5,1) node[draw, fill=gray] (3) {};
 	\draw (-.5,1) node[draw, fill=gray] (4) {};
 	\draw (1) -- (2) -- (3) -- (4);
 \end{tikzpicture}
 &
\begin{tikzpicture}[scale=1]
 	\tikzstyle{every node}=[minimum width=0pt, inner sep=1pt, circle]
 	\draw (-.5,0) node[draw] (1) {\tiny ${n_1}$};
 	\draw (.5,0) node[draw] (2) {\tiny ${n_2}$};
 	\draw (.5,1) node[draw, fill=gray] (3) {};
 	\draw (-.5,1) node[draw, fill=gray] (4) {};
 	\draw (4) -- (1) -- (2) -- (3);
 \end{tikzpicture}
 &
\begin{tikzpicture}[scale=1]
 	\tikzstyle{every node}=[minimum width=0pt, inner sep=1pt, circle]
 	\draw ({90+72*2}:.7) node[draw] (1) {\tiny ${n_1}$};
 	\draw ({90+72*3}:.7) node[draw] (2) {\tiny ${n_2}$};
 	\draw ({90+72*4}:.7) node[draw, fill=gray] (3) {};
 	\draw (90:.7) node[draw, fill=gray] (4) {};
 	\draw ({90+72}:.7) node[draw, fill=gray] (5) {};
 	\draw (1) -- (2) -- (3) -- (4) -- (5) -- (1);
 \end{tikzpicture}
 \\
 (i) $K_{n_1,n_2}$ & (ii) $\mathcal F_2^1$ & (iii) $\mathcal F_2^2$ & (iv) $\mathcal F_2^3$ 
\end{tabular}
\caption{ The family of graphs $\mathcal F_2$.
A white vertex represents a stable set of cardinality $n_v$ and a gray vertex represents a single vertex.}
\label{figure:omega2}
\end{figure}
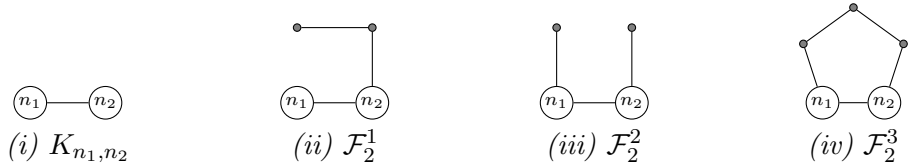
\end{Theorem}

Since each graph in $\mathcal F_2$ is isomorphic to an induced subgraph of a graph in $\mathcal F_1^1$, 
then Proposition \ref{theorem:F1belongstoGamma3} implies that each graph in $\mathcal F_2$ belongs to $\Gamma_{\leq 3}$.
Note that the graphs in $\Gamma_{\leq 1}$ and $\Gamma_{\leq 2}$ are induced subgraph of a graph in $\mathcal F_1^1$ (see Figure \ref{figure:omega3}.ii).

\begin{Theorem}\label{pro::G3W3}
Let $G$ be a simple connected graph with $\omega(G)=3$.
Then, $G$ is $\mathcal F_3$-free if and only if $G$ is isomorphic to an induced subgraph of a graph in $\mathcal F_1$ with clique number 3.
\end{Theorem}

Section \ref{proofoftheorem:pro::G3W3} is devoted to the proof of the Theorem \ref{pro::G3W3}.
Now we give the proof of Theorem \ref{pro::G3W2}.

\begin{proof}[Proof of Theorem \ref{pro::G3W2}]
Since each graph in $\mathcal F_2$ belongs to $\Gamma_{\leq 3}$, then each graph in $\mathcal F_2$ is $\mathcal F_3$-free.
Thus, we get one implication.

Suppose $G$ is $\mathcal F_3$-free.
Let $a,b\in V(G)$ such that $ab\in E(G)$.
Since $\omega(G)=2$, then there is no vertex adjacent with $a$ and $b$ at the same time.
For a vertex $v\in V(G)$, the neighbor set $N_G(v)$ of $v$ in $G$ is the set of all vertices adjacent with $v$.
Let ${\mathcal A}=N_G(a)-b$ and ${\mathcal B}=N_G(b)-a$.
Clearly, each of the sets $\mathcal A$ and $\mathcal B$ induces a trivial graph.
Let us define
\[
A=\{ u\in {\mathcal A} : \exists v \in {\mathcal B} \text{ such that } uv\in E\}, 
\hspace{.5cm}
A'=\{ u\in {\mathcal A} : \nexists v \in {\mathcal B} \text{ such that } uv\in E\}
\]
\[
B=\{ u\in {\mathcal B} : \exists v \in {\mathcal A} \text{ such that } uv\in E\}, 
\text{ and }
B'=\{ u\in {\mathcal B} : \nexists v \in {\mathcal A} \text{ such that } uv\in E\}.
\]
Thus we have two possible cases: when $A$ and $B$ are not empty and when $A$ and $B$ are empty.
\paragraph{\bf First we consider when $A$ and $B$ are not empty.}
In this case we have the following statements: 

\begin{Claim}
One of  the sets $A'$ or $B'$ is empty, and the other has cardinality at most one.
\end{Claim}
\begin{proof}
Suppose $A'$ and $B'$ are not empty.
Let $u\in A$, $v \in B$, $s\in A'$ and $t \in B'$.
Then the vertex set $\{a,b,u, v, s,t\}$ induces a graph isomorphic to $\Gaj$, which is impossible.
Now suppose $A'$ has cardinality more than 1.
Take $w_1, w_2\in A'$. 
Then $\{x, y, u, v, w_1, w_2\}$ induces a graph isomorphic to $\Gad$; a contradiction.
Thus $A'$ has cardinality at most one.
\end{proof}

\begin{Claim}
The edge set $E(A, B)$ induces either a complete bipartite graph or a complete bipartite graph minus an edge.
\end{Claim}
\begin{proof}
First note that each vertex in $A$ is incident with each vertex in $B$, except for at most one vertex.
It is because if $u\in A$ and $v_1, v_2, v_3 \in B$ such that $uv_1\in E(G)$ and $uv_2, uv_3 \notin E(G)$,
then the induced subgraph $G[\{x, y, u, v_1, v_2, v_3\}]$ is isomorphic to $\Gad$; which is impossible.
In a similar way, each vertex in $B$ is incident with each vertex in $A$, except for at most one vertex.
Thus the edge set $E(A, B)$ must be equal to the edge set $\{uv : u\in A, v\in B\}$ minus a matching.
In fact, the cardinality of this matching must be at most one.
Otherwise, if $u, v\in A$ and $s, t \in B$ such that $us, vt \notin E(G)$ and $ut, vs \in E(G)$, then the induced subgraph $G[\{t,u, a, v,s\}]$ is isomorphic to $\Gaa$; which is a contradiction. 
\end{proof}	

\begin{Claim}
It is not possible that, at the same time, the edge set $E(A,B)$ induces a complete bipartite graph minus an edge and  $A'\cup B'\neq \emptyset$.
\end{Claim}
\begin{proof}
Suppose both situations occur at the same time.
Let $u_1, u_2\in A$ and $v_1, v_2 \in B$ such that $u_1v_1, u_1v_2, u_2v_1\in E(G)$ and $u_2v_2 \notin E(G)$.
And let $w\in A'$.
Then the vertex set $\{u_1, u_2 ,v_1, v_2, y, w\}$ induces a graph isomorphic to $\Gaj$; which is a contradiction. 
\end{proof}

Thus there are three possible cases:
\begin{enumerate}[\it (a)]
\item $E(A, B)=\{uv: u\in A, v\in B\}$ and $A'\cup B'=\emptyset$,
\item $E(A, B)=\{uv: u\in A, v\in B\}$ and $A'=T_1$, $B'=\emptyset$ or
\item $E(A, B)$ induces a bipartite complete graph minus an edge and $A'\cup B'=\emptyset$.
\end{enumerate}

Let $V_\emptyset$ denote the set of vertices that are not adjacent with $a$ nor $b$.
In what follows we will describe the vertex set $V_\emptyset$.

\noindent {\it Case (a).} Let $w_1, w_2\in V_\emptyset$, $u_1, u_2 \in A$, and $v_1\in B$.
Note that it is not possible that a vertex in $V_\emptyset$ is adjacent with a vertex in $A$ and a vertex in $B$ at the same time, because then we get $\omega(G)\geq3$.
Moreover
 
\begin{Claim}\label{claim:omega2V0:1}
There exist no two vertices in $V_\emptyset$ such that one is adjacent with a vertex in $A$ and the other one is adjacent with a vertex in $B$.
\end{Claim}
\begin{proof}
Suppose $w_1u_1, w_2v_1 \in E(G)$.
There are two cases: either $w_1w_2 \in E(G)$ or $w_1w_2 \notin E(G)$.
If $w_1w_2 \in E(G)$, then $G[\{ w_1, w_2, v_1, b, a\}]\simeq \Gaa$; which is impossible.
And if $w_1w_2 \notin E(G)$, then $G[\{ w_1, w_2, v_1, u_1, b, a\}]\simeq \Gaj$; which is a contradiction.
\end{proof}

Without loss of generality, suppose $w_1\in V_\emptyset$ is adjacent with $u_1\in A$.
Thus
\begin{Claim}\label{claim:omega2V0:2}
The vertex set $V_\emptyset$ has cardinality at most 1.
\end{Claim}
\begin{proof}
There are two possible cases: either each vertex in $V_\emptyset$ is adjacent with a common vertex in $A$ or not.
Suppose there exists $w_2\in V_\emptyset$ such that $u_1$ is adjacent with both $w_1$ and $w_2$.
Then $w_1w_2 \notin E(G)$, because otherwise $w_1, w_2$ and $u_1$ induce a $K_3$.
Thus the vertex set $\{ w_1, w_2, u_1, v_1, a, b\}$ induces a graph isomorphic to $\Gad$; which is forbidden.
Then this case is not possible.
Now suppose $w_1$ and $w_2$ are not adjacent with a common vertex in $A$.
We have the following possible cases:
\begin{enumerate}
\item $w_1u_1, w_2u_2 \in E(G)$, 
\item $w_1u_1, w_2u_2, w_1w_2 \in E(G)$, or
\item $w_1 w_2 \in E(G)$.
\end{enumerate}
This yields a contradiction since in case $(1)$ the vertex set $\{ w_1, u_1, v_1, u_2, w_2\}$ induces a graph isomorphic to $\Gaa$, in case $(2)$ the vertex set $\{ w_2, w_1, u_1, v_1, b \}$ induces a graph isomorphic to $\Gaa$, and in case $(3)$ the vertex set $\{ w_2, w_1, u_1, v_1, b \}$ induces a graph isomorphic to $\Gaa$.
\end{proof}
 
If $|A|=2$, there are two possibilities: either $w_1$ is adjacent with each vertex in $A$ or $w_1$ is adjacent with only one vertex of $A$.
If $|A|\geq3$, then $w_1$ is adjacent with either all vertex in $A$ or only one vertex in $A$.
It is because if $w_1$ is adjacent with both $u_1, u_2\in A$ and $w$ is not adjacent with $u_3\in A$, then the vertex set $\{ w_1, u_1, u_2, u_3, a,b\}$ induces a graph isomorphic to $\Gad$.
Note that when $w_1$ is adjacent with each vertex in $A$, then the graph is isomorphic to an induced subgraph of a graph in $\mathcal F_2^2$.
Meanwhile, when $w_1$ is adjacent only with one vertex of $A$, then the graph is isomorphic to an induced subgraph of a graph in $\mathcal F_2^1$.
Finally, when $V_\emptyset=\emptyset$, the graph is a complete bipartite graph.
 
\noindent {\it Case (b).} 
Suppose without loss of generality that $A'\neq\emptyset$.
By Claims \ref{claim:omega2V0:1} and \ref{claim:omega2V0:2}, the vertex set $V_\emptyset$ has cardinality at most one.
Let $w\in V_\emptyset$, $u_1 \in A$, $u_2 \in A'$ and $v_1\in B$.
We have two cases: $w$ is adjacent with either a vertex in $\mathcal A$ or a vertex in $\mathcal B$.
Let us consider when $w$ is adjacent with a vertex in $\mathcal A$.
Here we have two possibilities: either $wu_2\in E(G)$ or $wu_2\notin E(G)$.
However, none of the two cases is allowed, since in the first case we get that the vertex set $\{w, u_2, a, b, v_1\}$ induces a graph isomorphic to $\Gaa$, and in the second case the vertex set $\{ w, u_2, a, b, u_1, v_1\}$ induces a graph isomorphic to $\Gaj$.
Thus the remaining case is that $w$ is adjacent with a vertex in $\mathcal B$.
In this case $w$ must be adjacent with $u_2$ and each vertex in $\mathcal B$, because otherwise the graph $\Gaa$ appears as induced subgraph.
Note that this graph is isomorphic to an induced subgraph of a graph in $\mathcal F_2^3$.
 
\noindent {\it Case(c).} 
Let $w\in V_\emptyset$, $u_1, u_2 \in A$, and $v_1, v_2\in B$ such that $u_1v_2\notin E(G)$ and $u_1v_1, u_2v_1, u_2v_2\in E(G)$.
Note that if $w$ is adjacent with $u_1$ or $v_2$, then $w$ is adjacent with both $u_1$ and $v_2$, because since $u_1v_2\notin E(G)$, we would obtain $P_5$ as induced subgraph; which is not possible.
In this case, when $w$ is adjacent with $u_1$ and $v_2$, the graph is isomorphic to an induced subgraph of $\mathcal F_2^3$.
Now we consider when $w$ is adjacent with a vertex in $(A-u_1)\cup (B-v_2)$.
Without loss of generality, we can suppose $w$ is adjacent with $u_2$.
The vertex $w$ is not adjacent with $v_1$ or $v_2$, because otherwise a clique of cardinality 3 is obtained.
On the other hand, $w$ is not adjacent with $u_1$, because otherwise $w$ would be adjacent with both vertices $u_1$ and $v_1$.
Thus $w$ is adjacent only with $u_2$, but the vertex set $\{w,u_1,u_2,v_2,a,b,\}$ induces a graph isomorphic to $\Gaj$; a contradiction.
Thus $V_\emptyset$ is empty, and the graph is isomorphic to an induced subgraph of $\mathcal F_2^2$.

\paragraph{\bf Now we consider the case when $A$ and $B$ are empty.}

One of the vertex sets $A'$ or $B'$ has cardinality at most one, because otherwise $A'\cup B'\cup\{a,b\}$ would contain $\Gab$ as induced subgraph.
Thus, let us assume that $A'=\{u\}$ and $|B'|>1$.
Let $V_\emptyset$ denote the vertex set whose vertices are not adjacent with both $a$ and $b$.

\begin{Claim}
If $A'$ and $B'$ are not empty, and $w\in V_\emptyset$ is adjacent with a vertex in $A'\cup B'$, then $w$ is adjacent with each vertex in $A'\cup B'$.
\end{Claim}
\begin{proof}
Let $v\in B'$.
Suppose one of the edges $uw$ or $vw$ does not exist.
Then $\{w, u, a, b, v\}$ induces a graph isomorphic to $\Gaa$; which is a contradiction.
\end{proof}

Note that the vertex set $V_\emptyset$ induces a stable set, because otherwise $\omega(G) > 2$.
Thus when $A'$ and $B'$ are not empty, the graph $G$ is isomorphic to an induced subgraph of a graph in $\mathcal F_2^3$.

Now consider the case when $B=\emptyset$ and $|A|>1$.
\begin{Claim}
Each vertex $w\in V_\emptyset$ is adjacent with either a unique vertex in $A$ or each vertex in $A$.
\end{Claim}
\begin{proof}
The result is easy to check when $|A|\leq2$.
So suppose $A$ has cardinality greater or equal than 3.
Let $u_1,u_2,u_3\in A$ such that $w$ is adjacent with both $u_1, u_2$, and $w$ is not adjacent with $u_3$.
Since the vertex set $\{ w, u_1, u_2, u_3, a,b\}$ induces a graph isomorphic to $\Gad$, then we get a contradiction and the result follows.
\end{proof}

\begin{Claim}
Let $w\in V_\emptyset$ such that it is adjacent with $u\in A$.
If $V_\emptyset$ has cardinality greater than 1, then each vertex in $V_\emptyset$ is adjacent either with $u$ or with each vertex in $A$.
\end{Claim}
\begin{proof}
Suppose there exists $w'\in V_\emptyset$ such that $w'$ is not adjacent with $u$.
Let $u'\in A$ such that $w'$ is adjacent with $w'$.
There are four possible cases:
\begin{itemize}
\item $wu',ww'\in E(G)$,
\item $ww'\in E(G)$ and $wu'\notin E(G)$,
\item $wu'\in E(G)$ and $ww'\notin E(G)$,
\item $wu',ww'\notin E(G)$.
\end{itemize}
In the cases when $ww'\in E(G)$, the vertex set $\{w,u,a,u',w'\}$ induces a graph isomorphic to $\Gaa$; those cases do not occur.
On the other hand if $wu'\in E(G)$ and $ww'\notin E(G)$, then $\{w,w',u,u',a,b\}$ induces a graph isomorphic to $\Gaj$, and this case can not occur.
Finally, if $wu',ww'\notin E(G)$, then the vertex set $\{w,u,a,u',w'\}$ induces a graph isomorphic to $\Gaa$.
Which is a contradiction.
Thus each pair of vertices in $V_\emptyset$ are adjacent with the same vertices in $A$.
And the result follows.
\end{proof}
Thus there are two cases: when each vertex in $V_\emptyset$ is adjacent with a unique vertex $u$ in $A$, and when each vertex in $V_\emptyset$ is adjacent with each vertex in $A$.
Note that $V_\emptyset$ induces a stable set, because otherwise $\omega(G)>2$.
In the first case we have that either $|V_\emptyset|\geq2$ or the vertex set $A-u$ is empty.
It is because otherwise $G$ would have $\Gab$ as induced subgraph.
Therefore, in each case $G$ is isomorphic to a graph in $\mathcal F_1^1$.
\end{proof}

\section{Proof of Theorem \ref{pro::G3W3}}\label{proofoftheorem:pro::G3W3}
One implication is easy because Proposition \ref{theorem:F1belongstoGamma3} implies that each graph in $\mathcal F_1$ is $\mathcal F_3$-free.
The other implication is much more complex.

Suppose $G$ is $\mathcal F_3$-free.
Let $W=\{a,b,c\}$ be a clique of cardinality 3.
For each $X\subseteq \{a,b,c\}$, let $V_X=\{u\in V(G) : N_G(u)\cap \{a, b, c\} = X\}$.
Note that $V_{\emptyset}$ denote the vertex set whose vertices are not adjacent with a vertex in $\{a,b,c\}$.

Since $\omega(G)=3$, then the vertex set $V_{a,b,c}$ is empty, and for each pair $\{x,y\}\subset\{a,b,c\}$, the vertex set $V_{x,y}$ induces a stable set.
Furthermore

\begin{Claim}\label{Claim:VxKmn2Kk1k2Tn}
For $x\in\{a,b,c\}$, the induced subgraph $G[V_x]$ is isomorphic to either $K_{m,n}$, $2K_2$, $K_2+K_1$, or $T_n$.
\end{Claim}
\begin{proof} 
First consider $K_3$, $P_4$, $K_2+T_2$ and $P_3+T_1$ as induced subgraphs of  $G[V_x]$.
Since $G[K_3\cup x]\simeq K_4$, $G[P_4 \cup \{x,y\}]\simeq \Gam$, $G[K_2 + T_2\cup \{a,b,c\}]\simeq \Gbc$, $G[P_3 + T_1\cup \{x,y\}]\simeq \Gae$ and all of them are forbidden for $G$, then the graphs $K_3$, $P_4$, $K_2+T_2$ and $P_3+T_1$ are forbidden in $G[V_x]$.
Thus $\omega(G[V_x])\leq2$.

If $\omega(G[V_x])=2$, then there exist $u,v\in V_x$ such that $uv\in E(G[V_x])$.
Clearly, $N_{G[V_x]}(u)\cap N_{G[V_x]}(v)=\emptyset$, and each vertex set $N_{G[V_x]}(u)$ and $N_{G[V_x]}(u)$ induces a trivial subgraph.
Since $G[V_x]$ is $P_4$-free, then $st\in E(V_x)$ for all $s\in N_{V_x}(u)\setminus\{v\}$ and $t\in N_{V_x}(v)\setminus\{u\}$.
Therefore, each component in $G[V_x]$ is a complete bipartite subgraph.
 		
If a component in $G[V_x]$ has cardinality at least three, then $G[V_x]$ does not have another component, 
because the existence of another component makes that $P_3+T_1$ appears as an induced subgraph in $G[V_x]$; which is impossible.
If there is a component in $G[V_x]$ of cardinality at least two, then there is at most another component in $G[V_x]$ since $K_2+T_2$ is forbidden in $G[V_x]$.
And thus the result turns out.
\end{proof}
\begin{Claim}\label{Claim:omega3:Vx2Cyvacio}
If there is $x\in \{a,b,c\}$ such that the induced subgraph $G[V_x]$ has at least 2 components or is isomorphic to a complete bipartite with at least three vertices, 
then the vertex set $V_y$ is empty for each $y\in \{a,b,c\}- x$.
\end{Claim}
\begin{proof}
Let $G[V_x]$ be as above.
Then there are two vertices $u,v\in V_x$ that are not adjacent.
Suppose $V_y$ is not empty, that is, there is $w\in V_y$.
There are three possibilities:
\begin{itemize}
\item $uw, vw \in E(G)$,
\item $uw \in E(G)$ and $vw \notin E(G)$, and
\item $uw, vw \notin E(G)$.
\end{itemize}
But in each case the vertex set $\{ a,b,c,u,v,w\}$ induces a graph isomorphic to $\Gap$, $\Gal$, and $\Gac$, respectively.
Since these graphs are forbidden, then we obtain a contradiction and then $V_y$ is empty.
\end{proof}
\begin{Claim}\label{Claim:omega3:Vx2CVxyvacio}
If there is $x\in \{a,b,c\}$ such that $G[V_x]$ has at least 2 components or is isomorphic to a complete bipartite with at least three vertices, then the vertex set $V_{x,y}$ is empty for each $y\in \{a,b,c\}- x$.
\end{Claim}
\begin{proof}
Let $G[V_x]$ be as above.
Then there are two vertices $u_1,u_2\in V_x$ that are not adjacent.
Suppose $V_{x,y}\neq \emptyset$.
Let $v\in V_{x,y}$.
There are three possibilities:
\begin{itemize}
\item $u_1 v\in E(G)$ and $u_2 v \in E(G)$,
\item $u_1 v\in E(G)$ and $u_2 v \notin E(G)$, and
\item $u_1 v\notin E(G)$ and $u_2 v \notin E(G)$.
\end{itemize}
Then in each case, $G[\{a,b,c, u_1, u_2, v \}]$ is isomorphic to $\Gaq$, $\Gam$ and $\Gae$, respectively.
Since these graphs are forbidden, we have that $V_{x,y}$ is empty.
\end{proof}

\begin{Remark}\label{remark:bigcases}
Claims \ref{Claim:omega3:Vx2Cyvacio} and \ref{Claim:omega3:Vx2CVxyvacio} imply that when  $V_x$ has connected 2 components or is a complete bipartite graph of at least 3 vertices, then the only non-empty vertex sets are $V_\emptyset$, $V_{x}$ and $V_{y,z}$, where $x,y$ and $z$ are different elements of $\{a,b,c\}$ 
Moreover, by Claim \ref{Claim:VxKmn2Kk1k2Tn}, the vertex set $V_x$ is one of the following vertex sets:
\begin{itemize}
\item $T_n$, where $n\geq 2$,
\item complete bipartite graph with cardinality at least 3,
\item $K_1+K_2$ or $2K_2$.
\end{itemize}
\end{Remark}

Next result describes the induced subgraph $G[V_a \cup V_b \cup V_c]$ when each set $V_x$ is connected of cardinality at most 2.

\begin{Claim}\label{cl:unstablecase}
Suppose for each $x\in \{a,b,c\}$ the vertex set $V_x$ is connected of cardinality at most 2.
If for all $x\in \{a,b,c\}$ the set $V_x$ is not empty, then $G[V_a \cup V_b \cup V_c]$ is isomorphic (where $x,y$ and $z$ are different elements of $\{a,b,c\}$) to one of the following sets:
\begin{itemize}
\item $V_x\vee(V_y+V_z)$ where $V_x=K_1$, $V_y=K_m$, $V_z=K_n$ and $m,n\in\{1,2\}$,
\item $V_x\vee(V_y+V_z)$ where $V_x=K_2$, $V_y=K_1$, $V_z=K_1$, or
\item $V_x\vee(V_y\vee V_z)$ where $V_x=K_1$, $V_y=K_1$, $V_z=K_1$.
\end{itemize}
If $V_z=\emptyset$, then $G[V_x \cup V_y]$ is isomorphic to one of the following sets:
\begin{itemize}
\item $V_x+V_y$, where $V_x=K_m$, $V_y=K_n$ and $m,n\in\{1,2\}$, or
\item $V_x\vee V_y$, where $V_x=K_1$, $V_y=K_m$ and $m\in\{1,2\}$.
\end{itemize}
If $V_y=V_z=\emptyset$, then $V_x$ is isomorphic to $K_1$ or $K_2$.
\end{Claim}
\begin{proof}
It is not difficult to prove that either $E(V_x,V_y)$ is empty or $E(V_x,V_y)$ induces a complete bipartite graph.
The result follows by checking the possibilities with a computer algebra system.
\end{proof}

In the rest of the proof, for each case obtained in Remark \ref{remark:bigcases} and Claim \ref{cl:unstablecase}, we will analyze the remaining edges sets and the vertex set $V_\emptyset$.
As before, we may refer to $x,y$ or $z$ as different elements of $\{a,b,c\}$.
Each case will be consider in a subsection.

\subsection{When $V_a\cup V_b\cup V_c=\emptyset$}
Now we describe the induced subgraph $G[V_{a,b}\cup V_{a,c} \cup V_{b,c}]$.

 \begin{Claim}\label{claim:EUVempty}
If the vertex sets $V_{x,y}$ and $V_{y,z}$ are not empty, and the edge set $E(V_{x,y},V_{y,z})$ is empty, then $|V_{x,y}|=|V_{y,z}|=1$.
\end{Claim}
\begin{proof}
Suppose $|V_{x,y}|\geq 2$ and $V_{y,z}\neq \emptyset$.
Take $u,u'\in V_{x,y}$ and $v\in V_{y,z}$.
The result follows since the vertex set $\{a,b,c,u,u',v\}$ induces a graph isomorphic to the forbidden graph $\Gaq$.
 \end{proof} 

\begin{Claim}\label{claim:EUVnotemptyEUWemptyEVWempty}
If $E(V_{x,y},V_{y,z})\neq\emptyset$, $E(V_{x,y},V_{x,z})=\emptyset$ and $E(V_{y,z},V_{x,z})=\emptyset$, then $V_{x,z}$ is empty.
\end{Claim}
\begin{proof}
Let $u\in V_{x,y}$ and $v\in V_{y,z}$ such that $uv\in E(G)$.
Suppose there exists $w\in V_{x,z}$.
Then $uw,vw\notin E(G)$.
Since the induced subgraph $G[\{a,b,c,u,w,v\}]$ is isomorphic to $\Gay$, then we get a contradiction.
Then $V_{x,z}$ is empty.
\end{proof}
   
\begin{Claim}\label{claim:EUVnotempty}
If $E(V_{x,y},V_{y,z})\neq\emptyset$, then $E(V_{x,y},V_{y,z})$ induces either a complete bipartite graph or a complete bipartite graph minus an edge.
\end{Claim}
\begin{proof}
Let $u\in V_{x,y}$.
Suppose there exist $v,v'\in V_{y,z}$ such that $u v, u v' \notin E(G)$.
Since the induced subgraph $G[\{a,b,c,u,v,v'\}]$ is isomorphic to $\Gaq$, which is forbidden, then the vertex $u$ is adjacent with at least all but one vertices in $V_{y,z}$.
In a similar way, we have that each vertex in $V_{y,z}$ is adjacent with at at least all but one vertices in $V_{x,y}$.
Therefore, the edge set $E(V_{x,y},V_{y,z})$ induces a complete bipartite graph minus a matching.
Now suppose this matching has cardinality greater or equal to 2.
Then there exist $u,u'\in V_{x,y}$ and $v,v'\in V_{y,z}$ such that $uv',u'v\notin E(G)$ and $uv,u'v'\in E(G)$.
Since $G[\{ v,u,x, u',v'\}]$ is isomorphic to $\Gaa$, then  we get a contradiction and the matching has cardinality at most 1.
Therefore, $E(V_{xy},V_{y,z})$ induces a complete bipartite graph or a complete bipartite graph minus an edge.
\end{proof}
 
\begin{Claim}\label{claim:EUVnotemptyEUWnotemptyEVWempty}
If $E(V_{x,y},V_{y,z})\neq\emptyset$, $E(V_{y,z},V_{x,z})\neq\emptyset$ and $E(V_{x,y},V_{x,z})=\emptyset$, then $V_{x,y}=\{v_1\}$, $V_{x,z}=\{v_2\}$ and one of the following two statements holds:
\begin{itemize}
\item $E(v_1,V_{y,z})=\{v_1v : v\in V_{y,z}\}$ and $E(v_2,V_{y,z})=\{v_2v : v\in V_{y,z}\}$, or
\item there exists $v_3\in V_{y,z}$ such that $E(v_1,V_{y,z})=\{v_1v : v\in V_{y,z}-v_3\}$ and $E(v_2,V_{y,z})=\{v_2v : v\in V_{y,z}-v_3\}$.
\end{itemize}
\end{Claim}
\begin{proof}
Since there is no edge joining a vertex in $V_{x,y}$ and a vertex in $V_{x,z}$, then Claim \ref{claim:EUVempty} implies that $|V_{x,y}|=|V_{x,z}|=1$.
Let $V_{x,y}=\{v_1\}$ and $V_{x,z}=\{v_2\}$.
By Claim \ref{claim:EUVnotempty}, each edge set $E(V_{x,y},V_{y,z})$ and $E(V_{y,z},V_{x,z})$  induces either a complete bipartite graph or a complete bipartite graph minus an edge.
If both sets $E(V_{x,y},V_{y,z})$ and $E(V_{y,z},V_{x,z})$ induce either a complete bipartite graph, then we are done.
So it remains to check two cases: 
\begin{itemize}
\item when one edge set induces a complete bipartite graph and the other one induces a complete bipartite graph minus an edge, and 
\item when both edge sets induce a complete bipartite graph minus an edge.
\end{itemize}
First consider the former case.
Suppose there exists $v_3\in V_{y,z}$ such that $E(v_1,V_{y,z})=\{v_1v : v\in V_{y,z}-v_3\}$ and $E(v_2,V_{y,z})=\{v_2v : v\in V_{y,z}\}$.
Since $G[\{ a,b,c,v_1,v_2,v_3\}]$ is isomorphic to $\Gay$, then this case is not possible.
Now consider second case.
Let $v_3, v_4\in V_{y,z}$.
There are two possible cases: either $v_1v_3,v_2v_3\notin E(G)$, or $v_1v_3,v_2v_4\notin E(G)$.
Suppose $v_1v_3,v_2v_4\notin E(G)$.
Since $G[\{ v_4,v_1,z,v_2,v_3\}]$ is isomorphic to $\Gaa$, then this case is impossible.
And therefore $v_1v_3,v_2v_3\notin E(G)$.
\end{proof}
 
\begin{Claim}
If each edge set $E(V_{x,y},V_{y,z})$ is not empty, then the induced subgraph $G[V_{a,b}\cup V_{a,c} \cup V_{b,c}]$ is isomorphic to one of the following graphs:
\begin{itemize}
\item a complete tripartite graph,
\item a complete tripartite graph minus an edge, or
\item a complete tripartite graph minus the edges $v_1 v_2, v_2 v_3, v_3 v_1$, where $v_1\in V_{x,y}$, $v_2\in V_{y,z}$, $v_3\in V_{x,z}$.
\end{itemize}
\end{Claim}
\begin{proof}
By Claim \ref{claim:EUVnotempty}, the induced subgraph $G[V_{a,b}\cup V_{a,c} \cup V_{b,c}]$ is a complete tripartite graph minus at most three edges.
We will analyze the cases where 2 or 3 edges have been removed.

Suppose $G[V_{a,b}\cup V_{a,c} \cup V_{b,c}]$ induces a complete tripartite graph minus 2 edges.
Since $E(V_{x,y},V_{y,z})$ induces complete bipartite graph or complete bipartite graph minus an edge, then the 2 edges cannot be removed from a unique edge set $E(V_{x,y},V_{y,z})$.
Then there are two possibilities:
\begin{itemize}
\item there exist $u, u' \in V_{x,y}$, $v\in V_{y,z}$ and $w\in V_{x,z}$ such that $uv,u'w \notin E(G)$, or
\item there exist $u \in V_{xy}$, $v\in V_{y,z}$ and $w\in V_{x,z}$ such that $uv,uw \notin E(G)$.
\end{itemize}
In the first case the graph induced by the set $\{x,v,u,u',w,z\}$ is isomorphic to the forbidden graph $\Gar$; which is impossible.
And in the second case, the the induced subgraph $G[\{a,b,c,u,v,w\}]$ is isomorphic to $\Gan$, which is impossible.
Thus the case where 2 edges are removed from $G[V_{a,b}\cup V_{a,c} \cup V_{b,c}]$ is not possible.

Suppose $G[V_{a,b}\cup V_{a,c} \cup V_{b,c}]$ is a complete tripartite graph minus 3 edges.
Since $E(V_{x,y},V_{y,z})$ induces a complete bipartite graph or a complete bipartite graph minus an edge, then the 3 edges cannot be removed from a unique edge set $E(V_{x,y},V_{y,z})$.
Thus there are four possible cases:
\begin{enumerate}[(a)]
\item $v_1 v_4, v_3 v_6, v_5 v_2 \notin E(G)$, where $v_1,v_2\in V_{x,y}$, $v_3, v_4\in V_{y,z}$ and $v_5,v_6\in V_{x,z}$
\item $v_1 v_3, v_2 v_5, v_5 v_4 \notin E(G)$, where $v_1,v_2\in V_{x,y}$, $v_3, v_4\in V_{y,z}$ and $v_5\in V_{x,z}$
\item $v_2 v_4, v_4 v_5, v_5 v_2 \notin E(G)$, where $v_1,v_2\in V_{x,y}$, $v_3, v_4\in V_{y,z}$ and $v_5\in V_{x,z}$ and
\item $v_1 v_3, v_3 v_5, v_5 v_2 \notin E(G)$, where $v_1,v_2\in V_{x,y}$, $v_3, v_4\in V_{y,z}$ and $v_5\in V_{x,z}$.
\end{enumerate}
Cases $(a)$, $(b)$ and $(d)$ are impossible, the argument is the following.
Case $(a)$ is impossible because the induced subgraph $G[\{v_2,v_3,v_5,v_6,x,y\}]$ is isomorphic to the forbidden graph $\Gar$.
In case $(b)$, the induced subgraph $G[\{v_2, v_5, v_4, a, b, c\}]$ is isomorphic to the forbidden graph $\Gay$.
And in case $(d)$, the induced subgraph $G[\{v_2, v_3, v_5, a,b,c\}]$ is isomorphic to the forbidden graph $\Gay$.
Thus when 3 edges are removed the only possible case is $(c)$.
\end{proof}
 
By previous Claims we have the following cases:
\begin{enumerate}
\item when the set $V_{x,y}$ is the only not empty set,
\item when $V_{x,z}=\emptyset$, $V_{x,y}=\{v_{xy}\}$, $V_{y,z}=\{v_{yz}\}$ and $v_{xy}v_{yz}\notin E(G)$,
\item when $V_{x,z}=\emptyset$ and $E(V_{x,y},V_{y,z})$ induces a bipartite complete graph,
\item when $V_{x,z}=\emptyset$ and $E(V_{x,y},V_{y,z})$ induces a bipartite complete graph minus an edge,
\item when $V_{a,b}=\{v_{ab}\}$, $V_{a,c}=\{v_{ac}\}$, $V_{b,c}=\{v_{bc}\}$ and $v_{ab}v_{ac},v_{ab}v_{bc},v_{ac}v_{bc}\notin E(G)$,
\item when $V_{x,z}=\{v_{xz}\}$, $V_{y,z}=\{v_{yz}\}$, $v_{xz}v_{yz}\notin E(G)$, and the edge sets $E(v_{xz},V_{x,y})$ and $E(v_{yz},V_{x,y})$ induce a complete bipartite graph,
\item when $V_{x,z}=\{v_{xz}\}$, $V_{y,z}=\{v_{yz}\}$, $v_{xz}v_{yz}\notin E(G)$, and  there exists $v_{xy}\in V_{x,y}$ 
such that $E(v_{xz},V_{x,y})=\{v_{xz}v : v\in V_{x,y}-v_{xy}\}$ and $E(v_{yz},V_{x,y})=\{v_{yz}v : v\in V_{x,y}-v_{xy}\}$.
\item when $G[V_{a,b}\cup V_{a,c} \cup V_{b,c}]$ is isomorphic to a complete tripartite graph,
\item when $G[V_{a,b}\cup V_{a,c} \cup V_{b,c}]$ is isomorphic to a complete tripartite graph minus an edge,
\item when $G[V_{a,b}\cup V_{a,c} \cup V_{b,c}]$ is isomorphic to a complete tripartite graph, where $v_1\in V_{x,y}$, $v_2\in V_{y,z}$, $v_3\in V_{x,z}$ and $v_1 v_2, v_2 v_3, v_3 v_1 \notin E(G)$.
\end{enumerate}

Now we describe the vertex set $V_\emptyset$.
\begin{Remark}\label{rmk:VaVbVcempty:V0}
Let $w,w'\in V_\emptyset$.
Suppose $w$ is adjacent with a vertex in $V_{x,y}$ and with $w'$.
Then the vertex $w'$ is adjacent with a vertex in $V_{x,y}$, because otherwise the shortest path from $w'$ to $z$ would contains a graph isomorphic to $\Gaa$.
Thus each vertex in $V_\emptyset$ is adjacent with a vertex in $V_{x,y}$ for some $\{x,y\}\subset\{a,b,c\}$.
\end{Remark}

\begin{Claim}\label{clm:VxythenEV0Vxy}
If $w\in V_\emptyset$ is adjacent with $v\in V_{x,y}$, then either $w$ is adjacent only with $v$ and with no other vertex in $V_{x,y}$, or $w$ is adjacent with each vertex in $V_{x,y}$.
Moreover, if each vertex in $V_\emptyset$ is adjacent with a vertex in $V_{x,y}$, then either exists a vertex $v\in V_{v,x}$ 
such that each vertex in $V_\emptyset$ is adjacent with $v$, or each vertex in $V_{\emptyset}$ is adjacent with each vertex in $V_{x,y}$.
\end{Claim}
\begin{proof}
Since the first statement is easy when $V_{x,y}$ has cardinality at most 2, then we assume that $V_{x,y}$ has cardinality at least 3.
Let $v',v''\in V_{x,y}$.
Suppose $w$ is adjacent with $v$ and $v'$ but not adjacent with $v''$.
Since $G[\{x,z,v,v',v'',w\}]$ is isomorphic to $\Gad$, then we get a contradiction.
And then $w$ is adjacent only with $v$ or with $v,v'$ and $v''$.

Let $w,w'\in V_{\emptyset}$.
Suppose there is $v\in V_{x,y}$ such that $wv\in E(G)$ and $w'v\notin E(G)$.
The vertex $w$ is not adjacent with $w'$, because otherwise by Remark \ref{rmk:VaVbVcempty:V0} we get a contradiction.
Let $v'\in V_{x,y}$ such that $w'$ is adjacent with $w'$.
Thus there are two possible cases:
\begin{itemize}
\item $ww',wv'\notin E(G)$, and
\item $wv'\in E(G)$ and $ww'\notin E(G)$.
\end{itemize}
Since in the first case $G[\{w,v,x,v',w'\}]$ is isomorphic to $\Gaa$ and in the second case $G[\{x,z,v,v',w,w'\}]$ is isomorphic to $\Gaj$, then we get a contradiction and thus there is no vertex in $V_{x,y}$ adjacent with a vertex in $V_\emptyset$ and not adjacent with other vertex in $V_\emptyset$.
\end{proof}

\begin{Claim}\label{clm:VxythenV0Kortrivial}
Let $v\in V_{x,y}$.
If each vertex in $V_\emptyset$ is adjacent with $v$, then $V_\emptyset$ induces either $K_2$ or it is a trivial graph.
If furthermore there exists $v'\in V_{x,y}$ such that no vertex in $V_\emptyset$ is adjacent with $v'$, then $V_\emptyset$ is a clique of cardinality at most 2.
\end{Claim}
\begin{proof}
First note that $P_3$ is forbidden as induced subgraph in $G[V_\emptyset]$.
It is because if the vertices $w_1, w_2, w_3 \in V_\emptyset$ induce a graph isomorphic to $P_3$, then $G[\{ x, y, v, w_1, w_2, w_3\}]\simeq \Gah$.
Now we will see that each component in $G[V_\emptyset]$ is a clique.
Let $C$ be a component in $G[V_\emptyset]$.
Suppose $C$ is not a clique, then there are two vertices not adjacent in $C$, say $w$ and $w'$.
Let $P$ be the smallest path contained in $C$ between $w$ and $w'$.
The length of $P$ is greater or equal to $3$. 
So $P_3$ is an induced subgraph of $P$, and hence of $C$.
Which is a contradiction, and therefore, $C$ is a clique.
On the other hand, the graph $K_2+K_1$ is forbidden as induced subgraph in $G[V_\emptyset]$.
It is because if $w_1, w_2, w_3 \in V_\emptyset$ such that $G[\{ w_1, w_2, w_3\}]\simeq K_2+K_1$, then $G[\{ x, y, v, w_1, w_2, w_3\}]\simeq \Gai$.
Therefore, if $G[V_\emptyset]$ has more than one component, then each component has cardinality one.

Let $v,v'\in V_{x,y}$ such that each vertex in $V_\emptyset$ is adjacent with $v$, and no vertex in $V_\emptyset$ is adjacent with $v'$.
Suppose $V_\emptyset$ induces a stable set of cardinality at least 2.
Take $w,w'\in V_\emptyset$.
Then we get a contradiction since the induced graph $G[\{w,w',v,v',x,z\}]$ is isomorphic to $\Gab$.
Thus $V_\emptyset$ is a clique of cardinality at most 2.
\end{proof}

Thus by Claims \ref{clm:VxythenEV0Vxy} and \ref{clm:VxythenV0Kortrivial}, in case (1) we have the following possible cases:
\begin{itemize}
\item $V_\emptyset$ is a clique of cardinality at most 2, and each vertex in $V_\emptyset$ is adjacent with only one vertex in $V_{x,y}$, and
\item $V_\emptyset$ is a trivial graph and $E(V_\emptyset,V_{x,y})$ induces a complete bipartite graph.
\end{itemize}
Note that these graphs are isomorphic to an induced subgraph of a graph in $\mathcal F_1^1$.

\begin{Claim}\label{clm:VabVbcVbcEV0VxyVyz1}
If $E(V_{x,y},V_{y,z})=\emptyset$ and $E(V_{\emptyset},V_{x,y}\cup V_{y,z})\neq\emptyset$, then $V_\emptyset$ is a clique of cardinality at most 2, and each vertex in $V_\emptyset$ is adjacent each vertex in $V_{x,y}\cup V_{y,z}$
\end{Claim}
\begin{proof}
Let $v_{xy}\in V_{x,y}$, $v_{yz}\in V_{y,z}$ such that $v_{xy}v_{yz}\notin E(G)$.
It is easy to see that if $w$ is adjacent with $v_{xy}$ or $v_{yz}$, then $w$ is adjacent with both vertices, because otherwise $G$ has an induced subgraph isomorphic to $\Gaa$.
Thus each vertex in $V_{\emptyset}$ is adjacent with each vertex in $V_{x,y}\cup V_{y,z}$.
Now suppose $w,w'\in V_\emptyset$ such that $w$ and $w'$ are not adjacent.
Since the induced subgraph $G[\{w,w',x,y,v_{xy},v_{yz}\}]$ is isomorphic to $\Gap$, then we get a contradiction and $V_\emptyset$ induces a clique of cardinality at most 2.
\end{proof}

By previous Claim we get that in case (2), $V_\emptyset$ is a clique of cardinality at most 2 and each vertex in $V_\emptyset$ is adjacent each vertex in $V_{x,y}\cup V_{y,z}$.
This graph is isomorphic to an induced subgraph of a graph in $\mathcal F_1^1$.

\begin{Claim}\label{clm:VabVbcVbcEV0VxyVyz2}
If the edge set $E(V_{x,y},V_{y,z})$ induces a complete bipartite graph and $E(V_{\emptyset},V_{x,y}\cup V_{y,z})\neq\emptyset$, then $V_\emptyset$ is a clique of cardinality at most 2, and each vertex in $V_\emptyset$ is adjacent only to one vertex in $V_{x,y}\cup V_{y,z}$.
\end{Claim}
\begin{proof}
First we prove that there each vertex in $V_\emptyset$ is adjacent with a vertex in only one of the sets $V_{x,y}$ or $V_{y,z}$.
Suppose there exists a vertex $w\in V_\emptyset$ adjacent with $v\in V_{x,y}$ and $u\in V_{y,z}$.
Since the induced subgraph $G[\{w,a,b,c,v,u\}]$ is isomorphic to $\Gay$, then we get a contradiction and each vertex in $V_\emptyset$ is adjacent only with vertices of one of the vertex sets $V_{x,y}$ or $V_{y,z}$.
Suppose there are two vertices $w,w'\in V_\emptyset$ such that $w$ is adjacent with $v\in V_{x,y}$, and $w'$ is adjacent with $u\in V_{y,z}$.
Since $G[\{w,w',u,v,x,z\}]$ is isomorphic to $\Gaj$, then this is impossible and the vertices of $V_\emptyset$ are adjacent only to vertices in one of the vertex sets either $V_{x,y}$ or $V_{y,z}$.
Now suppose that $w\in V_\emptyset$ is adjacent with two vertices in $V_{x,y}$, say $v$ and $v'$.
Take $u\in V_{y,z}$.
So $u$ is adjacent with both $v$ and $v'$.
Since the induced subgraph $G[\{w,v,v',u,a,b,c\}]$ is isomorphic to $\Gbl$, then this cannot occur.
Thus each vertex in $V_{\emptyset}$ is adjacent only with one vertex in $V_{x,y}\cup V_{y,z}$.
Finally suppose $V_{\emptyset}$ induces a trivial graph of cardinality at least 2.
Let $w,w'\in V_{\emptyset}$ adjacent with $v\in V_{x,y}$.
Take $u\in V_{y,z}$, so $u$ is adjacent with both $v$.
Since the induced subgraph $G[\{w,w,v,u,x,z\}]$ is isomorphic to $\Gag$, we get a contradiction and the result follows.
\end{proof}

By previous Claim we have that in case (3) the vertex set $V_\emptyset$ induces a clique of cardinality at most 2 and each vertex in $V_\emptyset$ is adjacent only to one vertex in $V_{x,y}\cup V_{y,z}$.
In this case, the graph is isomorphic to an induced subgraph of a graph in $\mathcal F_1^1$.

\begin{Claim}\label{clm:VabVbcVbcEV0VxyVyz3}
Let $u\in V_{x,y}$ and $v,v'\in V_{y,z}$ such that $u$ is adjacent with $v$ but not with $v'$.
If $w\in V_\emptyset$, then $w$ is not adjacent with $v$.
\end{Claim}
\begin{proof}
Suppose $w$ is adjacent with $v$.
Note that if $w$ is adjacent with $u$ or $v'$, then $w$ is adjacent with both $u$ and $v'$.
Thus there are two cases: either $w$ is adjacent only with $v$, or $w$ is adjacent with $v,v'$ and $u$.
Both cases are impossible because in the former case $G[\{x,z,u,v,v',w\}]$ is isomorphic to $\Gaj$, meanwhile in the second case $G[\{y,z,u,v,v',w\}]$ is isomorphic to $\Gar$; which is a contradiction.
\end{proof}
Consider case (4).
Let $u\in V_{x,y}$ and $v\in V_{y,z}$ such that $u$ is not adjacent with $v$.
By Claim \ref{clm:VabVbcVbcEV0VxyVyz3}, each vertex in $V_\emptyset$ is adjacent with $u$ or with $v$.
It is not difficult to see that in fact each vertex in $V_\emptyset$ is adjacent with both $u$ and $v$, because otherwise $G$ would contain $\Gaa$ as induced subgraph.
By applying Claim \ref{clm:VabVbcVbcEV0VxyVyz1} to the induced subgraph $G[\{u,v\}\cup V_\emptyset]$, we get that $V_\emptyset$ is a clique of cardinality at most 2.
And this graph is isomorphic to an induced subgraph of a graph in $\mathcal F_1^1$.

Now consider case (5), by Claim \ref{clm:VabVbcVbcEV0VxyVyz1}, we get that $V_\emptyset$ is a clique of cardinality at most 2, 
and each vertex in $V_\emptyset$ is adjacent each vertex in $V_{a,b}\cup V_{b,c}\cup V_{a,c}$.
And this graph is isomorphic to an induced subgraph of a graph in $\mathcal F_1^1$.

\begin{Claim}\label{clm:VabVbcVbcEV0VxyVyz4}
Let $u_1\in V_{x,y}$, $u_2\in V_{y,z}$ and $u_3\in V_{x,z}$ such that $u_1$ is adjacent with both $u_2$ and $u_3$, and $u_2 u_3\notin E(G)$.
If $w\in V_\emptyset$, then $w$ is not adjacent with $u_1$.
\end{Claim}
\begin{proof}
Suppose $w$ is adjacent with $u_1$.
Note that if $w$ is adjacent with $u_2$ or $u_3$, then $w$ is adjacent with both $u_2$ and $u_3$.
Thus there are two cases: either $w$ is adjacent only with $u_1$, or $w$ is adjacent with $u_1,u_2$ and $u_3$.
Both cases are impossible because in the former case $G[\{x,y,u_1,u_2,u_3,w\}]$ is isomorphic to $\Gam$, 
meanwhile in the second case $G[\{x,z,u_1,u_2,u_3,w\}]$ is isomorphic to $\Gaz$; which is a contradiction.
\end{proof}

Consider the Cases (6) and (9).
Let $v_{xz}\in V_{x,z}$ and $v_{yz}\in V_{y,z}$ such that $v_{xz}$ is not adjacent with $v_{yz}$.
By Claims \ref{clm:VabVbcVbcEV0VxyVyz1} and \ref{clm:VabVbcVbcEV0VxyVyz4}, each vertex in $V_\emptyset$ is adjacent only with both $v_{xz}$ and $v_{yz}$.
And by Claim \ref{clm:VabVbcVbcEV0VxyVyz1}, we get that $V_\emptyset$ is a clique of cardinality at most 2.
And these graphs are isomorphic to an induced subgraph of a graph in $\mathcal F_1^1$.

Now consider the Case (7) and (10).
Let $v_{xy}\in V_{x,y}$, $v_{xz}\in V_{x,z}$ and $v_{yz}\in V_{y,z}$ such that $v_{xy}$ is not adjacent with $v_{xz}$ and $v_{yz}$, and $v_{xz}$ is not adjacent with $v_{yz}$. 
By Claims \ref{clm:VabVbcVbcEV0VxyVyz4} and \ref{clm:VabVbcVbcEV0VxyVyz1}, each vertex in $V_\emptyset$ is adjacent only with $v_{xy},v_{xz}$ and $v_{yz}$.
And by Claim \ref{clm:VabVbcVbcEV0VxyVyz1}, we get that $V_\emptyset$ is a clique of cardinality at most 2.
And these graphs are isomorphic to an induced subgraph of a graph in $\mathcal F_1^1$.

Finally in the case (8), by Claim \ref{clm:VabVbcVbcEV0VxyVyz2} we have that the vertex set $V_\emptyset$ is a clique of cardinality at most 2 and each vertex in $V_\emptyset$ is adjacent with only one vertex in $V_{a,b}\cup V_{b,c}\cup V_{a,c}$.
This case corresponds to a graph that is isomorphic to an induced subgraph of a graph in $\mathcal F_1^1$.

\subsection{Case $V_x=T_n$, where $n\geq 2$.}
First we will obtain that $E(V_x,V_{y,z})$ satisfies one of the following statements:
\begin{itemize}
\item it induces a complete bipartite graph,
\item there exists a vertex $s\in V_{x}$, we called the apex, such that $E(V_x,V_{y,z})=\{uv : u\in V_x-s \text{ and }v\in V_{y,z}\}$, or
\item there exists a vertex $s\in V_{y,z}$, we called the apex, such that $E(V_x,V_{y,z})=\{uv : u\in V_x \text{ and }v\in V_{y,z}-s\}$.
\end{itemize}
To do it, we will prove the following claims.

\begin{Claim}\label{VxequalT2:noparallelnonedges}
Let $u_1, u_2\in V_x$ and $v_1, v_2 \in V_{y,z}$ such that $u_1 v_1, u_2 v_2 \notin E(G)$.
Then either $u_1=u_2$ or $v_1=v_2$.
\end{Claim}
\begin{proof}
Suppose $u_1\neq u_2$ and $v_1\neq v_2$.
There are three possible cases:
\begin{itemize}
\item $u_1 v_2, u_2 v_1 \notin E(G)$,
\item $u_1 v_2 \notin E(G)$ and $u_2 v_1 \in E(G)$,  or
\item $u_1 v_2, u_2 v_1 \in E(G)$.
\end{itemize}
The first two cases are not possible since the induced subgraph $G[\{ x, y, u_1, u_2, v_1, v_2\}]$ would be isomorphic to $\Gab$ and $\Gaj$, respectively.
In the last case, the induced subgraph $G[\{ x, u_1, u_2, v_1, v_2\}]$ is isomorphic to $\Gaa$; which is impossible.
Thus the result follows.
\end{proof}
The last claim implies that all {\it non-edges} in $E(V_x, V_{y,z})$ are incident to a one vertex: the apex $s$.

Suppose the vertex $s$ is in $V_x$, and there are vertices $v_1, v_2\in V_{y,z}$ such that $sv_1\in E(G)$, and $sv_2\notin E(G)$.
By Claim \ref{VxequalT2:noparallelnonedges}, each vertex in $V_x-s$ is adjacent with $v_1$ and $v_2$.
Then the induced subgraph $G[ \{a, b, c, u_1, u_2, v_1, v_2\}]\simeq \Gbk$, that is impossible.
This implies that if the vertex $s\in V_x$ is not adjacent with a vertex in $V_{y,z}$, then $s$ is not adjacent with all vertices in $V_{y,z}$.
A similar argument yields that  if the apex vertex $s$ is in $V_{y,z}$ and $s$ is not adjacent with a vertex in $V_{x}$, then $s$ is not adjacent with all vertices in $V_{x}$.

Thus, we have three cases:
\begin{enumerate}[(a)]
\item $E(V_x, V_{y,z})$ is complete bipartite minus the edges between a vertex $s$ (the apex) in $V_x$ and all vertices of $V_{y,z}$,
\item $E(V_x, V_{y,z})$ is complete bipartite minus the edges between a vertex $s$ (the apex) in $V_{y,z}$ and all vertices in $V_x$, and
\item $E(V_x, V_{y,z})$ is complete bipartite.
\end{enumerate}

\begin{Claim}
If $|V_x|\geq 3$ and $v\in V_{y,z}$, then $E(v,V_x)$ either it induces a complete bipartite graph or it is empty.
\end{Claim}
\begin{proof}
Let $u_1, u_2, u_3 \in V_x$.
Suppose one of the two following possibilities happen: $vu_1\in E(G)$ and $vu_2, vu_3\notin E(G)$ , or $vu_1, vu_2 \in E(G)$ and $vu_3\notin E(G)$.
In the first case the induced subgraph $G[\{ u_1, u_2, u_3, v, x, y\}]$ is isomorphic to $\Gad$, meanwhile in the second case the induced subgraph $G[\{ u_1, u_2, u_3, v, x, y, z\}]$ is isomorphic to $\Gbd$. 
Then both cases cannot occur, and we get the result.
\end{proof}

Thus case (a) occur only when $|V_x|=2$.

In what follows we describe the vertex set $V_\emptyset$, that is, the vertex set whose vertices have no edge in common with the vertex set $\{ a,b,c\}$.
\begin{Remark}\label{re:V0Vx}
Let $w\in V_\emptyset$.
The vertex $w$ is adjacent with a vertex in $V_x \cup V_{y,z}$, because otherwise the shortest path from $w$ to $\{a,b,c\}$ would contains the graph $\Gaa$ as induced subgraph.
Let $u_1,u_2\in V_{x}$.
If $w$ is adjacent with $u_1$ or $u_2$, then $w$ is adjacent with both vertices, because otherwise the induced subgraph $G[\{a, b, c, u_1, u_2, w\}]$ would be isomorphic to $\Gag$, which is forbidden.
\end{Remark}
In case (a), we will see that each vertex in $V_\emptyset$ is adjacent with each vertex in $V_x\cup V_{y,z}$. 
Let $w$ in $V_\emptyset$.
Supppose $s\in V_x$ is the vertex that is not adjacent with any vertex in $V_{y,z}$.
If $w\in V_\emptyset$ is adjacent with one of the vertices in $\{ s\}\cup V_{y,z}$, then $w$ must to be adjacent with $s$ and each vertex in $V_{y,z}$, because otherwise let $v\in V_{y,z}$, 
then the induced subgraph $G[\{ x, y, s, w, v\}]$ would be isomorphic to $\Gaa$.  
Then by Remark \ref{re:V0Vx}, $w$ is adjacent with each vertex in $V_x\cup V_{y,z}$.
\begin{Claim}\label{claim:V0T}
The vertex set $V_\emptyset$ induces a stable set.
\end{Claim}
\begin{proof}
Suppose $w_1, w_2 \in V_\emptyset$ are adjacent.
Since both $w_1$ and $w_2$ are adjacent with $u\in V_x-s$ and $v\in V_{y,z}$, then the induced subgraph $G[\{ u, v, w_1, w_2\}]$ is isomorphic to $K_4$; which is forbidden.
Thus $w_1$ and $w_2$ are not adjacent, and therefore $V_\emptyset$ is a stable set.
\end{proof}
Thus this case corresponds to a graph that is isomorphic to an induced subgraph of a graph in $\mathcal F_1^1$.

Now consider case (b).
Let $u_1,u_2\in V_x$ and $s\in V_{y,z}$ such that $s$ is not adjacent with $u_1$ and $u_2$.
\begin{Claim}\label{claim:V0VxVyz}
If $w\in V_\emptyset$ is adjacent with a vertex in $V_{y,z}\setminus \{s\}$, then each vertex in $V_\emptyset$ is adjacent with each vertex in $V_x \cup V_{y,z}$. 
\end{Claim}
\begin{proof}
Let $w\in V_\emptyset$.
Suppose $w$ is adjacent with $v\in V_{y,z}\setminus\{s\}$.
If $w$ is not adjacent with both vertices $u_1, u_2\in V_x$, then we get that the induced subgraph $G[\{ a, b, c, w, u_1, u_2, v\}]$ is isomorphic to $\Gbd$, which cannot be.
Thus $w$ is adjacent with at least one vertex in $V_x$.
Then by Remark \ref{re:V0Vx}, $w$ is adjacent with both vertices $u_1$ and $u_2$.
If there exist $v'\in V_{y,z}\setminus\{s\}$ such that $w$ is not adjacent with $v'$, then the vertex set $\{ w, u_1, x, y, v, v'\}$ would induce a graph isomorphic to $\Gap$.
On the other hand, if $w$ is not adjacent with $s$, then the vertex set $\{ w, u_1, x, y, s\}$ induces the subgraph $\Gaa$.
Therefore, $w$ is adjacent with each vertex in $V_x \cup V_{y,z}$.

Suppose there is another vertex $w'\in V_\emptyset$.
By the above argument if $w'$ is adjacent with a vertex in $V_{y,z}\setminus\{s\}$, then it must be adjacent with each vertex in $V_x \cup V_{y,z}$.
Also if $w'$ is adjacent with a vertex in $\{s\}\cup V_x$, then $w'$ must be adjacent with each vertex in $\{s\}\cup V_x$.
So suppose $w'$ is adjacent with each vertex in $s\cup V_x$ and $w'$ is not adjacent with each vertex in $V_{y,z}\setminus \{s\}$.
Then there are two possibilities: either $ww'\notin E(G)$ or $ww'\in E(G)$.
Let $v\in V_{y,z}\setminus\{s\}$.
In the first case $G[\{ x, y, s, v, w, w'\}]\simeq \Gap$ and in the second  case $G[\{ x, y, v, w, w'\}]\simeq \Gaa$.
Since both graphs are forbidden, then we get a contradiction and thus $w'$ is adjacent with $v$.
And therefore $w'$ is adjacent with each vertex in $V_x\cup V_{y,z}$.
\end{proof}
\begin{Claim}\label{lemma:V02paVxVxVyz}
Either each vertex in $V_\emptyset$ is adjacent with each vertex in $V_x\cup\{s\}$, or each vertex in $V_\emptyset$ is adjacent with each vertex in $V_x\cup V_{y,z}$.
\end{Claim}
\begin{proof}
Let $w\in V_\emptyset$.
Clearly, if $w$ is adjacent with a vertex in $V_x\cup \{ v\}$, then $w$ is adjacent with each vertex in $V_x\cup \{v\}$, because otherwise $\Gaa$ would be an induced subgraph.
By Claim \ref{claim:V0VxVyz}, we have that if there is a vertex $w\in V_\emptyset$ adjacent with a vertex in $V_{y,z}$, different to the apex $s$, then each vertex in $V_\emptyset$ is adjacent with each vertex in $V_x$ and $V_{y,z}\setminus \{ s\}$.
Thus we get the result.
\end{proof}
\begin{Claim}\label{clm:VxTnV0K2T}
The vertex set $V_\emptyset$ induces either a clique of cardinality at most 2 or a trivial graph.
\end{Claim}
\begin{proof}
First note that $P_3$ is forbidden as induced subgraph in $V_\emptyset$.
It is because if $w_1, w_2, w_3 \in V_\emptyset$ induce $P_3$, then $G[\{ x, y, u_1, w_1, w_2, w_3\}]\simeq \Gah$.
Now we will get that each component in $G[V_\emptyset]$ is a clique.
Let $C$ be a component in $G[V_\emptyset]$.
Suppose $C$ is not a clique, then it has two vertices not adjacent, say $u$ and $v$.
Let $P$ be the smallest path in $C$ between $u$ and $v$.
Thus the length of $P$ is greater or equal to $3$. 
So $P_3$ is an induced subgraph of $P$, and hence of $C$.
Therefore, $C$ is a complete graph.

On the other hand, the graph $K_2+K_1$ is forbidden for $G[V_\emptyset]$.
It is because if $w_1, w_2, w_3 \in V_\emptyset$ such that $G[\{ w_1, w_2, w_3\}]\simeq K_2+K_1$, then $G[\{ x, y, u_1, w_1, w_2, w_3\}]\simeq \Gai$; which cannot happen.
Therefore, if $G[V_\emptyset]$ has more than one component, then each component has cardinality one.
\end{proof}
In the first case of Claim \ref{lemma:V02paVxVxVyz}, if $|V_{y,z}|\geq 2$, then $V_\emptyset$ is either $K_1$ or $K_2$, 
because if $u_1\in V_x$, $v\in V_{y,z}\setminus\{s\}$, and $w_1, w_2\in V_\emptyset$ are adjacent, then $G[\{ w_1, w_2, u_1, v, x, y\}]\simeq \Gad$.
Otherwise if $V_{y,z}=\{s\}$, then both possibilities in Claim \ref{clm:VxTnV0K2T} are allowed.
In the second case of Claim \ref{lemma:V02paVxVxVyz}, if $|V_{y,z}|\geq 3$, we have that $V_\emptyset=\emptyset$, 
because if $w\in V_\emptyset$ is adjacent with $u_1, u_2 \in V_x$ and with two vertices $v_1,v_2 \in V_{y,z}\setminus\{s\}$, we get $G[\{x,y,u_1,u_2, v_1, v_2, w\}]\simeq\Gbo$ as forbidden subgraph. 
If $|V_{y,z}|=2$, then $V_\emptyset$ is trivial since $\omega(G)=3$.
And if $V_{y,z}=\{s\}$, then both possibilities in Claim \ref{clm:VxTnV0K2T} are allowed.
With no much effort the reader can see that each of these cases corresponds to a graph isomorphic to an induced subgraph of a graph in $\mathcal F_1^1$.

Case (c).
By Claim \ref{lemma:V02paVxVxVyz}, there are two possible cases:
\begin{itemize}
\item either each vertex in $V_\emptyset$ is adjacent with each vertex in $V_x \cup V_{y,z}$, or 
\item each vertex in $V_\emptyset$ is adjacent with each vertex in $V_x$, and no vertex in $V_\emptyset$ is adjacent with any vertex in $V_{y,z}$.
\end{itemize}
In first case when $|V_{y,z}|\geq 2$, the vertex set $V_\emptyset$ is empty.
Because if $w\in V_\emptyset$ is adjacent with the vertices $v_1,v_2 \in V_{y,z}$, and $u_1, u_2 \in V_x$, then the induced subgraph $G[\{x,y,u_1,u_2, v_1, v_2, w\}]\simeq\Gbo$ which is forbidden. 
Then the vertex set $V_\emptyset$ is empty.
Otherwise when $|V_{y,z}|=1$, then the vertex set $V_\emptyset$ must be a stable set, because if there exist two adjacent vertices in $V_\emptyset$, 
then by taking a vertex in $V_x$ and a vertex in $V_{y,z}$ we get $K_4$ that is forbidden.
Finally, in the case when each vertex in $V_\emptyset$ is adjacent with each vertex in $V_x$, and no vertex in $V_\emptyset$ is adjacent with a vertex in $V_{y,z}$, 
we get that $V_\emptyset$ is a clique of cardinality at most 2. 
It is because if $w_1, w_2 \in V_\emptyset$ are such that $w_1w_2\notin E(G)$, then we get $G[\{ w_1, w_2, u_1, v, x, y\}]\simeq \Gad$ that is forbidden.
And each of these graphs are isomorphic to an induced subgraph of a graph in $\mathcal F_1^1$.

\subsection{Case $V_x$ is a complete bipartite graph of cardinality at least 3.}
Assume $V_x$ is a complete bipartite graph of cardinality at lest three with $(A,B)$ the bipartition of $V_x$. 

\begin{Claim}\label{claim:E(v,Vx)isnotempty}
If $v\in V_{y,z}$, then $E(v,V_x)\neq \emptyset$.
\end{Claim}
\begin{proof}
Suppose $E(v,V_x)=\emptyset$.
There exist $u_1, u_2, u_3 \in V_x$ such that $G[\{ u_1, u_2, u_3\}]\simeq P_3$.
And then we get contradiction since the induced subgraph $G[\{ u_1, u_2, u_3, x, y, v\}]\simeq \Gah$; which is forbidden.
\end{proof}

\begin{Claim}
Let $u\in V_x$ and $v\in V_{y,z}$.
If $u$ and $v$ are adjacent, then $v$ is adjacent with each vertex in the part ($A$ or $B$) containing $u$.
\end{Claim}
\begin{proof}
Suppose $u\in A$ and $v$ is not adjacent with any vertex in $A\setminus \{u\}$.
If $|A|=1$, then the result follows, so we may assume $|A|\geq 2$.
Since $G[V_x]$ is connected and has cardinality at least 3, then there exists $u'\in A$ and $u''\in B$ such that $G[\{u, u'', u'\}]\simeq P_3$.
We have $u'$ and $v$ are not adjacent.
There are two possibilities: either $u''v\in E(G)$ or $u''v\notin E(G)$.
In the first case the induced subgraph $G[\{u', u''. u, v, y\}]$ is isomorphic to $\Gaa$, and in the second case the induced subgraph $G[\{ x, y, u, u', u'', v\}]$ is isomorphic to $\Gar$.
Since both cases are forbidden, we get a contradiction.
\end{proof}
Previous Claims suggest to divide the vertex set $V_{y,z}$ in three subsets:
\begin{itemize}
\item $V_{y,z}^A$, the vertices in $V_{y,z}$ that are adjacent with each vertex in $A$,
\item $V_{y,z}^B$, the vertices in $V_{y,z}$ that are adjacent with each vertex in $B$, and
\item $V_{y,z}^{AB}$, the vertices in $V_{y,z}$ that are adjacent with each vertex in $A\cup B$.
\end{itemize}

In what follows, we assume $|A|\geq |B|$.

\begin{Claim}
The cardinality of the sets $V_{y,z}^A$ and $V_{y,z}^B$ is no more than 1.
\end{Claim}
\begin{proof}
Suppose there exist $v, v' \in V_{y,z}^A$.
Let $u\in A$, and $u' \in B$.
Since $G[\{ u, u', v, v', x, y\}]$ is isomorphic to $\Gap$, which is forbidden, then we get a contradiction.
The case $V_{y,z}^B$ is similar.
\end{proof}

\begin{Claim}
If $|B|\geq 2$ and $V_{y,z}\neq\emptyset$, then one of the sets $V_{y,z}^B$ or $V_{y,z}^{AB}$ is empty.
\end{Claim}
\begin{proof}
Suppose $v\in V_{y,z}^B$ and $v'\in V_{y,z}^{AB}$.
Let $u, u'\in B$, and $u'' \in A$.
Thus $G[\{v, v', u, u', u'', x, y\}]\simeq \Gbo$.
Which is impossible.
\end{proof}

Since $|A|\geq2$, then by applying previous Claim to $A$, one of the sets $V_{y,z}^A$ or $V_{y,z}^{AB}$ is empty.

Thus the possible cases we have are the following:
\begin{enumerate}[(a)]
\item $V_{y,z}=\emptyset$,
\item $V_{y,z}^B\cup V_{y,z}^{AB}=\emptyset$ and $|V_{y,z}^A|=1$,
\item $V_{y,z}^A\cup V_{y,z}^{AB}=\emptyset$ and $|V_{y,z}^B|=1$,
\item $V_{y,z}^{AB}=\emptyset$ and $|V_{y,z}^A|=|V_{y,z}^B|=1$,
\item $V_{y,z}^A=\emptyset$, $|B|=1$, $|V_{y,z}^B|=1$ and $|V_{y,z}^{AB}|\geq 1$, and
\item $V_{y,z}^A\cup V_{y,z}^B=\emptyset$ and $|V_{y,z}^{AB}|\geq1$.
\end{enumerate}

Now we describe $V_\emptyset$, that is, the set of vertices not adjacent with any vertex in $\{a,b,c\}$.
Let $w\in V_\emptyset$.
The vertex $w$ is adjacent with a vertex in $V_x \cup V_{y,z}$, because otherwise the shortest path from $w$ to $\{x,y\}$ would contains the graph $\Gaa$ as induced subgraph.

\begin{Claim}\label{cla:wV0A}
Let $w\in V_\emptyset$.
If $w$ is adjacent with a vertex in $V_x$, then $w$ is adjacent with each vertex in the parts of the partition $(A,B)$ with cardinality greater or equal to 2.
\end{Claim}
\begin{proof}
Let $v\in V_x$ be a vertex adjacent with $w$.
Suppose $v\in B$.
we will prove two things: (1) if $|B|\geq 2$, then $w$ is adjacent with each vertex in $B$, and (2) $w$ is adjacent with each vertex  in $A$.

Let us consider case when $|B|\geq 2$.
We will see that $w$ is adjacent with each vertex in $B$. 
Suppose there is a vertex $v'\in B$ not adjacent with $w$.
Take $u\in A$.
Thus there are two possibilities: either $u$ and $w$ are adjacent or not.
The case $uw\in E(G)$ is impossible because $G[\{ u, v, v', w, x, y\}]\simeq \Gak$, which is forbidden.
Meanwhile, the case $uw\notin E(G)$ is impossible because $G[\{ u, v, v', w, x, y\}]\simeq \Gao$, which is forbidden.
Thus $w$ is adjacent with $v'$, and therefore $w$ is adjacent with each vertex in $B$.

Now we see that $w$ is adjacent with each vertex in $A$.
Note that in this case $|B|$ may be equal to 1.
Suppose $w$ is not adjacent with any vertex in $A$.
Let $u, u'\in A$.
Since the induced subgraph $G[\{ w, v, u, u', x, y\}]$ is isomorphic to the forbidden graph $\Gaf$, then we get a contradiction.
Thus $w$ is adjacent with a vertex in $A$.
Now applying the previous case (1) to $A$, we get that $w$ is adjacent with each vertex in $A$. 
\end{proof}

Next Claim show us what happens in the case when $|B|=1$.

\begin{Claim}\label{cla:wV0B=1}
If $|B|=1$ and $E(V_\emptyset,V_x)\neq\emptyset$, then only one of the edges sets $E(V_\emptyset,V_x)$ or $E(V_\emptyset,A)$ induces a complete bipartite graph.
\end{Claim}
\begin{proof}
Let $w,w'\in V_\emptyset$, $u\in B$, and $u'\in A$.
By Claim \ref{cla:wV0A}, vertices $w$ and $w'$ are adjacent with each vertex in $A$.
Suppose $w$ is adjacent with $u$, and $w'$ is not adjacent with $u$.
There are two possibilities: either $w$ and $w'$ are adjacent or not. 
If $ww'\in E(G)$, then the induced subgraph $G[w,w',u,x,y]$ is isomorphic to $\Gaa$; which is impossible.
On the other hand, if $ww'\notin E(G)$, then $G[\{ w, w', u, u', x, y\}]$ is isomorphic to $\Gak$, which is forbidden.
So we get a contradiction, and the result follows.
\end{proof}

\begin{Claim}\label{cla:wVyzAB2}
Let $w\in V_\emptyset$.
If $w$ is adjacent with $v\in V_{y,z}$, then $w$ is adjacent with each vertex in the parts of partition $(A,B)$ with cardinality greater or equal to 2.
Moreover, if $v\in V_{y,z}^A \cup V_{y,z}^{AB}$ and $|B|=1$, then $w$ is adjacent with the unique vertex in $B$.
\end{Claim}
\begin{proof}
Let $w\in V_\emptyset$ and $v\in V_{y,z}$ such that $w$ and $v$ are adjacent.
By Claim \ref{claim:E(v,Vx)isnotempty}, $E(v,V_x)\neq\emptyset$, and therefore there are three cases: $v\in V_{y,z}^A$, $v\in V_{y,z}^{AB}$, or $v\in V_{y,z}^B$.

Suppose $v\in V_{y,z}^A$.
If $E(w,V_x)=\emptyset$, then by taking $u,u\in A$, the forbidden induced subgraph $G[W\cup\{ w,v,u,u'\}]\simeq \Gbd$ would appear and we get a contradiction.
Thus $E(w,V_x)\neq\emptyset$.
By Claim \ref{cla:wV0A}, $w$ is adjacent with each vertex in the parts of $V_x$ with cardinality greater or equal to 2.
When $|B|=1$, take $u''\in B$.
If $wu''\notin E(G)$, then $G[\{ w, v, y, x, u''\}]$ would be isomorphic to the forbidden graph $\Gaa$.
Therefore, $wu''\in E(G)$.

Suppose $v\in V_{y,z}^{AB}$.
In a similar way than in previous case we get that $E(w,V_x)\neq \emptyset$, and $v$ is adjacent with each vertex in the parts of the partition $(A,B)$ of cardinality greater or equal to 2.
In the case when $|B|=1$, suppose $u\in B$ with $wu\notin E(G)$.
Take $u'\in A$, we know that $u'w\in E(G)$.
Since $G[\{w, u, u', x, y, v\}]\simeq \Gay$, then we get a contradiction.
Thus $wu\in E(G)$.

Now suppose $v\in V_{y,z}^B$.
It is easy to see that $v$ is adjacent with each vertex in $A$, because otherwise $\Gaa$ would appear.
Therefore $E(w,V_x)\neq \emptyset$, and by Claim \ref{cla:wV0A}, we have that $w$ is adjacent with each vertex in the parts of the partition $(A,B)$ of cardinality greater or equal to 2.
\end{proof}	

\begin{Claim}\label{cla:V0BKthenV0VABempty}
If $E(V_\emptyset, B)\neq \emptyset$, then $E(V_\emptyset,V_{y,z}^{AB})=\emptyset$.
Moreover, if $E(V_\emptyset, B)\neq \emptyset$ and $E(V_\emptyset, V_{y,z}^B)\neq\emptyset$, then $V_{y,z}^{AB}=\emptyset$.
\end{Claim}
\begin{proof}
Let $w\in V_\emptyset$ and $v\in V_{y,z}^{AB}$.
Suppose $w$ and $v$ are adjacent.
Take $u_1\in A$ and $u_2\in B$.
Since $v$ is in $V_{y,z}^{AB}$ and $u_1$ is adjacent with $u_2$, then $G[\{v, u_1,u_2\}]$ is isomorphic to $K_3$.
On the other hand, by Claims \ref{cla:wV0A} and \ref{cla:wV0B=1}, $w$ is adjacent with each vertex in $A\cup B$.
Thus the induced subgraph $G[\{ w, v, u_1, u_2\}]$ is isomorphic to $K_4$ that is forbidden, and therefore $w$ cannot be adjacent with $v$.

Let $w$ is adjacent with $u_2\in B$ and $v'\in V_{y,z}^B$, and suppose there exists a vertex $v\in V_{y,z}^{AB}$.
Since $w$ is not adjacent with $v$, then $G[\{ w, v, v', u_2, x, y\}]$ is isomorphic to $\Gap$; which it is not possible.
\end{proof}

\begin{Claim}	
Let $w\in V_\emptyset$ and $v\in V_{y,z}$.
If $w$ and $v$ are adjacent, then each vertex in $V_\emptyset$ is adjacent with $v$.
\end{Claim}
\begin{proof}
Suppose $w'\in V_\emptyset$ such that $w'$ is not adjacent with $v$.
By Claims \ref{cla:wV0A} and \ref{cla:wV0B=1}, both vertices $w$ and $w'$ are adjacent with each vertex in $A$.
Let $u\in A$.
There are four cases obtained by the combinations of the following possible cases: either $ww'\in E(G)$ or $ww'\notin E(G)$, and either $v\in V_{y,z}^{A}$ or $v\in V_{y,z}^{B}$.
When $ww'\in E(G)$ and $v\in V_{y,z}^{B}$, the induced subgraph $G[\{ w', w, v, y, x\}]$ is isomorphic to $\Gaa$; which is not possible.
When $ww'\notin E(G)$ and $v\in V_{y,z}^{B}$, the induced subgraph $G[\{ w', u, w, v, y\}]$ is isomorphic to $\Gaa$; which is not possible.
When $ww'\in E(G)$ and $v\in V_{y,z}^{A}$, the induced subgraph $G[\{ w', w, v, y, x\}]$ is isomorphic to $\Gaa$; which is not possible.
Finally, when $ww'\notin E(G)$ and $v\in V_{y,z}^{A}$, the induced subgraph $G[\{ w', w, x, y, v, u\}]$ is isomorphic to $\Gal$; which is not possible.
Thus $w'$ is adjacent with $v$.
\end{proof}

\begin{Claim}
Let $v\in V_{y,z}^A$ and $v'\in V_{y,z}^B$.
If $w\in V_\emptyset$ is adjacent with $v$ or $v'$, then $w$ is adjacent with both $v$ and $v'$.
\end{Claim}
\begin{proof}
First suppose $w$ is adjacent with $v$ and not with $v'$.
Let $u\in A$.
By Claim \ref{cla:wVyzAB2}, $w$ is adjacent with a vertex $u\in A$.
Thus $G[\{ w, u, x, y, v'\}]$ is isomorphic to $\Gaa$ that is a contradiction.
Therefore $w$ and $v'$ are adjacent.
Now suppose $w$ is adjacent with $v'$ and not with $v$.
Let $u\in B$.
If $u$ is adjacent with $w$, then a $\Gaa$ is obtained in a similar way than previous case.
Thus assume $u$ is not adjacent with $w$.
Then $G[\{ w, x, y, u, v, v'\}]$ is isomorphic to $\Gaj$; which is impossible.
Therefore, $w$ and $v$ are adjacent.
\end{proof}

\begin{Claim}
If $V_{y,z}^{AB}\neq\emptyset$, then $E(V_\emptyset,B)=\emptyset$.
\end{Claim}
\begin{proof}
Let $u_1, u_2\in A$, $u_3\in B$, $w\in V_\emptyset$ and $v\in V_{y,z}^{AB}$.
By Claim \ref{cla:wV0A}, $w$ is adjacent with $u_1$ and $u_2$.
Suppose $w$ is adjacent with $u_3$.
By Claim \ref{cla:V0BKthenV0VABempty}, the vertices $w$ and $v$ are not adjacent.
Thus $G[\{x,y,u_1,u_2,u_3,v,w\}]$ is isomorphic to $\Gbl$, which is a contradiction.
Therefore, $E(V_\emptyset,B)=\emptyset$.
\end{proof}

Thus applying previous Claims to cases (a) - (f), we obtain the following possibilities:
\begin{enumerate}
\item $V_{y,z}=\emptyset$, $|B|=1$, and $E(V_\emptyset,A)$ induces a complete bipartite graph,
\item $V_{y,z}=\emptyset$, $|B|\geq1$, and $E(V_\emptyset,V_x)$ induces a complete bipartite graph,
\item $V_{y,z}^B\cup V_{y,z}^{AB}=\emptyset$, $|V_{y,z}^A|=1$, $|B|\geq 1$ and $E(V_\emptyset,V_x\cup V_{y,z})$ induces a complete bipartite graph,
\item $V_{y,z}^A\cup V_{y,z}^{AB}=\emptyset$, $|V_{y,z}^B|=1$, $|B|=1$ and $E(V_\emptyset,A\cup V_{y,z})$ induces a complete bipartite graph,
\item $V_{y,z}^A\cup V_{y,z}^{AB}=\emptyset$, $|V_{y,z}^B|=1$, $|B|\geq 1$ and $E(V_\emptyset,V_x\cup V_{y,z})$ induces a complete bipartite graph,
\item $V_{y,z}^{AB}=\emptyset$, $|V_{y,z}^A|=|V_{y,z}^B|=1$, $|B|\geq 1$ and $E(V_\emptyset, V_x\cup V_{y,z})$ induces a complete bipartite graph,
\item $V_{y,z}^A\cup V_{y,z}^{B}=\emptyset$, $|V_{y,z}^{AB}|\geq 1$, $|B|= 1$ and $E(V_\emptyset, A)$ induces a complete bipartite graph,
\item $V_{y,z}^A\cup V_{y,z}^{B}=\emptyset$, $|V_{y,z}^{AB}|\geq 1$, $|V_{y,z}^B|=1$, $|B|= 1$ and $E(V_\emptyset, A\cup V_{y,z}^B)$ induces a complete bipartite graph.
\end{enumerate}

With a similar argument as in Claim \ref{clm:VxTnV0K2T}, we obtain that $V_\emptyset$ is either trivial or $K_2$.
In cases (1), (4), (7) and (8), $V_\emptyset$ cannot be $T_n$ with $n\geq 2$, since taking $w,w'\in V_\emptyset$, $u\in A$, $u'\in B$, then $G[\{ w, w', u, u', x, y\}]\simeq \Gac$.
On the other hand, since $\omega(G)=3$, then $V_\emptyset$ is not isomorphic to $K_2$ in the cases (2), (3), (5) and (6).
It is not difficult to see that in each case $G$ is isomorphic to an induced subgraph of a graph in $\mathcal F_1^1$.

\subsection{Case when $V_x$ induces $K_1+K_2$ or $2K_2$.}
Through this case we assume that $V_x=\{u_1, u_2,u_3,u_4\}$ such that $u_1 u_2, u_3 u_4\in E(G)$.
That is, $G[\{u_1,u_2,u_3,u_4\}]\simeq 2K_2$.
Let $A=\{u_1,u_2\}$ and $B=\{u_3,u_4\}$.
The following discussion also applies when one of the vertex set $A$ or $B$ has cardinality 1.

\begin{Claim}\label{clm:K2K1vVX}
If $v\in V_{y,z}$, then $E(v,V_x)\neq \emptyset$.
Moreover, if $v\in V_{y,z}$, then $v$ is adjacent with each vertex in one of the following sets $A$, $B$ or $V_x$.
\end{Claim}
\begin{proof}
Suppose there is no edge joining $v$ and a vertex in $V_x$.
Since $G[\{ u_1,u_2,u_3,x,y,v\}]\simeq \Gag$ is a forbidden induced subgraph, then a contradiction is obtained.
Then $v$ is adjacent with some vertex in $V_x$.

Suppose $v$ is adjacent with one of the vertices in $A$.
We will prove that $v$ cannot be adjacent only with one vertex of $A$, say $u_1$.
Thus suppose $v$ is adjacent with $u_1$, and $v$ is not adjacent with $u_2,u_3$ and $u_4$.
Since $G[\{ u_1, u_2, u_3, v, x, y\}]\simeq \Gal$, then $v$ is adjacent with $u_2$ or a vertex in $B$, say $u_3$.
If $v$ is adjacent with $u_2$, then we are done.
So we assume $v$ is adjacent with $u_1$, but $v$ is not adjacent with $u_2$.
This is not possible because $G[\{u_1, u_2, u_3, v, y, z\}]\simeq\Gak$.
Therefore, either $v$ is adjacent with both $u_1$ and $u_2$, or $v$ is not adjacent with neither of $u_1$ and $u_2$
\end{proof}

\begin{Claim}\label{clm:K2K1VYZVX}
If $|V_{y,z}|\geq2$, then each vertex in $V_{y,z}$ is adjacent with each vertex of only one of the following sets $A$, $B$ or $V_x$.
\end{Claim}
\begin{proof}
Let $v,v'\in V_{y,z}$.
Consider the following cases:
\begin{enumerate}[(a)]
\item $v$ is adjacent with each vertex in $A$ and $v'$ is adjacent with each vertex in $B$,
\item $v$ is adjacent with each vertex in $V_x$ and $v'$ is adjacent with each vertex in $B$, and 
\item $v$ is adjacent with each vertex in $V_x$ and $v'$ is adjacent with each vertex in $A$.
\end{enumerate}
Case (a) is impossible because $G[\{ u_1, v, y, v', u_3\}]\simeq P_5$, which is forbidden.
On the other hand, cases (b) and (c) are not allowed because $G[\{a,b,c, v,v',u_1,u_3\}]\simeq\Gbk$; which is forbidden.
Thus, the result follows.
\end{proof}
Thus $E(V_x,V_{y,z})$ satisfies only one of the following three cases:
\begin{enumerate}
\item $E(V_{y,z}, A)$ induces a complete bipartite graph,
\item $E(V_{y,z}, B)$ induces a complete bipartite graph, or
\item $E(V_{y,z}, V_x)$ induces a complete bipartite graph.
\end{enumerate}

Now we describe $V_\emptyset$, the set of vertices not adjacent with any vertex in $\{a, b, c\}$.
Let $w\in V_\emptyset$.
The vertex $w$ is adjacent with a vertex in $V_x \cup V_{y,z}$, because otherwise the shortest path from $w$ to $\{x,y\}$ would contains the graph $\Gaa$ as induced subgraph.

\begin{Claim}\label{clm:K1K2V0Vx}
If $w\in V_\emptyset$ is adjacent with a vertex in $V_x$, then $w$ is adjacent with each vertex in $V_x$.
\end{Claim}
\begin{proof}
Suppose $w$ is adjacent with $u_1$ and not with $u_3$.
Since $G[W\cup \{ w, u_1, u_3\}]\simeq \Gbf$ is forbidden, $w$ also must be adjacent with $u_3$.
In a similar way we get the opposite case and it turns out the result.
\end{proof}

Consider Cases (1) and (2).
The following arguments works on both cases.
First note that if there exists $w\in V_\emptyset$ adjacent with $v\in V_{y,z}$, then $w$ is adjacent with each vertex in $V_x$.
The reason is the following.
Suppose $w$ is not adjacent with any vertex in $V_x$.
Since $v$ is not adjacent with a vertex in $V_x$, say $u$, and the vertices $w$ and $u$ are not adjacent, then we have that $G[\{w,v,y,x,u\}]\simeq \Gaa$; which is a contradiction.
Then $w$ and $u$ are adjacent.
And by Claim \ref{clm:K1K2V0Vx}, the edge set $E(w,V_x)$ induces a complete bipartite graph.
On the other hand, we can prove in a similar way that if there exists $w\in V_\emptyset$ adjacent with each vertex in $V_x$, then $w$ is adjacent with each vertex in $V_{y,z}$.
Therefore, each vertex in $V_\emptyset$ is adjacent with each vertex in $V_x\cup V_{y,z}$.
Furthermore, the set $V_\emptyset$ is a stable set.
It is because if $w,w'\in V_\emptyset$ were adjacent, then by taking $u\in V_x$ and $v\in V_{y,z}$ such that $u$ and $v$ are adjacent, the induced subgraph $G[\{w,w',u, v, x, y\}]$ would be isomorphic to $\Gau$; which is impossible.
These cases correspond to a graph isomorphic to an induced subgraph of a graph in $\mathcal F_1^1$.

Now let us consider case (3).
\begin{Claim}\label{claim:K1K2:V0VxVyz}
If $w\in V_\emptyset$ is adjacent with a vertex in $V_{y,z}$, then each vertex in $V_\emptyset$ is adjacent with each vertex in $V_x \cup V_{y,z}$. 
\end{Claim}
\begin{proof}
Let $w\in V_\emptyset$ and $v\in V_{y,z}$ such that $wv\in E(G)$.
Suppose $w$ is not adjacent with any vertex in $V_x$.
Take $u_2,u_3\in V_x$ such that $u_2 u_3\in E(G)$.
Since the induced subgraph $G[\{ a, b, c, w, u_2, u_3, v\}]$ is isomorphic to $\Gbd$, then we get a contradiction; and $w$ is adjacent with a vertex in $V_x$.
Thus by Claim \ref{clm:K1K2V0Vx}, $w$ is adjacent with each vertex in $V_x$.

Suppose that there exists $v'\in V_{y,z}$ such that $w$ is not adjacent with $v'$.
Since the vertex set $\{ w, u_1, x, y, v, v'\}$ would induce $\Gap$, this does not occur.
Therefore, $w$ is adjacent with each vertex in $V_x \cup V_{y,z}$.

Suppose there is another vertex $w'\in V_\emptyset$.
By the above argument, if $w'$ is adjacent with a vertex in $V_{y,z}$, then it must be adjacent with each vertex in $V_x \cup V_{y,z}$.
And we are done.
On the other hand, by Claim \ref {clm:K1K2V0Vx}, if $w'$ is adjacent with a vertex in $V_x$, then $w'$ must be adjacent with each vertex in $V_x$.
So suppose $w'$ is adjacent with each vertex in $V_x$, but not adjacent with each vertex in $V_{y,z}$.
Then there are two possibilities: either $ww'\notin E(G)$ or $ww'\in E(G)$.
In the first case $G[\{ x, y, u_1, v, w, w'\}]\simeq \Gal$ and in the second  case $G[\{ x, y, v, w, w'\}]\simeq \Gaa$.
Since both graphs are forbidden, then $w'$ is adjacent with $v$.
And therefore $w'$ is adjacent with each vertex in $V_x\cup V_{y,z}$.
\end{proof}	
 
By Claims \ref{clm:K1K2V0Vx} and \ref{claim:K1K2:V0VxVyz}, we obtain that there are two possible cases: either each vertex in $V_\emptyset$ is adjacent with each vertex in $V_x \cup V_{y,z}$, or each vertex in $V_\emptyset$ is adjacent only with each vertex in $V_x$.
Consider first case. 
If there exists a vertex $w\in V_\emptyset$, then $w$ adjacent with the vertices $v\in V_{y,z}$, and $u_1, u_2 \in V_x$.
Thus $G[\{w,v, u_1,u_2\}]$ is isomorphic to $K_4$ that is not allowed, then $V_\emptyset$ is empty.
Now consider second case.
Let $w\in V_\emptyset$ and $v\in V_{y,z}$.
Thus $w$ is adjacent with $u_1,u_2$ and $u_3$.
Since $G[\{w,u_1,u_2,u_3,v,x,y\}]$ is isomorphic to $\Gbk$, we get a contradiction.
Then $V_\emptyset=\emptyset$. 
The graph in this case is isomorphic to an induced subgraph of a graph in $\mathcal F_1^1$.

\subsection{Cases when $G[V_a \cup V_b \cup V_c]=V_x\vee(V_y+V_z)$ is one of the following graphs: $K_1\vee 2K_1$, $K_1\vee (K_1+K_2)$, $K_1\vee 2K_2$, or $K_2\vee 2K_1$.}
For the sake of clarity, we suppose $E(V_a,V_b)$ and $E(V_a,V_c)$ induce a complete bipartite graph, and $E(V_b,V_c)$ is empty.
Now we are going to obtain some claims that describe the edge sets joining $V_x$ and $V_{y,z}$.

\begin{Claim}\label{lem:unstabLLLG1}
Let $x,y\in \{a,b,c\}$.
If $E(V_x,V_y)$ is not empty, then $E(V_x,V_{xy})$ and $E(V_y,V_{xy})$ are empty.
\end{Claim}
\begin{proof}
Let $v_x\in V_x$, $v_y\in V_y$ and $v_{xy}\in V_{x,y}$.
Suppose $v_{xy}$ is adjacent with both $v_x$ and $v_y$. 
Then $G[\{a,b,c,v_{xy}, v_x, v_y\}]\simeq \Gay$; which is a contradiction.
Now suppose $v_{xy}$ is adjacent with $v_x$ and not with $v_y$.
In this case $G[\{a,b,c,v_{xy}, v_x, v_y\}]\simeq \Gar$; which is impossible.
And therefore result turns out.
\end{proof}

Claim \ref{lem:unstabLLLG1} implies that $E(V_{a,b},V_a\cup V_b)=\emptyset$ and $E(V_{a,c},V_a\cup V_c)=\emptyset$.

\begin{Claim}\label{lem:unstabLLLG2}
Let $x,y\in \{a,b,c\}$.
If $V_{x,y}\neq\emptyset$ and $E(V_x,V_y)=\emptyset$, then each edge set $E(V_x,V_{x,y})$ and $E(V_y,V_{x,y})$ induces a complete bipartite graph.
\end{Claim}
\begin{proof}
Let $v_x\in V_x$, $v_y\in V_y$ and $v_{xy}\in V_{x,y}$.
Suppose $v_{xy}$ is not adjacent with both $v_x$ and $v_y$. 
Then $G[\{a,b,c,v_{xy}, v_x, v_y\}]\simeq \Gaf$; which is a contradiction.
Finally, suppose $v_{xy}$ is adjacent with $v_x$ and not with $v_y$.
Since $G[\{a,b,c,v_{x,y}, v_x, v_y\}]\simeq \Gao$ is forbidden, then we get a contradiction.
And the result turns out.
\end{proof}

Claim \ref{lem:unstabLLLG2} implies that each vertex in $V_{b,c}$ is adjacent with each vertex in $V_b\cup V_c$.

\begin{Claim}\label{lem:unstabLLLG3}
If $V_{b,c}\neq\emptyset$, then $E(V_a,V_{b,c})$ is empty.
\end{Claim}
\begin{proof}
Let $v_a\in V_a$, $v_b\in V_b$ and $v_{bc}\in V_{b,c}$.
Suppose $v_{bc}$ is adjacent with $v_a$. 
Then $G[\{a,b,c,v_a, v_b, v_{bc}\}]$ is isomorphic to $\Gaz$ that is forbidden.
Therefore, there is no edge joining a vertex in $V_a$ with a vertex in $V_{b,c}$.
\end{proof}

\begin{Claim}\label{lem:unstabLLLG4}
Let $x\in \{b,c\}$. 
If $V_{a,x}\neq\emptyset$, then $E(V_{a,x},V_x)$ induces a complete bipartite graph.
\end{Claim}
\begin{proof}
Let $v_b\in V_b$, $v_c\in V_c$ and $v_{ax}\in V_{a,x}$.
Suppose $v_{ax}$ is not adjacent with $v_x$.
Since $G[\{a,b,c,v_{ax}, v_b, v_c\}]$ is isomorphic to $\Gak$ that is forbidden, then each vertex in $V_x$ is adjacent with each vertex in $V_{a,x}$.
\end{proof}

\begin{Claim}
The edge set $E(V_{a,b},V_{a,c})$ induces a complete bipartite graph.
\end{Claim}
\begin{proof}
Let $v_a\in V_a$, $v_b\in V_b$, $v_{ab}\in V_{a,b}$, and $v_{ac}\in V_{a,c}$. 
We know that $v_a$ is not adjacent with both $v_{ab}$ and $v_{ac}$, and $v_b$ is adjacent with $v_{ac}$, but $v_b$ is not adjacent with $v_{ab}$.
Suppose $v_{ab}$ and $v_{ac}$ are not adjacent.
Then $G[\{ b,c, v_a, v_b, v_{ab}, v_{ac},\}]\simeq \Gaj$; which is impossible.
And therefore, each vertex in $V_{a,b}$ is adjacent with each vertex in $V_{a,c}$.
\end{proof}

\begin{Claim}
Let $x\in\{b,c\}$
The edge set $E(V_{a,x},V_{b,c})$ induces a complete bipartite graph.
\end{Claim}
\begin{proof}
Let $v_b\in V_b$, $v_c\in V_c$, $v_{ax}\in V_{a,x}$, and $v_{bc}\in V_{b,c}$. 
We know that $v_b v_{bc},v_{bc} v_c, v_c v_{ab}, v_b v_{ac}\in E(G)$ and $v_x v_{ax}, v_b v_c\notin E(G)$.
Suppose $v_{ax}$ and $v_{bc}$ are not adjacent.
Then $G[\{ b,c, v_b, v_c, v_{ax}, v_{bc},\}]\simeq \Gao$; which is impossible.
Therefore each vertex in $V_{a,x}$ is adjacent with each vertex in $V_{b,c}$.
\end{proof}

Now we analyze $V_\emptyset$.

\begin{Claim}\label{lem:unstabLLLG5}
If $w\in V_\emptyset$, then $w$ cannot be adjacent with any vertex in $V_a\cup V_b\cup V_c$.
\end{Claim}
\begin{proof}
Let $x\in \{b,c\}$, $y\in\{b,c\}-x$, $v_a\in V_a$ and $v_x\in V_x$.
Suppose $w\in V_\emptyset$ such that $E(w,V_a\cup V_b\cup V_c)\neq\emptyset$.
Consider the following cases:
\begin{enumerate}[(a)]
\item $w$ is adjacent only with $v_a$,
\item $w$ is adjacent only with $v_x$, or
\item $w$ is adjacent with both $v_a$ and $v_x$.
\end{enumerate}
Cases (a) and (b) are impossible, because $G$ would have $\Gaa$ as induced subgraph obtained by $G[\{w,v_a,v_x,x,y\}]$ and $G[\{w,v_x,v_a,a,y\}]$, respectively.
Finally in case (c), the induced subgraph $G[\{ a,b,c,w, v_a, v_x\}]$ is isomorphic to $\Gav$; which is impossible.
\end{proof}

\begin{Claim}\label{lem:unstabLLLG6}
There exists no vertex in $V_\emptyset$ adjacent with a vertex in $V_{b,c}$.
\end{Claim}
\begin{proof}
Let $w\in V_\emptyset$, $v_a\in V_a$, $v_b\in V_b$, $v_c\in V_c$, and $v_{bc}\in V_{b,c}$.
Suppose $w$ is adjacent with $v_{bc}$.
A contradiction is obtained since $G[\{w, v_a, v_b, v_c, v_{bc}, w\}]\simeq \Gal$.
Thus $w$ is not adjacent with any vertex in $V_{b,c}$
\end{proof}

\begin{Claim}\label{lem:unstabLLLG7}
There exists no vertex in $V_\emptyset$ adjacent with a vertex in $V_{a,b}\cup V_{a,c}$.
\end{Claim}
\begin{proof}
Let $x\in \{b,c\}$, $y\in\{b,c\}-x$, $v_a\in V_a$, $v_b\in V_b$ and $v_{ax}\in V_{a,x}$.
Suppose $w$ is adjacent with $v_{ax}$.
Since the induced subgraph $G[\{w, v_{ax}, v_a, v_b, v_c\}]$ is isomorphic to $\Gaf$, then we get a contradiction.
And then the result follows.
\end{proof}

Claims \ref{lem:unstabLLLG5}, \ref{lem:unstabLLLG6} and \ref{lem:unstabLLLG7} imply that no vertex in $V_\emptyset$ 
is adjacent with a vertex in $G\setminus V_\emptyset$, which implies that $V_\emptyset=\emptyset$.
Thus the graph is isomorphic to an induced subgraph of a graph in $\mathcal F_1^1$.

\subsection{Case $G[V_a \cup V_b \cup V_c]=K_{1,1,1}$, where each vertex set $V_x=K_1$.}
Let $V_a=\{ v_a\}$, $V_b=\{ v_b\}$, and $V_c=\{ v_c\}$. 
By Claim \ref{lem:unstabLLLG1}, the edge sets $E(V_{x,y},V_x)$ and $E(V_{x,y},V_y)$ are empty for $x,y\in \{a,b,c\}$.
By Claim \ref{lem:unstabLLLG3}, the edge set $E(V_{x,y}, V_z)$ is empty for $x,y,z\in \{a,b,c\}$.
Now let $v_{xy}\in V_{x,y}$.
Since $G[\{ v_y, v_z,z,x, v_{x,y}\}]$ is isomorphic to $\Gaa$ which is a forbidden, then $V_{xy}$ is empty for each pair $x,y\in \{a,b,c\}$.
On the other hand, by Claim \ref{lem:unstabLLLG5} the edge set $V_\emptyset$ is empty.
This graph is isomorphic to $G_1$, see Figure \ref{figure:omega3}.i.

\subsection{Case $V_z=\emptyset$ and $G[V_x \cup V_y]=V_x+V_y$, where $V_x=K_m$, $V_y=K_n$ and $m,n\in\{1,2\}$.}
Without loss of generality, suppose $V_c=\emptyset$, and $E(V_a,V_b)$ is empty.
By Claim \ref{lem:unstabLLLG2}, each vertex in $V_{a,b}$ is adjacent with each vertex in $V_a\cup V_b$.
\begin{Claim}
Let $x\in\{a,b\}$.
If $V_{x,c}\neq \emptyset$, then $E(V_{x},V_{x,c})=\emptyset$.
\end{Claim}
\begin{proof}
Let $y\in\{a,b\}-x$, $v_{xc}\in V_{x,c}$, $v_a\in V_a$ and $v_b\in V_b$.
Suppose $v_{x}$ and $v_{xc}$ are adjacent.
There are two possible cases: either $v_y$ and $v_{xc}$ are adjacent or not.
Since in the first case $G[\{a,b,c,v_a,v_b,v_{xc}\}]\simeq\Gar$ and in the second case $G[\{ v_a, v_b, y, v_{xc},c\}]$ is isomorphic to $\Gaa$, then we get a contradiction.
Thus $v_{x}$ and $v_{xc}$ are not adjacent.
\end{proof}

By Claim \ref{lem:unstabLLLG4}, each vertex in $V_{a,c}$ is adjacent with each vertex in $V_b$, and each vertex in $V_{b,c}$ is adjacent with each vertex in $V_a$.

\begin{Claim}
Each vertex in $V_{a,b}$ is adjacent with each vertex in $V_{a,c}\cup V_{b,c}$.
\end{Claim}
\begin{proof}
Let $v_{ab}\in V_{a,b}$, $v_a\in V_a$ and $v_b\in V_b$.
Suppose there exists $v_{xc}\in V_{xc}$ with $x\in\{a,b\}$ such that $v_{ab}v_{xc}\notin E(G)$.
Since $G[\{ v_a, v_b, v_{ab}, v_{xc},c\}]$ is isomorphic to $\Gaa$, then we get a contradiction; and the vertices $v_{ab}$ and $v_{xc}$ are adjacent.
And the result turns out.
\end{proof}

\begin{Claim}
Each vertex in $V_{a,c}$ is adjacent with each vertex in $V_{b,c}$.
\end{Claim}
\begin{proof}
Suppose there are $v_{ac}\in V_{a,c}$ and $v_{bc}\in V_{b,c}$ such that $v_{bc}v_{ac}\notin E(G)$.
Let $v_a\in V_a$ and $v_b\in V_b$
Since $G[\{ v_b, v_{ac}, c, v_{bc}, v_a\}]\simeq P_5$, then we get a contradiction.
\end{proof}

Now we describe the vertex set $V_\emptyset$, that is, the set of vertices that are not adjacent with any vertex in $\{a,b,c\}$.

\begin{Claim}
If $w\in V_\emptyset$, then $w$ cannot be adjacent with any vertex in $V_a\cup V_b$.
\end{Claim}
\begin{proof}
Let $v_a\in V_a$ and $v_b\in V_b$.
Suppose $w\in V_\emptyset$ such that $E(w,V_a\cup V_b)\neq\emptyset$.
Consider the following cases:
\begin{enumerate}[(a)]
\item $w$ is adjacent only with $v_a$, or
\item $w$ is adjacent with both $v_a$ and $v_b$.
\end{enumerate}
Cases (a) and (b) are impossible, because $G$ would have $\Gaa$ as induced subgraph obtained by $G[\{w,v_b,b,a,v_a\}]$ and $G[\{v_a,w,v_b,b,c\}]$, respectively.
Thus $w$ is not adjacent with any vertex in $V_{a}\cup V_{b}$.
\end{proof}

\begin{Claim}
There is no vertex $w\in V_\emptyset$ adjacent with a vertex in $V_{a,b}$.
\end{Claim}
\begin{proof}
Suppose $w\in V_\emptyset$ is adjacent with $v_{ab}\in V_{a,b}$.
Let $v_a\in V_a$ and $v_b\in V_b$.
Since $G[\{ v_a, v_b, v_{a,b}, w, a,c\}]$ is isomorphic to $\Gac$, then we get a contradiction and $w$ and $v_{ab}$ are not adjacent.
\end{proof}

\begin{Claim}
There is no vertex $w\in V_\emptyset$ adjacent with a vertex in $V_{a,c}\cup V_{b,c}$.
\end{Claim}
\begin{proof}
Let $x\in\{a,b\}$, $v_{xc}\in V_{x,c}$, $y\in\{a,b\}-x$ and $v_y\in V_y$.
Suppose $w\in V_\emptyset$ is adjacent with $v_{xc}$.
Since $G[\{ v_y, y, c,v_{xc}, w\}]$ is isomorphic to $\Gaa$, then we get a contradiction and $w$ is not adjacent with $v_{xc}$.
\end{proof}

Thus there is no edge between $W$ and $G\setminus W$, and therefore $W$ is empty.
Therefore, the graph is isomorphic to an induced subgraph of a graph in $\mathcal F_1^1$.

\subsection{Case $V_z=\emptyset$ and $G[V_x \cup V_y]=V_x\vee V_y$, where $V_x=K_1$, $V_y=K_m$ and $m\in\{1,2\}$.}
Without loss of generality, suppose $V_c=\emptyset$, and $E(V_a,V_b)$ is complete.
By Claim \ref{lem:unstabLLLG1}, $E(V_{a,b},V_x)=\emptyset$ for $x\in\{a,b\}$.

\begin{Claim}
Let $x\in\{a,b\}$.
Either $E(V_{x,c},V_a)$ induces a complete bipartite graph and $E(V_{x,c},V_b)=\emptyset$, or $E(V_{x,c},V_a)=\emptyset$ and $E(V_{x,c},V_b)$ induces a complete bipartite graph.
\end{Claim}
\begin{proof}
Let $v_a \in V_a$, $v_b\in V_b$, $v_{xc}\in V_{x,c}$ and $y\in\{a,b\}-x$.
First note that $v_{xc}$ cannot be not adjacent with $v_a$ and $v_b$ at the same time, because otherwise $G[\{v_a,v_b,y,c,v_{xc}\}]$ would be isomorphic to $\Gaa$.
Also $v_{xc}$ cannot be adjacent with $v_a$ and $v_b$ at the same time, because otherwise $G[\{a,b,c,v_a,v_b,v_{xc}\}]$ would be isomorphic to $\Gaz$.
Thus $v_{xc}$ is adjacent only with one vertex either $v_a$ or $v_b$.

Suppose $v_{xc}$ is adjacent with $v_y$.
Let $v_y' \in V_y$ and $v_{xc}'\in V_{x,c}$.
If $v_{xc}$ is not adjacent with $v_y'$, then $G[\{v_y,v_y',x,c,v_x,v_{xc}\}]$ is isomorphic to $\Gav$.
Thus $E(V_y,v_{xc})$ induces a complete bipartite graphs.
On the other hand, if $v_y$ and $v_{xc}'$ are not adjacent, then $G[\{v_a,v_b,v_{xc},c,v_{xc}'\}]$ is isomorphic to $\Gaa$.
Then $E(V_{x,c},V_y)$ induces a complete bipartite graph.

Suppose $v_{xc}$ is adjacent with $v_x$.
Let $v_x' \in V_x$ and $v_{xc}'\in V_{x,c}$.
If $v_x'$ and $v_{xc}$ are not adjacent, then $G[\{v_a,v_b,v_x',x,c,v_{xc}\}]$ is isomorphic to $\Gaw$.
Thus $E(V_x,v_{xc})$ induces a complete bipartite graph.
On the other hand, if $v_{xc}'$ and $v_x$ are not adjacent, then $G[\{v_y,v_x,v_{xc},c,v_{xc}'\}]$ is isomorphic to $\Gaa$.
Thus $E(V_{x,c},V_x)$ induces a complete bipartite graph.
\end{proof}

\begin{Claim}
If $V_{a,c}$ and $V_{b,c}$ are not empty, then each vertex in $V_{a,c}\cup V_{b,c}$ is adjacent with each vertex in either $V_a$ or $V_b$ .
\end{Claim}
\begin{proof}
Let $v_{ac}\in V_{a,c}$, $v_{bc}\in V_{b,c}$, $v_{a}\in V_{a}$ and $v_b\in V_b$.

Suppose  $v_{ac}$ is adjacent with $v_a$, and $v_{bc}$ is adjacent with $v_b$.
Then there are two cases: either $v_{ac}$ and $v_{bc}$ are adjacent or not.
In the first case $G[\{a,b,c,v_b,v_{ac},v_{bc}\}]$ is isomorphic to $\Gay$; then this case is impossible.
And in the second case $G[\{a,b,v_a,v_b,v_{ac},v_{bc}\}]$ is isomorphic to $\Gav$; then this case is not possible.

Suppose  $v_{ac}$ is adjacent with $v_b$ and $v_{bc}$ is adjacent with $v_a$.
Then there are two cases: either $v_{ac}$ and $v_{bc}$ are adjacent or not.
In the first case $G[\{a,b,c,v_a,v_b,v_{ac},v_{bc}\}]$ is isomorphic to $\Gbo$; then this case is impossible.
And in the second case $G[\{a,b,c,v_b,v_{ac},v_{bc}\}]$ is isomorphic to $\Gar$; then this case is not possible.

Thus $v_{ac}$ and $v_{bc}$ are adjacent with the same vertex: either $v_a$ or $v_b$.
And the result follows.
\end{proof}

\begin{Claim}
If $V_{a,c}$ and $V_{b,c}$ are not empty, then the set $E(V_{ac},V_{bc})$ induces a complete bipartite graph.
\end{Claim}
\begin{proof}
Let $x\in\{a,b\}$ and $y\in\{a,b\}-x$, $v_{ac}\in V_{a,c}$, $v_{bc}\in V_{b,c}$, $v_{a}\in V_{a}$ and $v_b\in V_b$.
Suppose $v_{xc}$ and $V_{yc}$ are not adjacent, and $v_{xc}$ and $v_{yc}$ are adjacent with $v_x$.
Since $G[\{ x,y, v_x, v_y, v_{x,c}, v_{y,c}\}]\simeq \Gap$, we get a contradiction.
And then $E(V_{ac},V_{bc})$ induces a complete bipartite graph.
\end{proof}

\begin{Claim}
Let $x\in \{a,b\}$.
If $V_{x,c}\neq \emptyset$, then $E(V_{x,c},V_{a,b})$ induces a complete bipartite graph.
\end{Claim}
\begin{proof}
Let $y\in\{a,b\}-x$, $v_{xc}\in V_{x,c}$, $v_{ab}\in V_{a,b}$, $v_{a}\in V_{a}$ and $v_b\in V_b$.
Suppose $v_{xc}$ and $v_{ab}$ are not adjacent.
There are two cases: either $v_{xc}$ is adjacent with $v_x$ or $v_{xc}$ is adjacent with $v_y$. 
If $v_{xc}$ is adjacent with $v_x$, then the induced subgraph $G[\{v_{ab},y,v_y,v_x,v_{xc}\}]$ is isomorphic to $\Gaa$; which is a contradiction.
Then this case is impossible.
On the other hand, if $v_{xc}$ is adjacent with $v_y$, then the induced subgraph $G[\{v_{ac},v_a,v_b,a,b,v_{xc}\}]$ is isomorphic to $\Gaj$; which is a contradiction.
Thus this case is also impossible, and therefore $E(V_{x,c},V_{a,b}$ induces a complete bipartite graph.
\end{proof}

Now let us describe $V_\emptyset$.
By Claim \ref{lem:unstabLLLG5}, there exists no vertex in $V_\emptyset$ adjacent with a vertex in $V_a\cup V_b$.

\begin{Claim}
There is no vertex in $V_\emptyset$ adjacent with a vertex in $V_{a,b}$.
\end{Claim}
\begin{proof}
Let $v_{ab}\in V_{a,b}$ and $w\in V_\emptyset$.
Suppose $v_{ab}$ and $w$ are adjacent.
Since $G[\{w, v_{a,b}, a, v_a, v_b \}]$ is isomorphic to $\Gaa$, then we get a contradiction.
And therefore $v_{ab}$ and $w$ are not adjacent.
\end{proof}

\begin{Claim}
Let $x\in\{a,b\}$.
There is no vertex in $V_\emptyset$ adjacent with a vertex in $V_{x,c}$.
\end{Claim}
\begin{proof}
Let $y\in\{a,b\}-x$, $v_{xc}\in V_{x,c}$, $v_{ab}\in V_{a,b}$, $v_{a}\in V_{a}$ and $v_b\in V_b$.
Suppose the vertex $w\in V_\emptyset$ is adjacent with $v_{xc}$.
There are two cases: either $v_{xc}$ is adjacent with $x$ or $v_{xc}$ is adjacent with $y$.
First case is impossible since $G[\{w,v_{xc}, v_x, v_y, y\}]$ is isomorphic to $\Gaa$; which is forbidden.
And second case cannot occur because $G[\{ w,y,v_{xc},c, v_a, v_b\}]\simeq \Gaj$.
Then it follows that $w$ is adjacent with no vertex in $V_{x,c}$.
\end{proof}

Thus by previous Claims, the vertex set $V_\emptyset$ is empty, because there is no vertex in $V_\emptyset$ adjacent with a vertex in $G\setminus V_\emptyset$.
Then $G$ is isomorphic to an induced subgraph of a graph in $\mathcal F_1^1$.

\subsection{Case $V_y\cup V_z=\emptyset$ and $V_x$ is $K_{1}$ or $K_2$.} 
Without loss of generality, suppose $V_b=V_c=\emptyset$ and $V_a=\{ u_1, u_2\}$.

\begin{Claim}
Let $x\in \{b,c\}$.
If $V_{a,x}\neq\emptyset$, then either $E(V_a,V_{a,x})$ induces a complete bipartite graph or $E(V_a,V_{a,x})$ is empty.
\end{Claim}
\begin{proof}
Let $u\in V_a$.
Suppose there exist $v_1,v_2 \in V_{a,x}$ such that $uv_1\in E(G)$ and $uv_2\notin E(G)$.
Since the induced subgraph $G[\{a,b,c,u,v_1,v_2\}]$ is isomorphic to $\Gaq$, then we get a contradiction.
Thus this case is not possible and therefore $u$ is adjacent with either each vertex in $V_{a,x}$ or no vertex in $V_{a,x}$.
Now we are going to discard the possibility that $E(u_1,V_{a,x})$  induces a complete bipartite graph and $E(u_2,V_{a,x})=\emptyset$.
Suppose there is a vertex $v\in V_{a,x}$ such that $u_1v\in E(G)$ and $u_2v\notin E(G)$.
Since the induced subgraph $G[\{c,b,v,u_1,u_2\}]$ is isomorphic to $\Gaa$, then we get a contradiction.
Thus either $E(V_a,V_{a,x})$ induces a complete bipartite graph or $E(V_a,V_{a,x})$ is empty.
\end{proof}

\begin{Claim}\label{claim:VbcupVcemptyAndVanotemptyEVaVbc_1}
Let $u\in V_a$.
If $V_{b,c}\neq\emptyset$, then $E(u,V_{b,c})$ satisfies only one of the following:
\begin{itemize}
\item it induces a complete bipartite graph,
\item it is an empty edge set, or
\item it induces a complete bipartite graph minus an edge.
\end{itemize}
\end{Claim}
\begin{proof}
Suppose there exist $v_1,v_2,v_3 \in V_{b,c}$ such that $uv_1\in E(G)$ and $uv_2, uv_3\notin E(G)$.
Since the induced subgraph $G[\{a,b,u,v_1,v_2,v_3\}]$ is isomorphic to $\Gad$, then we get a contradiction.
Thus this case is not possible and the result follows. 
\end{proof}

\begin{Claim}\label{claim:VbcupVcemptyAndVanotemptyEVaVbc}
If $V_a=\{u_1,u_2\}$ and $V_{b,c}\neq\emptyset$, then $E(V_a,V_{b,c})$ satisfies one of the following:
\begin{itemize}
\item it induces a complete bipartite graph,
\item it is an empty edge set,
\item it induces a complete bipartite graph minus an edge,
\item it induces a perfect matching and $|V_a|=|V_{b,c}|=2$, or
\item it induces a complete bipartite graph minus two edges $u_1v$ and $u_2v$, where $v\in V_{b,c}$.
\end{itemize}
\end{Claim}
\begin{proof}
Since cases where $|V_{b,c}|\leq 2$ can be checked easily with a computer algebra system or with similar arguments to the rest of the proof, then we asume $|V_{b,c}|\geq 3$.
By Claim \ref{claim:VbcupVcemptyAndVanotemptyEVaVbc_1}, we only have to check the possibilities of the edge sets $E(u_1,V_{b,c})$ and $E(u_2,V_{b,c})$.
The possible cases we have to discard are the following:
\begin{itemize}
\item $E(u_1,V_{b,c})=\emptyset$ and $E(u_2,V_{b,c})$ induces a complete bipartite graph,
\item $E(u_1,V_{b,c})=\emptyset$ and $E(u_2,V_{b,c})$ induces a complete bipartite graph minus an edge, and
\item each edge set $E(u_1,V_{b,c})$ and $E(u_2,V_{b,c})$ induces a complete bipartite graph  minus an edge and the two removed edges don't share a common vertex.
\end{itemize}
Let $v_1,v_2,v_3 \in V_{b,c}$.
Suppose we are in the first case.
Thus $u_1$ is adjacent with no vertex in $V_{b,c}$, and $u_2$ is adjacent with each vertex in $V_{b,c}$.
Since the induced subgraph $G[\{a,b,u_1,u_2,v_1,v_2\}]$ is isomorphic to $\Gap$, then we get a contradiction and this case is not possible.
Suppose we are in the second case.
Thus $u_1$ is adjacent with no vertex in $V_{b,c}$, and $u_2$ is adjacent only with each vertex in $V_{b,c}-v_1$.
Since the induced subgraph $G[\{v_1,b,v_2,u_2,u_1\}]$ is isomorphic to $\Gaa$, then we get a contradiction and this case is not possible.
Finally, suppose we are in the third case.
Thus $u_1$ is adjacent with each vertex in $V_{b,c}-v_1$, and $u_2$ is adjacent with each vertex in $V_{b,c}-v_2$.
Since the induced subgraph $G[\{a,u_1,u_2,v_1,v_2,v_3\}]$ is isomorphic to $\Gaf$, then we get a contradiction and this case is not possible.
And the result turns out.
\end{proof}

\begin{Claim}\label{claim:VbcupVcemptyAndVanotempty1}
If $E(V_a,V_{a,b})$ and $E(V_a,V_{a,c})$ are empty, then $E(V_{a,b},V_{a,c})$ induces a complete bipartite graph.
\end{Claim}
\begin{proof}
Let $v_{ab}\in V_{a,b}$ and $v_{ac}\in V_{a,c}$.
Suppose $v_{ab}$ and $v_{ac}$ are not adjacent.
Since the induced subgraph $G[\{a,b,c,u_1,v_{ab},v_{ac}\}]$ is isomorphic to $\Gam$, then we get a contradiction.
And thus each vertex in $V_{a,b}$ is adjacent with each vertex in $V_{a,c}$.
\end{proof}

\begin{Claim}\label{claim:VbcupVcemptyAndVanotempty2}
If each of the edge sets $E(V_a,V_{a,b})$ and $E(V_a,V_{a,c})$ induces a complete bipartite graph, then $|V_a|=|V_{a,b}|=|V_{a,c}|=1$ and $E(V_{a,b},V_{a,c})=\emptyset$.
\end{Claim}
\begin{proof}
First we will prove that $E(V_{a,b},V_{a,c})=\emptyset$.
Let $v_{ab}\in V_{a,b}$ and $v_{ac}\in V_{a,c}$.
Suppose $v_{ab}$ and $v_{ac}$ are adjacent.
Since the induced subgraph $G[\{a,b,c,v_{ab},v_{ac},v_a\}]$ is isomorphic to $\Gba$, then we get a contradiction.
And thus $E(V_{a,b},V_{a,c})=\emptyset$.

Now let $x\in \{b,c\}$.
Suppose $V_{a,x}$ has cardinality at least 2.
Let $y\in\{b,c\}-x$, $v_{ax},v_{ax}'\in V_{a,x}$ and $v\in V_{a,y}$.
Then $v_{ax}v,v_{ax}v\notin E(G)$.
Since the induced subgraph $G[\{a,y,v_{ax},v_{ax}',v_a,v\}]$ is isomorphic to $\Gaq$, then we get a contradiction.
And then $|V_{a,b}|=|V_{a,c}|=1$

Finally suppose that $|V_a|=2$.
Let $v_{ab}\in V_{a,b}$ and $v_{ac}\in V_{a,c}$.
Since $G[\{v_{ab},v_{ac},u_1,u_2,a,x\}]$ is isomorphic to $\Gax$, then we get a contradiction.
Thus $V_a$ has cardinality at most 1.
\end{proof}

\begin{Claim}\label{claim:VbcupVcemptyAndVanotempty3}
Let $x\in \{b,c\}$ and $y\in\{b,c\}-x$.
If $E(V_a,V_{a,x})=\emptyset$ and $E(V_a,V_{a,y})$ induces a complete bipartite graph, then $E(V_{a,b},V_{a,c})$ induces a complete bipartite graph.
\end{Claim}
\begin{proof}
Let $v_{ab}\in V_{a,b}$ and $v_{ac}\in V_{a,c}$.
Suppose $v_{ab}$ and $v_{ac}$ are not adjacent.
Since the induced subgraph $G[\{v_a,v_{ay},y,x,v_{ax}\}]$ is isomorphic to $\Gaa$, then we get a contradiction.
And thus $E(V_{a,b},V_{a,c})$ induces a complete bipartite graph.
\end{proof}

\begin{Claim}\label{claim:VbcupVcemptyAndVanotempty4}
Let $x\in \{b,c\}$.
If $V_{a,x}\neq\emptyset$, $V_{b,c}\neq\emptyset$ and $E(V_a,V_{a,x})=\emptyset$, then only one of the following statements is true:
\begin{itemize}
\item $E(V_{a},V_{b,c})=\emptyset$ and $E(V_{a,x},V_{b,c})$ induces a complete graph, or
\item each edge set $E(V_{a},V_{b,c})$ and $E(V_{a,x},V_{b,c})$ induce a complete bipartite graph.
\end{itemize}
\end{Claim}
\begin{proof}
We will analyze the following four cases:
\begin{itemize}
\item when $E(V_{a},V_{b,c})=\emptyset$,
\item when $E(V_{a},V_{b,c})$ induces a complete bipartite graph,
\item when there exist $u\in V_a$ and $v_1,v_2\in V_{b,c}$ such that $uv_1\in E(G)$ and $uv_2\notin E(G)$, and
\item when there exists $v\in V_{b,c}$ such that $u_1v\in E(G)$ and $u_2v\notin E(G)$, where $u_1,u_2\in V_a$.
\end{itemize}

Consider the first case.
Let $v_{ax}\in V_{a,x}$ and $v_{bc}\in V_{b,c}$.
Suppose $v_{ax}$ and $v_{bc}$ are not adjacent.
Since the induced subgraph $G[\{a,b,c,u_1,v_{ax},v_{bc}\}]$ is isomorphic to $\Gao$, then we get a contradiction.
And thus $E(V_{a,x},V_{b,c})$ induces a complete bipartite graph.

Now consider the second case.
Let $v_{ax}\in V_{a,x}$ and $v_{bc}\in V_{b,c}$.
Suppose $v_{ax}$ and $v_{bc}$ are not adjacent.
Since the induced subgraph $G[\{a,b,c,u_1,v_{ax},v_{bc}\}]$ is isomorphic to $\Gar$, then we get a contradiction.
And thus $E(V_{a,x},V_{b,c})$ induces a complete bipartite graph.

In what follows we will prove that the last two cases are impossible which implies that the rest possibilities of Claim \ref{claim:VbcupVcemptyAndVanotemptyEVaVbc} are impossible.

Now consider the third case.
Let $v_{ax}\in V_{a,x}$.
Suppose there exist $u\in V_a$ and $v_1,v_2\in V_{b,c}$ such that $uv_1\in E(G)$ and $uv_2\notin E(G)$.
There are four possible cases:
\begin{itemize} 
\item $v_{ax}v_1,v_{ax}v_{2}\in E(G)$,
\item $v_{ax}v_1,v_{ax}v_2\notin E(G)$, 
\item $v_{ax}v_1\in E(G)$ and $v_{ax}v_2\notin E(G)$, or
\item $v_{ax}v_1\notin E(G)$ and $v_{ax}v_2\in E(G)$.
\end{itemize}
Since in first case $G[\{a,b,c,u,v_{ax},v_1,v_2\}]\simeq\Gbl$, in second case $G[\{b,c,u,v_{ax},v_1,v_2\}]\simeq\Gak$, in third case $G[\{b,c,u,v_{ax},v_1,v_2\}]\simeq\Gao$ and in fourth case $G[\{u,v_1,v_y,v_2,v_{ax}\}]\simeq\Gaa$, then we get a contradiction.

In the last case, we get a contradiction because otherwise by first case: $v_{a,x}v\notin E(G)$, but by second case: $v_{a,x}v\in E(G)$.
Thus this case is impossible.
\end{proof}

\begin{Claim}\label{claim:VbcupVcemptyAndVanotempty5}
Let $x\in \{b,c\}$.
If $V_{a,x}\neq\emptyset$, $V_{b,c}\neq\emptyset$, and $E(V_a,V_{a,x})$ induces a complete bipartite graph, then the edge sets $E(V_{a},V_{b,c})$ and $E(V_{a,x},V_{b,c})$ induce complete bipartite graphs.
\end{Claim}
\begin{proof}
Let $v_{ax}\in V_{a,x}$, $v\in V_{b,c}$ and $u\in V_{a}$.
Suppose $uv\notin E(G)$.
There are two cases: either $vv_{ax}\in E(G)$ or $vv_{ax}\notin E(G)$.
Since the induced subgraph $G[\{a,b,c,u,v,v_{ax}\}]$ is isomorphic to $\Gay$ and $\Gaw$, in the first case and in the second case, respectively, then we get a contradiction.
Thus $uv\in E(G)$ and therefore the edge set $E(V_{a},V_{b,c})$ induces a complete bipartite graph.

Now suppose $v v_{ax}\notin E(G)$.
By previous result, $uv\in E(G)$.
Since the induced subgraph $G[\{a,b,c,u,v,v_{ax}\}]$ is isomorphic to $\Gaz$, then we get a contradiction and thus $v v_{ax}\in E(G)$.
Therefore $E(V_{a,x},V_{b,c})$ induces a complete bipartite graph.
\end{proof}

By previous Claims we have the following cases:

\begin{enumerate}
\item $E(V_a,V_{a,b}\cup V_{ac}\cup V_{b,c})=\emptyset$, and each edge set $E(V_{a,b},V_{a,c})$ and $E(V_{a,b}\cup V_{a,c},V_{b,c})$ induces a complete bipartite graph,
\item $E(V_a,V_{a,b}\cup V_{ac})=\emptyset$, and each edge set $E(V_{a,b},V_{a,c})$ and $E(V_a\cup V_{a,b}\cup V_{a,c},V_{b,c})$ induces a complete bipartite graph,
\item $E(V_a,V_{a,b})=\emptyset$, and each edge set $E(V_a\cup V_{a,b},V_{a,c})$ and $E(V_a\cup V_{a,b}\cup V_{a,c},V_{b,c})$ induces a complete bipartite graph,
\item $E(V_{a,b},V_{a,c})=\emptyset$, and each edge set $E(V_a,V_{a,b}\cup V_{a,c}\cup V_{b,c})$ and $E(V_{a,b}\cup V_{a,c},V_{b,c})$ induces a complete bipartite graph and $|V_a|=|V_{a,b}|=|V_{a,c}|=1$,
\item $V_{b,c}=\emptyset$, $E(V_a,V_{a,b}\cup V_{ac})=\emptyset$, and $E(V_{a,b},V_{a,c})$ induces a complete bipartite graph,
\item $V_{b,c}=\emptyset$, $E(V_a,V_{a,x})=\emptyset$, and $E(V_{a}\cup V_{a,x},V_{a,y})$ induces a complete bipartite graph, where $x\in\{a,b\}$ and $y\in\{a,b\}-x$, 
\item $V_{b,c}=\emptyset$, $E(V_a,V_{a,b}\cup V_{a,c})$ induce a complete bipartite graph, $E(V_{a,b},V_{a,c})=\emptyset$ and $|V_a|=|V_{a,b}|=|V_{a,c}|=1$,
\item $V_{a,y}=\emptyset$, $E(V_a,V_{a,x}\cup V_{b,c})=\emptyset$, and $E(V_{a,x},V_{b,c})$ induces a complete bipartite graph, where $x\in\{a,b\}$ and $y\in\{a,b\}-x$, 
\item $V_{a,y}=\emptyset$, $E(V_a,V_{a,x})=\emptyset$, and $E(V_a\cup V_{a,x},V_{b,c})$ induces a complete bipartite graph, where $x\in\{a,b\}$ and $y\in\{a,b\}-x$, 
\item $V_{a,y}=\emptyset$, $E(V_a,V_{a,x}\cup V_{b,c})$ and $E(V_{a,x},V_{b,c})$ induce a complete bipartite graph, where $x\in\{a,b\}$ and $y\in\{a,b\}-x$, 
\item $V_{a,b}\cup V_{a,c}=\emptyset$, and $E(V_a, V_{b,c})$ induces a complete bipartite graph,
\item $V_{a,b}\cup V_{a,c}=\emptyset$ and $E(V_a, V_{b,c})=\emptyset$,
\item $V_{a,b}\cup V_{a,c}=\emptyset$, and $E(V_a, V_{b,c})$ induces a complete bipartite graph minus an edge, 
\item $V_{a,b}\cup V_{a,c}=\emptyset$, $|V_{b,c}|\geq 2$ and $E(V_a, V_{b,c})$ induces a complete bipartite graph minus two edges $u_1v$ and $u_2v$, where $v\in V_{b,c}$,
\item $V_{a,b}\cup V_{a,c}=\emptyset$, $|V_a|=|V_{b,c}|=2$ and $E(V_a, V_{b,c})$ induces a perfect matching, 
\item $V_{a,y}\cup V_{b,c}=\emptyset$ and $E(V_a, V_{a,x})$ induces a complete bipartite graph, where $x\in\{a,b\}$ and $y\in\{a,b\}-x$, and
\item $V_{a,y}\cup V_{b,c}=\emptyset$ and $E(V_a, V_{a,x})=\emptyset$, where $x\in\{a,b\}$ and $y\in\{a,b\}-x$.
\end{enumerate}

Now we describe $V_\emptyset$, the set of vertices that are not adjacent with any vertex in $\{a, b, c\}$.
Let $w\in V_\emptyset$.
The vertex $w$ is adjacent with a vertex in $V_a \cup V_{a,b}\cup V_{a,c}\cup V_{b,c}$, because otherwise the shortest path from $w$ to $\{a,b,c\}$ would contains the graph $\Gaa$ as induced subgraph.

\begin{Claim}\label{claim:VbcupVc0AndVanot0:Vaxnot0thenVaCupVax0}
If $V_{a,x}\neq\emptyset$ for some $x\in\{b,c\}$, then $E(V_\emptyset,V_a\cup V_{a,x})=\emptyset$.
\end{Claim}
\begin{proof}
Let $w\in V_\emptyset$, $v_{ax}\in V_{a,x}$ and $u\in V_a$.
Suppose $w$ is adjacent with $u$ or $v_{ax}$.
Then there are three possible cases:
\begin{enumerate}[(a)]
\item $w$ is adjacent only with $u$,
\item $w$ is adjacent only with $v_{ax}$, or
\item $w$ is adjacent with both vertices $u$ and $v_{ax}$.
\end{enumerate}
First consider case (a).
This case has two subcases: either $u$ is adjacent with $v_{ax}$ or not.
If they are adjacent, then $G[\{u,w,v_{ax},b,c\}]$ is isomorphic to $\Gaa$.
Since it is forbidden, then $v_a$ is not adjacent with $v_{ax}$.
On the other hand, if they are not adjacent, then $G[\{a,b,c,w,u,v_{ax}\}]$ is isomorphic to $\Gah$.
Therefore, case (a) is impossible.
Now consider case (b).
There are two possible cases: either $u$ is adjacent with $v_{ax}$ or not.
If they are adjacent, then $G[\{a,b,c,w,u,v_{ax}\}]$ is isomorphic to $\Gao$; which is forbidden.
Thus  $u$ is not adjacent with $v_{ax}$.
But if they are not adjacent, then we have that $G[\{a,b,c,w,u,v_{ax}\}]$ is isomorphic to $\Gak$.
Thus case (b) cannot occur.
Finally, consider case (c).
The subcases are: either $u$ is adjacent to $v_{ax}$ or not.
If they are adjacent, then $G[\{a,b,c,w,u,v_{ax}\}]$ is isomorphic to $\Gaw$; which is forbidden.
Then $u$ is not adjacent with $v_{ax}$
But if they are not adjacent, then the induced subgraph $G[\{u,w,v_{ax},b,c\}]$ is isomorphic to $\Gaa$; which is forbidden.
Thus, we get that case (c) is impossible.
And therefore, $w$ is not adjacent with $u$ or $v_{ax}$.
From which the result follows.
\end{proof}

\begin{Claim}\label{claim:VbcupVc0AndVanot0:Vaxnot0thenVbc0}
Let $x\in \{b,c\}$.
If $V_{a,x}\neq\emptyset$, then $E(V_\emptyset,V_{b,c})=\emptyset$.
Moreover, if $V_{a,x}\neq\emptyset$, then $V_\emptyset=\emptyset$.
\end{Claim}
\begin{proof}
Let $w\in V_\emptyset$, $v_{ax}\in V_{a,x}$, $v_a\in V_a$ and $v_{bc}\in V_{b,c}$.
Suppose $w$ is adjacent with $v_{bc}$.
By Claim \ref{claim:VbcupVc0AndVanot0:Vaxnot0thenVaCupVax0}, we have that $w$ is not adjacent with a vertex in $V_a\cup V_{a,x}$.
By Claims \ref{claim:VbcupVcemptyAndVanotempty4} and \ref{claim:VbcupVcemptyAndVanotempty5}, the vertex $v_{ax}$ is adjacent with $v_{bc}$, and one of the following three cases occur:
\begin{enumerate}[(a)]
\item $v_a$ is adjacent with $v_{ax}$ and $v_{bc}$,
\item $v_a$ is adjacent only with $v_{bc}$, and
\item $v_a$ is not adjacent with both $v_{ax}$ and $v_{bc}$.
\end{enumerate}
Case (a) is not possible because the induced subgraph $G[\{v_a, v_{ax}, b, c, v_{bc}, w\}]$ would be isomorphic to $\Gam$; which is impossible.
In case (b), we have that $G[\{v_a, v_{ax}, b, c, v_{bc}, w\}]$ is isomorphic to $\Gae$.
Since it is forbidden, then this case is not possible.
Finally, case (c) is not possible since $G[\{ v_a, a, x, v_{bc}, w\}]$ is isomorphic to $\Gaa$ that is forbidden.
Thus $w$ is not adjacent with $v_{bc}$.
And there is no vertex in $V_\emptyset$ adjacent with a vertex in $G\setminus V_\emptyset$.
Therefore $V_\emptyset$ is empty.
\end{proof}

Thus by previous Claim, in cases (1) to (10), (16) and (17) the vertex set $V_\emptyset$ is empty.
The cases when $V_{a,b}\cup V_{a,c}=\emptyset$ (cases (4) and (7)) correspond to a graph isomorphic to an induced subgraph of $\mathcal F_1^2$.
The rest of these cases correspond to a graph isomorphic to an induced subgraph of $\mathcal F_1^1$.

\begin{Claim}\label{claim:VbcupVc0AndVanot0:Vbcnot0EV_aV_bcComplere1}
Let $w_a\in V_\emptyset$ such that $w_a$ is adjacent with $u\in V_a$, and $w_a$ is not adjacent with a vertex in $V_{b,c}$.
Let $w_{bc}\in V_\emptyset$ such that $w_{bc}$ is adjacent with $v\in V_{b,c}$ and $w_{bc}$ is not adjacent with a vertex in $V_{a}$.
Let $w\in V_\emptyset$ such that $w$ is adjacent with $u'\in V_a$ and $v'\in V_{b,c}$.
If $E(V_a,V_{b,c})$ induces a complete bipartite graph, then  no two vertices of $\{w_a,w_{bc},w\}$ exist at the same time.
\end{Claim}
\begin{proof}
Suppose $w_a$ and $w_{bc}$ exist at the same time.
There there are two possible cases: either $w_a$ and $w_{bc}$ are adjacent or not.
If $w_a$ and $w_{bc}$ are adjacent, then $G[\{b,a,u,w_a,w_{bc}\}]$ is isomorphic to $\Gaa$; which is impossible.
If $w_a$ and $w_{bc}$ are not adjacent, then $G[\{b,c,u,v,w_a,w_{bc}\}]$ is isomorphic to $\Gag$; which is impossible.
Then $w_a$ and $w_{bc}$ do not exist at the same time.

Suppose $w_a$ and $w$ exist at the same time.
There there are two possible cases: either $u=u'$ or $u\neq u'$.
Suppose $u=u'$, then we have two possible cases: either $w_a$ and $w$ are adjacent or not.
If $w_a$ and $w$ are adjacent, then $G[\{w_a,w,v',b,a\}]$ is isomorphic to $\Gaa$; which is impossible.
But if $w_a$ and $w$ are not adjacent, then $G[\{b,c,u,v',w_a,w\}]$ is isomorphic to $\Gal$; which is impossible.
Then $u\neq u'$.
Thus suppose $u\neq u'$.
Note that in the induced subgraph $G[\{a,b,c,u,v',w_a,w\}]$, the vertex $w$ is only adjacent with $v'$, and $w_a$ is only adjacent with $u$.
Then applying previous $(w_a,w_{bc})$ case in this induced subgraph, we get that $w$ and $w_a$ do not exist at the same time.

Suppose $w_{bc}$ and $w$ exist at the same time.
There are two possible cases: either $v=v'$ or $v\neq v'$.
Suppose $v=v'$, then we have two possible cases: either $w_{bc}$ and $w$ are adjacent or not.
If $w_a$ and $w_{bc}$ are adjacent, then $G[\{w_{bc},w,u,a,b\}]$ is isomorphic to $\Gaa$; which is impossible.
But if $w_a$ and $w_{bc}$ are not adjacent, then $G[\{b,c,u,v,w_{bc},w\}]$ is isomorphic to $\Gal$; which is impossible.
Then $v\neq v'$.
Suppose $v\neq v'$.
Note that in the induced subgraph $G[\{a,b,c,u,v,w_a,w\}]$, the vertex $w$ is only adjacent with $u$, and $w_{bc}$ is only adjacent with $v$.
Thus by applying first case in this induced subgraph, we get that $w$ and $w_{bc}$ do not exist at the same time.
\end{proof}

\begin{Claim}	\label{claim:VbcupVc0AndVanot0:Vbcnot0EV_aV_bcComplere21}
Let $w\in V_\emptyset$, $u\in V_a$ and $v\in V_{b,c}$.
If $E(V_a,V_{b,c})$ induces a complete bipartite graph and $w$ is adjacent with $u$ and $v$, then $E(w,V_a\cup V_{b,c})$ induces a complete graph.
\end{Claim}
\begin{proof}
First we see that if $v'\in V_{b,c}-v$, then $w$ is adjacent with $v'$.
Suppose $w$ and $v'$ are not adjacent.
Since the induced subgraph $G[\{a,b,u,v,v',w\}]$ is isomorphic to the forbidden graph $\Gap$, then we get a contradiction.
Thus $w$ is adjacent with each vertex in $V_{b,c}$.
Now we see that if $u'\in V_a-u$, then $w$ is adjacent with $u'$.
Suppose $w$ and $u'$ are not adjacent.
Since the induced subgraph $G[\{a,b,u,u,u',w\}]$ is isomorphic to $\Gap$, then we get a contradiction.
Therefore, $w$ is adjacent with each vertex in $V_{a}\cup V_{b,c}$.
\end{proof}

\begin{Claim}\label{claim:VbcupVc0AndVanot0:Vbcnot0EV_aV_bcComplere2}
If $|V_a|=2$ and $E(V_\emptyset,V_{a})\neq\emptyset$, then either each vertex in $V_\emptyset$ is adjacent only with one vertex in 
$V_a$ and $V_\emptyset$ is a clique, or each vertex in $V_\emptyset$ is adjacent with $u_1$ and $u_2$, and $V_\emptyset$ is trivial.
\end{Claim}
\begin{proof}
Let $w,w'\in V_\emptyset$ and $i,j\in\{1,2\}$ with $i\neq j$.
By proving that the following cases are not possible, it follows that the only possible cases are that either $E(\{u_1,u_2\},\{w,w'\})$ 
is equal to $\{wu_i,w'u_j\}$ and $ww'\in E(G)$, or $E(\{u_1,u_2\},\{w,w'\})$ is equal to $\{wu_i,wu_j,w'u_i,w'u_j\}$ and $ww'\in E(G)$.
Which implies the result.
\begin{enumerate}[(a)]
\item $ww',wu_j,w'u_j\notin E(G)$ and $wu_i,w'u_i\in E(G)$,
\item $ww',wu_j,w'u_i\notin E(G)$ and $wu_i,w'u_j\in E(G)$,
\item $wu_j,w'u_i\notin E(G)$ and $ww',wu_i,w'u_j\in E(G)$,
\item $ww',w'u_j\notin E(G)$ and $wu_i,wu_j,w'u_i\in E(G)$,
\item $w'u_j\notin E(G)$ and $ww',wu_i,wu_j,w'u_i\in E(G)$, and
\item $w'u_j,ww',wu_i,wu_j,w'u_i\in E(G)$.
\end{enumerate}
Since in cases (a), (b) and (d) the induced subgraph $G[\{a,b,u_1,u_2,w,w'\}]$ is isomorphic to $\Gac$, $\Gai$, $\Gak$, respectively, then these cases are impossible.
On the other hand, in case (c) the induced graph $G[\{w,w',u_j,a,b\}]$ is isomorphic to $\Gaa$: which is impossible.
Also in case (e) the induced subgraph $G[\{w',w,u_j,a,b\}]$ is isomorphic to $\Gaa$; which is not possible.
Finally in case (f) the induced subgraph $G[\{a,b,c,u_1,u_2,w,w'\}]$ is isomorphic to $\Gbp$; which is not possible.
\end{proof}

\begin{Claim}\label{claim:VbcupVc0AndVanot0:Vbcnot0EV_aV_bcComplere3}
If $E(V_a,V_{b,c})$ induces a complete bipartite graph, and each vertex in $V_\emptyset$ is adjacent only with one vertex in 
$V_{a}$, then there is no vertex in $V_\emptyset$ adjacent with a vertex in $V_{b,c}$.
\end{Claim}
\begin{proof}
Let $w\in V_\emptyset$, $v\in V_{b,c}$.
Suppose $w$ is adjacent with $u_1$ and $v$, but $w$ is not adjacent with $u_2$.
Since the induced subgraph $G[\{a,b,w,u_1,u_2,v\}]$ is isomorphic to $\Gar$, then we get a contradiction and the result follows.
\end{proof}

\begin{Claim}\label{claim:VbcupVc0AndVanot0:Vbcnot0EV_aV_bcComplere4}
Let $w\in V_\emptyset$ and $u\in V_a$.
If $E(V_a,V_{b,c})$ induces a complete bipartite graph, $E(w,V_{b,c})\neq \emptyset$ and $E(w,V_a)=\emptyset$, 
then each vertex in $V_\emptyset$ is adjacent with only one vertex in $V_{b,c}$, and $V_\emptyset$ is a clique of cardinality at most 2.
\end{Claim}
\begin{proof}
Let $v,v'\in V_{b,c}$.
Suppose $w$ is adjacent with $v$ and $v'$.
Since $G[\{a,b,c,u,v,v',w\}]$ is isomorphic to $\Gbk$, then we get a contradiction.
And we have that $w$ is adjacent only with one vertex in $V_{b,c}$

Now let $w,w'\in V_\emptyset$ and $v,v'\in V_{b,c}$ such that $wv\in E(G)$ and $w'v'\notin E(G)$.
Suppose $v\neq v'$.
There are two cases: either $w$ and $w'$ are adjacent or not.
If $w$ and $w'$ are adjacent, then the induced subgraph $G[\{b,c,v,u,w,w'\}]$ is isomorphic to $\Gag$, then we get a contradiction and $ww'\notin E(G)$.
Now if $w$ and $w'$ are not adjacent, then the induced subgraph $G[\{w,v,u,v',w'\}]$ is isomorphic to $\Gaa$ and we get a contradiction.
Thus $v=v'$.

Finally, let $w,w'\in V_\emptyset$ adjacent with $v\in V_{b,c}$.
Suppose $w$ and $w'$ are not adjacent.
Since the induced subgraph $G[\{a,b,v,u,w,w'\}]$ is isomorphic to $\Gad$, then we get a contradiction and $ww'\in E(G)$.
\end{proof}

In case (11), by Claims \ref{claim:VbcupVc0AndVanot0:Vbcnot0EV_aV_bcComplere1}, \ref{claim:VbcupVc0AndVanot0:Vbcnot0EV_aV_bcComplere21}, 
\ref{claim:VbcupVc0AndVanot0:Vbcnot0EV_aV_bcComplere2}, \ref{claim:VbcupVc0AndVanot0:Vbcnot0EV_aV_bcComplere3} and 
\ref{claim:VbcupVc0AndVanot0:Vbcnot0EV_aV_bcComplere4}, we have the following possible cases:
\begin{itemize}
\item each vertex in $V_\emptyset$ is adjacent with each vertex in $V_a\cup V_{b,c}$, and $V_\emptyset$ is a trivial graph,
\item each vertex in $V_\emptyset$ is adjacent only with a vertex in $V_a$, and $V_\emptyset$ is a clique of cardinality at most 2, or
\item each vertex in $V_\emptyset$ is adjacent only with a vertex in $V_{b,c}$, and $V_\emptyset$ is a clique of cardinality at most 2.
\end{itemize}
Each of these cases corresponds to a graph isomorphic to an induced subgraph of $\mathcal F_1^1$.

\begin{Remark}\label{claim:VbcupVc0AndVanot0:WVaVbc0thenEV0VaVbcnot0}
Let $w\in V_\emptyset$, $u\in V_a$ and $v\in V_{b,c}$.
If $uv\notin E(G)$, and $w$ adjacent with $u$ or $v$, then $w$ is adjacent with both vertices $u$ and $v$.
\end{Remark}

In case (12), by Claim \ref{claim:VbcupVc0AndVanot0:Vbcnot0EV_aV_bcComplere2} and Remark \ref{claim:VbcupVc0AndVanot0:WVaVbc0thenEV0VaVbcnot0}, 
we have that each vertex in $V_\emptyset$ is adjacent with each vertex in $V_a\cup V_{b,c}$, and 
\begin{itemize}
\item $V_\emptyset$ is a clique or a trivial graph when $|V_a|=1$, or
\item $V_\emptyset$ is a clique when $|V_a|=2$.
\end{itemize}
It is not difficult to see that each of these graphs are isomorphic to an induced subgraph of $\mathcal F_1^1$

\begin{Claim}\label{claim:VbcupVc0AndVanot0:VaVbcKminuseThenV0:1}
Let $w\in V_\emptyset$, $u\in V_a$ and $v\in V_{b,c}$.
Suppose $u$ is adjacent with each vertex in $V_{b,c}-v$.
If $E(w,V_{b,c}\cup \{u\})\neq\emptyset$, then one of the following statements holds:
\begin{itemize}
\item each vertex in $V_\emptyset$ is adjacent with $u$ and $v$,or 
\item $E(w,V_{b,c}\cup \{u\})$ induces a complete bipartite graph and $|V_{b,c}-v|=1$.
\end{itemize}
\end{Claim}
\begin{proof}
First case follows by Remark \ref{claim:VbcupVc0AndVanot0:WVaVbc0thenEV0VaVbcnot0}.
Now we check when $w$ is adjacent with a vertex in $V_{b,c}-v$.
Let $v'\in V_{b,c}$.
Then $u$ is adjacent with $v'$.
Suppose $w$ is adjacent only with $v'$.
Since the induced subgraph $G[\{a,b,u,v,v',w\}]$ is isomorphic to $\Gaj$, then we have a contradiction and $w$ is adjacent also with $u$ or $v$.
Thus by Remark \ref{claim:VbcupVc0AndVanot0:WVaVbc0thenEV0VaVbcnot0}, $w$ is adjacent with $u$, $v$ and $v'$.
Finally by applying Claim \ref{claim:VbcupVc0AndVanot0:Vbcnot0EV_aV_bcComplere21} to the induced subgraph $G[\{w,u\}\cup (V_{b,c}-v)]$, we get that $w$ also is adjacent with each vertex in $V_{b,c}-v$.
Thus $w$ is adjacent with each vertex in $\{u\}\cup V_{b,c}$.

Suppose the cardinality of $V_{b,c}-v$ is at least 2.
Take $v,v'\in V_{b,c}-v$.
Thus $w$ is adjacent with $v,v'$ and $v''$.
Since the induced subgraph $G[\{u,v,v',v'',a,w\}]$ is isomorphic to $\Gaf$, then we get a contradiction.
Thus the cardinality of $V_{b,c}-v$ is at most 1.
\end{proof}

\begin{Claim}\label{claim:VbcupVc0AndVanot0:VaVbcKminuseThenV0:2}
Let $w\in V_\emptyset$ and $v\in V_{b,c}$.
Suppose $u_1$ is adjacent with each vertex in $V_{b,c}$, and $u_2$ is adjacent with each vertex in $V_{b,c}-v$.
If $E(w,V_{b,c})\neq\emptyset$, then either $w$ is adjacent only with $u_2$ and $v$, or $|V_{b,c}|=1$ and $E(w,V_a\cup V_{b,c})$ induces a complete bipartite graph.
\end{Claim}
\begin{proof}
First case follows by Remark \ref{claim:VbcupVc0AndVanot0:WVaVbc0thenEV0VaVbcnot0}.
Now we prove that if $w$ is adjacent with $u_1$, then $w$ is adjacent with $u_2$ and $v$.
Suppose $w$ is adjacent only with $u_1$.
Since the induced subgraph $G[\{a,b,u_1,u_2,v,w\}]$ is isomorphic to $\Gal$, then we have a contradiction and $w$ is adjacent also with $u_2$ or $v$.
Thus by Remark \ref{claim:VbcupVc0AndVanot0:WVaVbc0thenEV0VaVbcnot0}, $w$ is adjacent with $u_1$, $u_2$ and $v$.
Finally, suppose $|V_{b,c}|\geq 2$.
Let $v'\in V_{b,c}$.
By applying Claim \ref{claim:VbcupVc0AndVanot0:Vbcnot0EV_aV_bcComplere21} to the induced subgraph $G[\{w,u_1\}\cup V_{b,c}]$, we get that $w$ also is adjacent with each vertex in $V_{b,c}$.
Thus $w$ is adjacent with each vertex in $V_{a}\cup V_{b,c}$.
But since $G[\{a,u_1,u_2,v,v',w\}]$ is isomorphic to $\Gas$, then we get a contradiction, and $|V_{b,c}|=1$.
\end{proof}

Let $u\in V_a$ and $v\in V_{b,c}$ such that $uv\notin E(G)$.
Therefore, in case (13), by Claims \ref{claim:VbcupVc0AndVanot0:Vbcnot0EV_aV_bcComplere2}, \ref{claim:VbcupVc0AndVanot0:VaVbcKminuseThenV0:1} and \ref{claim:VbcupVc0AndVanot0:VaVbcKminuseThenV0:2}, we have that one of the following cases holds:
\begin{itemize}
\item $V_\emptyset$ is a clique of cardinality at most 2, and each vertex in $V_\emptyset$ is adjacent with $u$ and $v$, 
\item $V_a=\{u\}$, $|V_{b,c}\setminus\{v\}|\leq 1$, $V_\emptyset$ is trivial and each vertex in $V_\emptyset$ is adjacent with each vertex in $\{u\}\cup V_{b,c}$, or
\item $|V_a|=2$, $V_\emptyset$ is trivial and $V_{b,c}=\{v\}$ and  each vertex in $V_\emptyset$ is adjacent with $u_1,u_2$ and $v$.
\end{itemize}
It can be checked that each of these graphs is isomorphic to an induced subgraph or a graph in $\mathcal F_1^1$.

\begin{Claim}\label{claim:VbcupVc0AndVanot0:VaVbcKminuseeThenV0:1}
Let $w\in V_\emptyset$ and $v\in V_{b,c}$ such that $v$ is not adjacent with $u_1$ and $u_2$.
Suppose $|V_{b,c}|\geq 2$ and $E(V_a,V_{b,c}-v)$ induces a complete bipartite graph.
If $E(w,V_{b,c})\neq\emptyset$, then $V_\emptyset=\{w\}$, and $w$ is adjacent with $u_1$, $u_2$ and $v$.
\end{Claim}
\begin{proof}
By Remark \ref{claim:VbcupVc0AndVanot0:WVaVbc0thenEV0VaVbcnot0}, we have that if $w$ is adjacent with one vertex of $\{v,u_1,u_2\}$, then $w$ is adjacent with $v,u_1$ and $u_2$.
Now we see that $w$ is not adjacent to any other vertex.
Suppose $w$ is adjacent with $v'\in V_{b,c}-v$.
First consider the case when $w$ is not adjacent with $v,u_1$ and $u_2$.
Since $G[\{b,v,v',u_1,u_2,w\}]$ is isomorphic to $\Gag$, then we get a contradiction and thus $w$ is adjacent with $v,v',u_1,$ and $u_2$.
But if $w$ is adjacent with $v,v',u_1,$ and $u_2$, then $G[\{b,v,v',u_1,u_2,w\}]$ is isomorphic to $\Gau$; which is impossible.
Then $w$ is not adjacent with any vertex in $V_{b,c}-v$.

Suppose there exist another vertex $w'\in V_\emptyset$.
By previous discussion $w'$  is adjacent with $v,u_1$ and $u_2$.
Since $G[\{a,b,w,w',v,v'\}]$ is isomorphic to $\Gab$, then we get a contradiction.
Thus $V_\emptyset$ has cardinality at most 1.
\end{proof}

Let $v\in V_{b,c}$ such that $u_1v,u_2v\notin E(G)$.
Therefore, in case (14), by Claims \ref{claim:VbcupVc0AndVanot0:Vbcnot0EV_aV_bcComplere2} and \ref{claim:VbcupVc0AndVanot0:VaVbcKminuseeThenV0:1}, we have that $V_{\emptyset}=\{w\}$ and $w$ is adjacent with $u_1,u_2$ and $v$.
Then $G$ is isomorphic to an induced subgraph of $\mathcal F_1^1$.

\begin{Claim}\label{claim:VbcupVc0AndVanot0:EVaVbcperfect mathcingThenV0:1}
Let $V_{b,c}=\{ v,v'\}$ such that $vu_1,v'u_2\in E(G)$ and $vu_2,v'u_1\notin E(G)$.
If $w\in V_\emptyset$ and $E(w,V_a\cup V_{b,c})\neq\emptyset$, then $w$ is adjacent with $u_1,u_2,v$ and $v'$.
\end{Claim}
\begin{proof}
Since $w$ is adjacent with a vertex in $\{u_1,u_2,v,v'\}$, then $w$ is adjacent with $u_1$ and $v'$, or $w$ is adjacent with $u_2$ and $v$.
Suppose $w$ is adjacent with $u_1$ and $v'$, but $w$ is not adjacent with $u_2$ and $v$.
Since $G[\{a,u_1,u_2,v,v',w\}]$ is isomorphic to $\Gal$, then we get a contradiction, and therefore $w$ is also adjacent with $u_2$ and $v$.
\end{proof}

Let $V_{b,c}=\{ v,v'\}$ such that $vu_1,v'u_2\in E(G)$ and $vu_2,v'u_1\notin E(G)$.
Thus in case (15), by Claims \ref{claim:VbcupVc0AndVanot0:Vbcnot0EV_aV_bcComplere2} and \ref{claim:VbcupVc0AndVanot0:EVaVbcperfect mathcingThenV0:1}, we have that $V_{\emptyset}$ is trivial and each vertex in $V_\emptyset$ is adjacent only with $u_1,u_2,v$ and $v'$.
Therefore, $G$ is isomorphic to an induced subgraph of $\mathcal F_1^1$.


\end{document}